\documentclass[12pt]{article}
\usepackage{amsmath}
\usepackage{amsthm}
\usepackage{amssymb}
\usepackage{delarray}
\usepackage[dvips]{graphics}
\usepackage{epsfig}
\usepackage{color}
\usepackage{url}

\newcommand{\C}{{\mathbb C}}
\newcommand{\Q}{{\mathbb Q}}
\newcommand{\Z}{{\mathbb Z}}
\newcommand{\F}{{\mathbb F}}
\newcommand{\N}{{\mathbb N}}
\newcommand{\R}{{\mathbb R}}

\newcommand{\zZ}{{\mathcal Z}}
\newcommand{\M}{{\mathcal M}}

\newcommand{\vol}{{\rm vol}\,}
\newcommand{\Gal}{{\rm Gal}\,}
\newcommand{\fp}{{\mathfrak p}}
\newcommand{\fq}{{\mathfrak q}}
\newcommand{\cO}{{\mathcal O}}
\newcommand{\cC}{{\mathcal C}}
\newcommand{\p}{{\mathfrak p}}

\newcommand{\cI}{J}
\newcommand{\m}{\mathfrak{m}}
\def\CL{{\mathcal C}}

\newcommand{\disc}{\rm{disc}\, }

\newcommand{\mD}{{\mathcal D}}
\newcommand{\mX}{{\mathcal X}}
\newcommand{\bP}{{\mathbb P}}
\newcommand{\mY}{{\mathcal Y}}
\newcommand{\mE}{{\mathcal E}}
\newcommand{\mZ}{{\mathcal Z}}
\newcommand{\GL}{{\mathrm{GL}}}
\newcommand{\SL}{{\mathrm{SL}}}

\newtheorem{theorem}{Theorem}

\newtheorem{conjecture}{Conjecture}
\newtheorem{lemma}{Lemma}
\newtheorem{proposition}{Proposition}
\newtheorem{cor}{Corollary}

\newtheorem{prob}{Problem}[section]

\theoremstyle{definition}
\newtheorem{defn}{Definition}

\theoremstyle{remark}

\newtheorem{rem}{Remark}

\makeatletter
\def\url@leostyle{%
  \@ifundefined{selectfont}{\def\UrlFont{\sf}}{\def\UrlFont{\small\ttfamily}}}
\makeatother
\title{Distribution of orders in number fields}
\author{Nathan Kaplan\\ Yale University\\ New Haven, CT\\ \texttt{nathan.kaplan@yale.edu}
\and Jake Marcinek \\ California Institute of Technology \\ Pasadena, CA \\ \texttt{jakemarcinek@gmail.com}
\and Ramin Takloo-Bighash\\ University of Illinois at Chicago\\ Chicago, IL\\ \texttt{rtakloo@math.uic.edu}}

\begin{document}

\maketitle

\begin{abstract} In this paper we study the distribution of orders of bounded discriminants in number fields.  We give an asymptotic formula for the number of orders contained in the ring of integers of a quintic number field.
\end{abstract}

\tableofcontents

\section{Introduction}

Let $K/\Q$ be an extension of degree $n$ with ring of integers $\cO_K$. An order $\cO$ is a subring of $\cO_K$ with identity that is a $\Z$-module of rank $n$. Set 
$$
N_K(B) := \left|\{ \cO \subseteq \cO_K; \cO \text{ an order }, |\disc \cO| \leq B\}\right|. 
$$
In this paper we study the asymptotic growth of $N_K(B)$ as $B$ grows.

\subsection{Results}

Our first theorem, which is a consequence of the motivic framework used here, is the following result: 
\begin{theorem}\label{basicmain}
There is $\alpha_K \in \Q_{>0}, \beta_K \in \N, C_K \in \R_{>0}$ such that 
$$N_K(B) \sim C_K B^{\alpha_K} (\log B)^{\beta_K-1} 
$$
as $B \to \infty$. 
\end{theorem}

Let $E/\Q$ be the normal closure of $K$ with Galois group $G= \Gal(E/\Q)$. Then $G$ has a natural embedding in $S_n$ as a transitive subgroup. Let $V_2$ be the vector space whose basis is the set of $2$-element subsets of $\{1, \cdots, n\}$. The group $G$ has a natural action on $V_2$. Let  $r_2$ be the dimension of the space of $G$ fixed vectors in $V_2$. Then we have the following theorem: 
\begin{theorem}\label{majorthm}
Let $K/\Q$  number field of degree $n$. 
\begin{enumerate} 
\item For $n \leq 5$, there is a constant $C_K >0 $ such that 
$$
N_K(B) \sim C_K B^{1/2}(\log B)^{r_2 -1}
$$
as $B \to \infty$; 
\item For any $n >5$, 
$$
B^{1/2} (\log B)^{r_2 -1} \ll N_K(B)  \ll_\epsilon B^{\frac{n}{4} - \frac{7}{12}+\epsilon}.
$$
\end{enumerate} 
\end{theorem} 

Table \ref{table:nonlin} lists the transitive subgroups of $S_n$ for small $n$ and the corresponding $r_2$ values. The reference for the list of subgroups up to conjugation is \cite{DM}, \S 2.9.  For the computation of $r_2$, see \S \ref{remarks:r2}. 

\begin{table}[h]
\caption{Transitive subgroups up to conjugation} 
\centering 
\begin{tabular}{c c c c c} 
\hline\hline 
$n$ & Order & Group Name & Generators & $r_2$ \\ [0.5ex] 
\hline 
$3$ & $3$ & $\Z/3 \Z$ & $(1 \,\, 2 \,\, 3)$ & 1\\ 
$3$ & $6$ & $S_3$ & $(1 \,\, 2), (1 \,\, 3)$ &  $1$ \\ 
$4$ & $4$ & $\Z/4\Z$ & $(1 \,\, 2 \,\, 3 \,\, 4)$ &  2 \\ 
$4$ & $4$ & $\Z/2 \Z \times \Z/2 \Z$ & $(1 \,\,2)(3 \,\, 4), (1\,\,4)(2 \,\, 3)$ & 3\\
$4$ & $8$ & $D_4$ & $(1 \,\, 2\,\, 3 \,\, 4), (1 \,\, 3)$ & $2$\\
$4$ & $12$ & $A_4$ & $(1 \,\, 2\,\, 4), (2\,\, 3 \,\, 4)$ & $1$\\
$4$ & $24$ & $S_4$ & $(1 \,\,2), (1 \,\, 3), (1 \,\, 4)$ & $1$\\ 
$5$ & $5$  & $\Z / 5 \Z$ &  $(1 \,\, 2\,\, 3 \,\, 4 \,\, 5)$ & $2$\\ 
$5$ & $10$ & $D_5$ & $(1 \,\, 2\,\, 3 \,\, 4 \,\, 5), (1\,\, 4)(2 \,\, 3)$ & $2$ \\ 
$5$ & $20$ & AGL(1, 5)  & $(1 \,\, 2\,\, 3 \,\, 4 \,\, 5), (2 \,\, 3\,\, 5 \,\, 4 )$ & $2$ \\ 
$5$ & $60$ & $A_5$ & $(1 \,\, 2\,\, 4), (3\,\, 4 \,\, 5), (2 \,\, 3\,\, 5)$ & 1\\ 
$5$ & $120$ & $S_5$ & $(1 \,\,2), (1 \,\, 3), (1 \,\, 4), (1 \,\, 5)$ & 1 \\ 
[1ex] 
\hline 
\end{tabular}
\label{table:nonlin} 
\end{table}

In order to study the behavior of $N_K(B)$ we form the counting zeta function 
$$
\eta_K(s) = \sum_{\cO \text{ order}}\frac{1}{|\disc \cO|^s},
$$
where $\cO_K$ is the ring of integers of $K$ and $\cO$ is an order.  This series converges absolutely for $\Re s$ large, and in its domain of absolute convergence we have 
$$
\eta_K(s) = |\disc \cO_K|^{-s} \tilde\eta_K(2s) 
$$
where 
$$
\tilde\eta_K(s) = \sum_{\cO \text { order}}\frac{1}{[\cO_K: \cO]^s}. 
$$

The zeta function $\tilde\eta_K$ has an Euler product of the form 
$$
\tilde\eta_K(s) = \prod_p \tilde\eta_{K, p}(s) 
$$
where 
$$
\tilde\eta_{K, p}(s) = \sum_{\cO}\frac{1}{[\cO_K\otimes_\Z \Z_p: \cO]^s}
$$
and the summation in the last expression is over full rank sublattices of $\cO_K \otimes_\Z \Z_p$ that are subrings with identity. We define the coefficients $a_i(p)$ by 
$$
\tilde\eta_{K, p}(s) = 1+\sum_{i=1}^\infty \frac{a_i(p)}{p^{is}}. 
$$
The number $a_i(p)$ is what in \S\ref{method} is denoted by $a_{\cO_K}^{1, <}(p^i)$. 

\

The proof of Theorem \ref{majorthm}  has two main steps. The first step which is arithmetic is the following theorem: 

\begin{theorem}[Arithmetic Step]\label{thm:main}
The Euler product 
$$
f(s) = \prod_{p \text{ unramified} } (1 + a_1(p) p^{-s})
$$
converges absolutely for $\Re s$ large, and it has an analytic continuation to a meromorphic function on an open set containing $\Re s \geq 1$ with a unique pole at $s=1$ of order $r_2$. 
\end{theorem}

\begin{rem}\label{conj:1}
It is reasonable to conjecture that for $n$ small the function $\tilde\eta_K(s)$ is holomorphic for $\Re s >1$, and has an analytic continuation to a domain containing $\Re s \geq 1$ with a unique pole of order $r_2$ at $s=1$. If this is true, then there is a nonzero constant $C_K$ such that 
$$
N_K(B) = C_K B^{1/2} (\log B)^{r_2 - 1} ( 1 + o(1)) 
$$
as $B \to \infty$. 
The  conjecture is true for $n \leq 5$ by Theorem \ref{majorthm}. The results of Brakenhoff \cite{Br}, summarized in \S\ref{comp} below, show that for $n \geq 8$ there is a pole to the right of $\Re s =1$.
\end{rem}

\
 
The second step of the proof of the main theorem is geometric. Since by Lemma 4.15 of \cite{duSG} the finitely many bad primes do not contribute to the main pole, part 1 of Theorem \ref{majorthm} is a consequence of the following theorem: 

\begin{theorem}[Geometric Step for small $n$]\label{conj:3}
Let $n \leq 5$. There is a finite set $S$ of primes such that the series 
$$
\sum_{p \not\in S} \sum_{i=2}^\infty \frac{a_{i}(p)}{p^{i\sigma}}
$$
converges for any real $\sigma > 19/20$. 
\end{theorem}
We give heuristic reasoning for why this result should hold in the case $n=5$. Let $b_i(p)$ be the number of subrings with identity of $\Z_p^5$, i.e. orders, whose index is $p^i$. It is reasonable to expect that 
\begin{equation}\label{wishful}
a_i(p) \leq b_i(p)
\end{equation}
 for all $i$ and $p$. 
It is then a consequence of Theorem \ref{error-5} that the series 
$$
\sum_{p \text{ odd prime}} \sum_{i=2}^\infty \frac{b_{i}(p)}{p^{i\sigma}}
$$
converges for $\sigma > 19/20$.  Alas, we have not been able to prove \eqref{wishful}--even though we are confident it is true. Here we employ an alternative method based on $p$-adic integration. 

\

Part 2 of Theorem \ref{majorthm} is a consequence of the following theorem and Lemma 4.15 of \cite{duSG}: 

\begin{theorem}[Geometric Step for large $n$]\label{n>5bound} Let $n >5$. There is a finite set $S$ of primes such that the series 
$$
\sum_{p \not\in S} \sum_{i=2}^\infty \frac{a_{i}(p)}{p^{i\sigma}}
$$
converges for any real $\sigma > \frac{n}{2} - \frac{7}{6}$. 
\end{theorem}

\begin{rem}
Note that by Theorem 1.5 of \cite{duSG}, the zeta function $\eta_K(s)$ has an analytic continuation to a domain of the form $\Re s > \alpha - \epsilon$ with $\alpha >0$ the abscissa of convergence and $\epsilon >0$. 
\end{rem}

\begin{rem}
A byproduct of our methods, stated as Corollary \ref{81} and Corollary \ref{improvement} in \S\ref{general}, is an improvement of the upper bounds obtained by Brakenhoff \cite{Br}, Theorem 5.1 and Theorem 8.1. 
\end{rem}

\begin{rem}
It would be interesting to obtain information about the constant $C_K$. For the cubic case, the results of \cite{DW} stated below give precise values for $C_K$.  Corollary 1 of Nakagawa \cite{N} gives the value of $C_K$ in terms of certain Euler products, but it is not clear if these Euler products have any conceptual meaning. For higher degree extensions we know nothing about the constants $C_K$. 
\end{rem}

More generally if $L = \prod_i K_i$ is an \'{e}tale $\Q$-algebra with $K_i$'s number fields, we define $\cO_L = \prod_i \cO_{K_i}$. Clearly $\cO_L$ is $\Z$-algebra which is free as a $\Z$-module of rank $d=\sum_i [K_i: \Q]$. We define an order $\cO$ in $\cO_L$ to be a subring with identity of $\cO_L$ which is of $\Z$-rank $d$. Again we set 
$$
\widetilde\eta_L(s) = \sum_{\cO\, \text{order}}\frac{1}{[\cO_L: \cO]^s}. 
$$
As usual knowing the analytic properties of $\widetilde\eta_L(s)$ via Tauberian arguments, e.g. Theorem \ref{tauberian application}, gives us information about the function
$$
\widetilde{N}_L(B) := \left|\{ \cO \subset \cO_L; \cO \text{ an order}, [\cO_L: \cO] \leq B\}\right|. 
$$
Our methods give an asymptotic formula for $\widetilde{N}_L(B)$ whenever $[L: \Q] \leq 5$. 

\

Let us explain the simplest possible case.  For $d \in \N$, we set $N_d(B):=\widetilde{N}_{\Q^d}(B)$. Given $k \in \N$, we define $f_d(k)$ to be the number of orders in $\Z^d$ of index equal to $k$. Clearly, $N_d(B) = \sum_{k \leq B} f_d(k)$.  It is easy to see that the function $f_d(k)$ is multiplicative, i.e. if $k_1, k_2$ are coprime integers then $f_d(k_1k_2) = f_d(k_1)f_d(k_2)$. 

\

This is the prototype of the problem that we study in this paper: 

\begin{prob}
Let $d\in \N$. Study the function $N_d(B)$ as $B\to \infty$. 
\end{prob}

Despite its innocent appearance this is a very difficult problem, and prior to our work the only cases for which an asymptotic formula is known for $N_d(B)$ are $d=2, 3, 4$ \cite{Liu}.  Here we obtain an asymptotic formula for $N_5(B)$, and give non-trivial bounds for $N_d(B)$ when $d >5$. 

\begin{defn}
Let $d, k \in \N$. We define $a^<_{\Z^d}(k)$ to be number of subrings $S$ of $\Z^d$, not necessarily with identity, such that $[\Z^d: S]=k$. 
\end{defn}
A subring $S$ in $\Z^d$ which is of finite index as an additive group will necessarily be a free $\Z$-module of rank $d$. Such subrings are called {\em multiplicative sublattices} in \cite{Liu}. An elementary proposition in \cite{Liu} states that for any $d, k \in \N$, $d \geq 2$, we have 
$$
f_{d}(k) = a^<_{\Z^{d-1}}(k). 
$$
As a result, with the notation of \S \ref{method}
$$
\tilde\eta_{\Q^d}(s) = \zeta_{\Z^{d-1}}^<(s). 
$$

Determining the asymptotic behavior of $N_1(B)$ and $N_2(B)$ is trivial. In this paper we will use the method of $p$-adic integration as in \S\ref{method} to  prove the following theorem:
\begin{theorem}\label{mainthm:1}
\begin{enumerate}
\item Let $d \leq 5$.  There is a positive real number $C_d$ such that
  $$
  N_d(B) \sim C_d B (\log B)^{{d \choose 2}-1}
  $$
  as $B\to \infty$.
\item Suppose $d \geq 6$. Then for any $\epsilon >0$ we have
$$
B (\log B)^{{d \choose 2}-1} \ll N_d(B) \ll_\epsilon B^{\frac{d}{2} - \frac{7}{6}+ \epsilon}
$$
as $B \to \infty$.
\end{enumerate}
\end{theorem}

We actually prove a more precise statement and give error estimates; See
Theorems \ref{error-3}, \ref{error-4} and \ref{error-5}.  We include the $d=3$ case to illustrate our method in a simple case. Our results for $d \geq 5$ are new. 

\

Theorem \ref{mainthm:1} is more than just a prototype of the type of result we can prove. The computations in \S \ref{volume estimates for n=5} form the backbone of the proof of Theorem \ref{majorthm}. In fact, Theorem \ref{thm:alphanarrow} shows that, essentially, whatever estimate we obtain for the volumes of the sets considered in \S\ref{volume estimates for n=5} works in general. 

\

We expect the asymptotic formula in Part 1 of Theorem \ref{mainthm:1} to be valid for $d < 8$.  The formalism of $p$-adic integration shows that $N_d(B)$ has an asymptotic formula of the form $C B^a (\log B)^{b-1}$, for a rational number $a$ and a natural number $b$, but for $d \geq 8$ it is not clear what the numbers $a, b$ should be.

We finish this introduction with the following conjecture: 

\begin{conjecture}
Let $K/\Q$ be a number field of degree $d$. Then with the notation of Theorem \ref{basicmain}, we have  
$$
\alpha_K = \frac{1}{2} \lim_{B \to \infty} \frac{\log N_d(B)}{\log B}. 
$$
In particular, $\alpha_K$ only depends on the extension degree of $K$ over $\Q$. 
\end{conjecture} 

\subsection{Comparison with previous results}\label{comp} 

If we write 
$$
\zeta_{\Z^n}(s) : = \sum_{\Lambda \subset \Z^n} \frac{1}{[\Z^n: \Lambda]^s},  
$$
where $\Lambda$ is a sublattice of $\Z^n$, it can be seen that for $\Re(s) > n$ we have 
$$
\zeta_{\Z^n}(s) = \zeta(s)\zeta(s-1)\cdots \zeta(s-n+1). 
$$
As a result $\zeta_{\Z^d}(s)$ has a pole of order $1$ at $s=n$ with residue \\
$\zeta(n)\zeta(n-1)\cdots \zeta(2)$. Consequently, 
$$
|\{ \Lambda \leq \Z^n \, | \, \Lambda \text{ sublattice}, [\Z^n: \Lambda] \leq B\}| \sim \frac{\zeta(n)\zeta(n-1)\cdots \zeta(2)}{d} B^n
$$
as $B \to \infty$. The book \cite{LS} contains five distinct proofs of this fact. 

\

Since in this work we are counting sublattices with additional structure we expect slower asymptotic growth. Theorem \ref{majorthm} is trivial for a quadratic field as the counting zeta function is simply the Riemann zeta function $\zeta(s)$. For $K$ a cubic or quartic extension of $\Q$, Theorem \ref{majorthm} is due to Datsovsky--Wright \cite{DW}  for the cubic case, and Nakagawa \cite{N} for the quartic case. 

\

 In the cubic setting, there is a bijection between the set of equivalence classes of integral binary cubic forms and the set of orders of cubic fields. Then it follows  from Shintani's theory of zeta functions associated to the prehomogeneous vector space of binary cubic forms combined with a theorem of \cite{DW} that 
$$
\tilde\eta_K(s) = \frac{\zeta_K(s)}{\zeta_K(2s)} \zeta(2s)\zeta(3s -1). 
$$
In the quartic setting, Nakagawa explicitly computes the local factors of the zeta function $\tilde\eta_K$ using an intricate combinatorial argument involving counting the number of solutions of some very complicated congruences. Due to computational difficulties at the prime $2$, Nakagawa's theorem assumes some mild ramification conditions. Under these conditions, he shows that the zeta function $\tilde\eta_K(s)$ has an analytic continuation to $\Re s > 2/3$. Nakagawa's explicit local computations can also be used to prove Theorem \ref{mainthm:1} for $d=4$. The paper \cite{Liu} contains a different approach to Theorem \ref{mainthm:1} using combinatorial arguments. Here, too, the local Euler factors of the counting zeta function are explicitly computed, though the proof follows from elegant recursive formulas, c.f. Propositions 6.2 and 6.3 of \cite{Liu}. 

\

In a series of spectacular papers, Bhargava studies orders in quintic fields. In \cite{B1}, he shows that there is a canonical bijection between the set of orbits of $\GL_4(\Z) \times \SL_5(\Z)$ on the space $\Z^4 \otimes \wedge^2 \Z^5$ and the set of isomorphism classes of pairs $(R, S)$ with $R$ a quintic ring and $S$ a sextic resolvent ring of $R$.   An impressive theorem of Bhargava \cite{B} which is proved using the above bijection says that  
$$
\sum_{K \text{ quintic}} N_K(B) \sim c B 
$$
as $B \to \infty$. Bhargava's methods do not identify the contribution of each $N_K(B)$ to the sum. 

\

The thesis \cite{Br} contains an array of interesting results on the distribution of orders in number fields. In keeping with our notation below, for a number field $K$ we let 
$$
a_{\cO_K}^{1, <}(m) = \left| \{ \cO \subset \cO_K; \cO \text{ an order}, [\cO_K: \cO]=m\}\right|. 
$$
We then let 
$$
a^{1, <}(n, m) = \max_{K/\Q \text{ extension of degree }n} a_{\cO_K}^{1, <}(m). 
$$
Theorem 5.1 of \cite{Br} is the statement that 
$$
c_7(n) \leq \limsup_{m\to \infty} \frac{\log a^{1, <}(n, m)}{\log m} \leq c_8(n)
$$
with $c_7(n) =\max_{0\leq d \leq n-1} \frac{d(n-1-d)}{n-1+d}$ and $c_8(n)$ given by the following table: 
\begin{table}[h]
\caption{Values of $c_8(n)$} 
\centering 
\begin{tabular}{c | c c c c c c c c c c c c c} $n$ & $2$  & $3$ & $4$ & $5$ & $6$ & $7$ & $8$ & $9$ & $10$ & $11$ & $12$ & $13$ & $\geq 14$  \\ [0.5ex] 
\hline 
$c_8(n)$ & $0$ & $\frac{1}{3}$ & $1$ & $\frac{20}{11}$ & $\frac{29}{11}$ & $\frac{186}{53}$ & $\frac{49}{11}$ & $\frac{119}{22}$ & $\frac{70}{11}$ & $\frac{388}{53}$ & $\frac{440}{53}$ & $\frac{492}{53}$ & $n - \frac{8}{3}$  \\ 
[1ex] 
\end{tabular}
\label{table:c8}
\end{table}

Furthermore, 
$$
\liminf_{n\to \infty} \frac{1}{n} \limsup_{m \to \infty} \frac{\log a^{1, <}(n, m)}{\log m} \geq 3 - 2 \sqrt{2} 
$$
and 
$$
\limsup_{n \to \infty} \frac{1}{n} \limsup_{m \to \infty} \frac{\log a^{1, <}(n, m)}{\log m} \leq 1. 
$$
One can compute the values of $c_7(n)$ explicitly as follows: 
$$
c_7(n) = \begin{cases}
\frac{k(2k -1)}{4k-1} & n = 3k; \\ 
\frac{k}{2} & n = 3k+1 ; \\
\frac{k(k+1)}{2k+1} & n = 3k +2. 
\end{cases}
$$
In particular, for $n \geq 8$, $c_7(n) > 1$. 

\

Theorem 8.1 of \cite{Br}, which is used in \cite{B}, is the following result: If $K/\Q$ is a quintic field, then for any prime $p$
$$
\sum_{k=1}^\infty \frac{a_{\cO_K}^{1, <}(p^k)}{p^{2k}} = O(1/ p^2).  
$$

We improve the upper bounds in these theorems in \S \ref{general}, Corollary \ref{81} and Corollary \ref{improvement}.

\subsection{Our method}\label{method}
Given a ring $R$ whose additive group is isomorphic to $\Z^d$ for some $d \in \N$ we define 
$$
a^<_R (k) := \left| \left\{ S \text{ subring of } R \,\, | \,\, [R:S]=k\right\}\right|. 
$$
 For any $k \in \N$, $a_R^<(k)$ is finite. 
We define the {\em subring zeta function of $R$} by 
$$
\zeta_R^<(s) : = \sum_{k=1}^\infty \frac{a_R^<(k)}{k^s} = \sum_{S \leq R} \frac{1}{[R:S]^s}. 
$$

We view $\zeta_R^<(s)$ not just as a formal series, but as a series converging on some non-trivial subset of the complex numbers. The idea is that the analytic properties of the resulting complex function have consequences for the distribution of subrings of finite index in $R$. In particular, by various Tauberian theorems e.g. Theorem \ref{tauberian application}, the location of poles and their orders gives information about the function $s_R^<(B)$ defined by 
$$
s_R^<(B) : = \sum_{k \leq B} a_R^<(k) = \left| \left\{ S \text{ subring of } R \,\, | \,\, [R:S]\leq B\right\}\right|.
$$
Similar constructions can be made for subgroups of finitely generated groups and ideals in rings, but in this introduction we only consider subring zeta functions. 
We have the following theorem which is a summary of results from \cite{GSS, duSG}
\begin{theorem} \label{subringzeta}
\begin{enumerate}
\item The series $\zeta_R^<(s)$ converges in some right half plane of $\C$. The abscissa of convergence $\alpha_R^<$ of $\zeta_R^<(s)$ is a rational number. There is a $\delta>0$ such that $\zeta_R^<(s)$ can be meromorphically continued to the domain  $\{s \in \C \, | \, \Re(s) > \alpha_R^< - \delta\}$. Furthermore, the line $\Re(s) = \alpha_R^<$ contains at most one pole of $\zeta_R^<(s)$ at the point $s = \alpha_R^<$. 
\item There is an Euler product decomposition 
$$
\zeta_R^<(s) = \prod_p \zeta_{R, p}^<(s)
$$
with the local Euler factor given by
$$
\zeta_{R, p}^<(s) = \sum_{l=0}^\infty \frac{a_R^<(p^l)}{p^{ls}}. 
$$
This local factor is a rational function of $p^{-s}$; there are polynomials $P_p, Q_p \in \Z[x]$ such that $\zeta_R^<(s) = P_p(p^{-s})/Q_p(p^{-s})$. The polynomials $P_p, Q_p$ can be chosen to have bounded degree as $p$ varies. The local Euler factors satisfy functional equations. 
\end{enumerate} 
\end{theorem} 
The functional equation mentioned in the theorem is proved in \cite{Voll}; also see Chapter 4 of \cite{duSW}. A corollary of this theorem is that the asymptotic behavior of the function $s_R^<(B)$ is of the form $c_R^< B^{\alpha_R^<} (\log B)^{b_R^<-1}$ as $B \to \infty$. Here $b_R^<$ is the order of pole of $\zeta_R^<(s)$ at $s=\alpha_R^<$. It is known that $b_R^< \geq 1$. It is a fundamental problem in the subject to relate the numbers $\alpha_R^<, b_R^<, c_R^< \in \R$ to structure of $R$. 

\

The paper \cite{GSS} introduced a $p$-adic formalism to study the local Euler factors $\zeta_R^<(s)$. Fix a $\Z$-basis for $R$ and identify $R$ with $\Z^d$. The multiplication in $R$ is given by a bi-additive map
$$
\beta: \Z^d \times \Z^d \to \Z^d
$$
which extends to a bi-additive map 
$$
\beta_p: \Z_p^d \times \Z_p^d \to \Z_p^d
$$
giving $R_p = R\otimes_\Z \Z_p$ the structure of a $\Z_p$-algebra. 
\begin{defn}\label{Mp}
Let $\M_p(\beta)$ be the subset of the set of $d \times d$ lower triangular matrices $M$ with entries in $\Z_p$ such that if the rows of $M = (x_{ij})_{1 \leq i, j \leq d}$ are denoted by $v_1, \dots, v_d$, then for for all $i, j$ satisfying $1 \leq i,j \leq d$, there are $p$-adic integers $c_{ij}^1, \dots, c_{ij}^d$ such that 
$$
\beta(v_i, v_j) = \sum_{k=1}^d c_{ij}^k v_k. 
$$
\end{defn}
Let $dM$ be the normalized additive Haar measure on ${\mathrm T}_d(\Z_p)$, the set of $n \times n$ lower triangular matrices with entries in $\Z_p$. Proposition 3.1 of \cite{GSS} says:
\begin{equation}\label{coneintegral}
\zeta_{R, p}^<(s) = (1-p^{-1})^{-d} \int_{\M_p(\beta)} |x_{11}|^{s-d} |x_{22}|^{s-d+1} \cdots |x_{dd}|^{s-1}\, dM. 
\end{equation}
Most of the statements of Theorem \ref{subringzeta} are proved using this $p$-adic formulation. The integral appearing in \eqref{coneintegral} is an example of a {\em cone integral}. The beauty of the equation \eqref{coneintegral} is that it allows us to express the number of subrings of a given index in terms of volumes of certain $p$-adic domains.

\

Let $D=(f_0, g_0, f_1, g_1, \cdots, f_l, g_l)$ be polynomials in the variables $x_1, \dots, x_m$ with rational coefficients. We call $D$ the {\em cone integral data}. For a prime number $p$ we define 
$$
\M_p(D) : = \{ x \in \Z_p^m \, | \, v_p(f_i(x)) \leq v_p(g_i(x)),\, \text{ for all } 1 \leq i \leq l\}, 
$$
and we define the {\em cone integral} associated to the cone integral data $D$ by 
$$
Z_D(s, p) = \int_{\M_p(D)} |f_0(x)|_p^s |g_0(x)|_p \, dx
$$
with $dx$ is the normalized additive Haar measure. The study of such integrals in special cases was started by Igusa \cite{Igusa, Igusa2}. Igusa's original method was based on the resolution of singularities. Igusa's approach was generalized by Denef \cite{Denef} and du Sautoy and Grunewald \cite{duSG}. Denef \cite{Denef} also introduced the use of elimination of quantifiers in $\Q_p$ as an alternative approach. For surveys on cone integrals and their applications to zeta functions of groups and rings, as well as references and examples, see \cite{duSG2, duSW, Voll2}. In general, calculating cone integrals is difficult and requires explicit desingularizations of highly singular varieties. For a ``simple" example, see \cite{duST}. 

\

There is a modification of this formalism to treat subrings with identity. Again, let $R$ be a ring with identity whose additive group is isomorphic to $\Z^d$ and for simplicity assume that the identity of $R$ is sent to $(1, 1, \dots, 1)$ under this isomorphism. For $k \in \N$, let 
$$
a_R^{1, <}(k) : = | \{ S \text{ subring with identity of }R \, | \, [R:S] = k\}|. 
$$
Now define the unitary subring zeta function of $R$ by 
$$
\zeta_R^{1, <}(s) : = \sum_{k=1}^\infty \frac{a_R^{1, <}(k)}{k^s}. 
$$
As before we have an Euler product expansion 
$$
\zeta_R^{1, <}(s)= \prod_p \zeta_{R,p}^{1, <}(s). 
$$
We let 
$$
s_R^{1, <}(B) : = \sum_{k \leq B} a_R^{1, <}(k) = \left| \left\{ S \text{ unitary subring of } R \,\, | \,\, [R:S]\leq B\right\}\right|.
$$
Again suppose after identifying $R$ with $\Z^d$, the multiplication on $R$ is given by a bi-additive map 
$$
\beta: \Z^d \times \Z^d \to \Z^d
$$
which extends to a bi-linear map 
$$
\beta_p: \Z_p^d \times \Z_p^d \to \Z_p^d. 
$$

\begin{defn}
Let $\M^1_p(\beta)$ be the subset of $\M_p(\beta)$ whose rows generate a unitary subring. 
\end{defn}

Then it is not hard to see that 
\begin{equation}\label{coneintegral1}
\zeta_{R, p}^{1,<}(s) = (1-p^{-1})^{-d} \int_{\M^1_p(\beta)} |x_{11}|^{s-d} |x_{22}|^{s-d+1} \cdots |x_{dd}|^{s-1}\, dM. 
\end{equation}

This integral too is a cone integral as we will see in \S \ref{general}.  As a result, the asymptotic behavior of $s_R^<(B)$ is of the form $c_R^< B^{\alpha_R^<} (\log B)^{b_R^<-1}$ as $B \to \infty$. Again, we use the expression \eqref{coneintegral1} to write the number of unitary subrings of a given index in terms of volumes of certain $p$-adic sets. 

\

In our problems of interest, the ring $R$ is a product of rings of integers of number fields. The two usual methods to study the cone integrals coming from subring zeta functions are resolution of singularities and elimination of quantifiers. Neither of these methods, however, can be applied in any obvious fashion to the problem of counting subrings of such $R$. This is due to the fact that our cone integrals are too complicated (see Sections
\ref{volume estimates for n=4} and \ref{volume estimates for
n=5}). In general there is no effective algorithm to eliminate
quantifiers for a complicated $p$-adic domain, and resolution of
singularities, while in principle computationally tractable, is
dreadful for domains of the type considered here. For example, the domain
needed to study $\Z^d$ would involve about $d^3$ inequalities of
the form $v_p(f(\underline{x})) \leq v_p(g(\underline{x}))$ with
$\underline{x}$ a vector of variables of length about $d^2$, and
$f$, $g$ ranging over polynomials with integer coefficients of
degrees $2$ to $d$. 

In this paper, inspired by \cite{S-T-T}, we propose a different approach. So far as the determination of the fundamental quantities $\alpha_R^<, b_R^<$ is concerned, we do not need explicit computations of the local integrals. Instead, in favorable circumstances such as those under consideration here, we can accomplish this by computing the first two terms of the Euler factors and estimating the rest of the terms. It is precisely for this reason that our method can be applied to more cases that what was treated in the earlier papers \cite{DW, Liu, N}. Here the difficulty lies in estimating volumes of certain $p$-adic sets that arise in the split situation of $\Z^d$, see \S \ref{volume estimates for n=4}, \ref{volume estimates for
n=5}, and \ref{general n}. Once this has been accomplished, we will use the results of \S \ref{alphanarrow} to show that the volume estimates obtained for the $\Z^n$ setting automatically extend to an arbitrary $R$ of the sort considered here.

\subsubsection*{Organization of the paper} 
The rest of the paper is organized as follows.  In \S\ref{geometry-padic} we recall results by Denef \cite{Denef-degree}, and use them to prove Theorem \ref{thm:alphanarrow}. We prove Theorem \ref{thm:main} in \S\ref{proof:thmmain}, using the outline explained in \S\ref{outline-thm2}. Section
\ref{section tauberian} contains the statements of the Tauberian
theorems we use in this work. We discuss the values of $r_2$ in \S \ref{remarks:r2}. The proof of Theorem \ref{mainthm:1} is presented in \S\ref{proof}.  The outline of the proof is sketched in \S\ref{outline} and details are postponed to later sections. In Section \ref{general facts} we collect several lemmas used in estimating
volumes. Section \ref{n=3} contains the treatment of the simple
case of $\Z^3$. We include this simple case to illustrate the
method. In Sections \ref{volume estimates for n=4} and \ref{volume
estimates for n=5} we give bounds for the volumes of our domains
for $n=4$ and $n=5$, respectively. These bounds are then used in
Sections \ref{convergence for n=4} and \ref{convergence for n=5}
to establish  Theorems \ref{error-3},
\ref{error-4}, and \ref{error-5} which imply the first part of Theorem \ref{mainthm:1}. The proof of the second part of Theorem \ref{mainthm:1} is presented in Section \ref{general n}.  The paper ends with the proof of Theorem \ref{majorthm} in \S\ref{general}. 

\subsubsection*{Notation} In this paper a ring $R$ is an additive group with a bi-additive multiplication such that the underlying additive group is finitely generated. We write $S \leq R$ if $S$ is a subring of $R$. The number $[R:S]$ is defined to be the index of $S$ in $R$ as an additive subgroup.   Throughout this paper $p$ is a prime number. When $p$ is used as the index of a sum or product, we will always understand that it ranges through the primes. The symbols $\Q_p$ and $\Z_p$ are the field of $p$-adic numbers and its ring of integers, respectively. We let $U_p$ denote the group of units of $\Z_p$. We normalize the additive Haar measure on $\Q_p$ such that $\vol(\Z_p)=1$, and the volume of a subset of $\Q_p$ is always with respect to this measure. For example, if $P(x)$ is a statement about a $p$-adic number $x$, the volume of $x \in \Q_p$ such that  $P(x)$ means the normalized Haar measure of the set $\{x \in \Q_p; P(x)\}$. The measure on $\Q_p^r$ for any $r >0$ is normalized similarly. 
The function $v_p: \Q_p \to \Z \cup \{ \infty\}$ is the $p$-adic valuation. If $f: S \to \C$ and $g: S \to \R_+$ are functions defined on a set $S$ to the set of positive real numbers $\R_+$ and $\C$, respectively, the notation $f(x) = O(g(x))$ means there is a constant $C>0$ such that for all $x \in S$ we have $|f(x)| \leq C g(x)$; this is also sometimes denoted by $f(x) \ll g(x)$.  If $S, T$ are sets, and $f: S \to \C$ and $g: S \times T \to \R_+$ are functions, the notation $f(x) = O_y (g(x, y))$ means that for every $y \in T$, there is a constant $C(y) >0$ such that for every $x \in S$ we have $|f(x)| \leq C(y) g(x, y)$. 

If $f(x), g(x): \R_+ \to \R_+$, we say that $f(x) \sim g(x)$ as $x \to +\infty$ if 
$\lim_{x \to +\infty} f(x)/g(x) =1$. For a complex number $s$, $\Re(s)$, usually denoted by $\sigma$,  is the real part of $s$. We will, without explicit mention, repeatedly use the fact that $\sum_{p \, \text{prime}} p^{a-b s}$, with $a, b$ real numbers, converges for $\Re (s) > (a+1)/b$. The collection of $n\times n$ matrices with entries in a ring $R$ is denoted by ${\mathrm M}_n(R)$. The set of lower triangular matrices in ${\mathrm M}_n(R)$ is written ${\mathrm T}_n(R)$. A finite extension $K/\Q$ is called a number field, and its absolute discriminant is denoted by $\disc_K$. The ring of integers of $K$ is written $\cO_K$. A subring with identity of $\cO_K$ which is a $\Z$-module of rank equal to the $\Z$-rank of $\cO_K$ is called an order. We write $\zeta(s)$ for the Riemann zeta function. If $\psi$ is a property of integers, and $f$ an arithmetic function, $\sum_{ p \,\, \psi}f(p) $ means the sum of the values of $f$ over all prime numbers $p$ which satisfy $\psi$; for example if $S$ is a set of integers, $\sum_{p \not\in S} f(p)$ means the sum is over all those prime numbers which are not in $S$.

\subsection*{Acknowledgements}
This work owes a great deal of intellectual debt to Ricky Liu's
paper \cite{Liu}.  This paper has its genesis in the Princeton
senior thesis \cite{K}. The first author wishes to acknowledge
support from a National Science Foundation Graduate Research
Fellowship.  The second author wishes to acknowledge support from
the National Science Foundation (Award number DMS-0701753), a grant from the National Security Agency (Award number 111011), and a Collaboration Grant from the Simons Foundation (Award number 245977). In the course of the preparation of this work we benefited from conversations with Nir Avni, Manjul Bhargava, Alice Medvedev, and Alireza Salehi-Golsefidi, and Christopher Voll. We wish to thank Bhama Srinivasan for useful conversations, and especially  for her crucial observation Lemma \ref{lem:bhama}, and Eun Hye Lee for numerical computations on UIC's Argo cluster to provide support for equation \eqref{wishful}. Thanks are also due to the referee for reading the paper very carefully, and for pointing out errors and inconsistencies that have led to the improvement of the paper.

\section{Geometry and $p$-adic integrals}\label{geometry-padic}
In this section we study a multivariable version of the Igusa zeta integral following the method of \cite{Denef-degree} and \cite{duSG}. We start with some geometric preparation.

\subsection{Resolutions with good reduction}\label{good-reduction}
We recall the the material of Section 2 of \cite{Denef-degree}. In this section $K$ is an arbitrary field of characteristic zero, $R$ a discrete valuation subring of $K$ with field of fractions $K$, $P$ unique maximal ideal, and residue field $\overline{K}$ which we assume to be perfect. Let $f(\underline{X}) \in K[\underline{X}]$, $\underline{X} = (X_1, \cdots, X_m)$ be a nonzero polynomial. Let $\mX = Spec \, K[\underline{X}]$, $\tilde{\mX} = Spec \, R[\underline{X}]$, $\overline{\mX} = Spec \, \overline{K}[\underline{X}]$, and 
$$
\mD = Spec \, \left( K[\underline{X}] / (f) \right) \subset \mX. 
$$
A resolution $(\mY, h)$ for $f$ over $K$ consists of a closed integral subscheme $\mY$ of $\bP_\mX^k$ for some $k$, and the morphism $h: \mY \to \mX$ which is the restriction of the projective morphism $\bP_\mX^k \to \mX$ such that:
\begin{enumerate} 
\item $\mY$ is smooth over $Spec \, K$; 
\item the restriction $h: \mY \setminus h^{-1}(\mD) \to \mX \setminus \mD$ is an isomorphism;
\item the reduced scheme $(h^{-1}(\mD))_{red}$ associated to $h^{-1}(\mD)$ has only normal crossings. 
\end{enumerate} 
Let $\mE_i$, $i \in T$, be the irreducible components of $(h^{-1}(\mD))_{red}$. For $i \in T$, we define $N_i$ to be the multiplicity of $\mE_i$ in the divisor of $div \, (f \circ h)$ on $\mY$, and let $\nu_i-1$ be the multiplicity of $\mE_i$ in the divisor of $h^*(dx_1 \wedge \cdots \wedge dx_m)$. We have $N_i, \nu_i \geq 1$ for all $i\in T$. 

\

We think of $\bP_\mX^k$ as an open subscheme of $\bP_{\tilde{\mX}}^k$. If $\zZ$ is a closed subscheme of $\bP_\mX^k$, we define $\tilde{\zZ}$ to be the  scheme theoretic closure of $\zZ$ in $\bP_{\tilde{\mX}}^k$. We also set $\overline{\mZ} = \tilde{\mZ} \times_R Spec \, \overline{K}$, and we call it the reduction of $\mZ$ mod $P$. 

\

Let $\tilde{h}: \tilde{\mY} \to \tilde{\mX}$ be the restriction to $\tilde\mY$ of the projective morphism $\bP_{\tilde\mX}^k \to \tilde{\mX}$, and $\bar{h}: \overline{\mY} \to \overline{\mX}$ be obtained from $\tilde{h}$ by base extension. We say $(\mY, h)$ has good reduction mod $P$ if the following two conditions are satisfied: 
\begin{enumerate} 
\item $\overline{\mY}$ is smooth over $Spec \, \overline{K}$; 
\item $\bar{\mE}_i$ is smooth over $Spec \, \overline{K}$, for all $i \in T$, and $\cup_i \bar{\mE}_i$ has only normal crossings; and 
\item for $i \ne j$, $\bar{\mE}_i$ and $\bar{\mE}_j$ have no common irreducible components. 
\end{enumerate} 

Let $K'$ be a field containing $K$, $R'$ a discrete valuation subring of $K'$ who fraction field is $K'$, and which contains $R$, with maximal ideal $P'$ containing $P$, and with perfect residue field. Suppose $(\mY, h)$ be a resolution of $f$ over $K$ as above. Let $\mY' = \mY \times_K Spec \, K'$ and $h': \mY' \to \mX' = Spec \, K'[\underline{X}]$ be obtained from $h$ by base extension. Proposition 2.3 of \cite{Denef-degree} says that then $(\mY', h')$ is a resolution of $f$ over $K'$. Moreover, if $(\mY, h)$ is a resolution with good reduction mod $P$, $(\mY', h')$ has good reduction mod $P'$. 

\

In the arithmetic case, let $F$ be a number field, and $\cO_F$ its ring of integers. Let $f(\underline{X}) \in F[\underline{X}]$, $\underline{X} = (X_1, \cdots, X_m)$. Let $(\mY, h)$ be a resolution for $f$. For any maximal ideal $\fp$, we consider the discrete valuation ring $\cO_{F, \fp}$ with maximal ideal $\fp \cO_{F, \fp}$. Note that the field of fractions of $\cO_{F, \fp}$ is $F$. Theorem 2.4 of \cite{Denef-degree} then states that for almost all $\fp$, $(\mY, h)$ is a resolution with good reduction mod $\fp \cO_{F, \fp}$. As a corollary, if $F_\fp$ is the $\fp$-adic completion of $F$, and $\cO_\fp$ its ring of integers, and by abuse of notation $\fp$ its unique prime ideal, then $(\mY, h)$ is a resolution of $f$ over $F_\fp$ with good reduction mod $\fp$ for almost all $\fp$. 

\subsection{Multivariable cone integral} \label{multicone}
For a finite extension $F$ of $\Q_p$, we let be $\cO_F$ its ring of integers, $\fq$ the maximal ideal, $|.|_F$ its normalized absolute value, and $v_F$ the corresponding discrete valuation.  Let $q$ be the size of $\overline{F}$, the residue field of $F$. 

\

Let $f_1, \cdots, f_l$ and $g_1, \cdots, g_l$ be polynomials in the variables $\underline{X} = (X_1, \cdots, X_m)$ with rational coefficients. We denote by $\psi_F(\underline{X})$ the first order formula 
$$
v_F(f_i(\underline{X})) \leq v_F(g_i(\underline{X})), \,\,\,\,\, i = 1, \dots, l. 
$$
As before we call the formula $\psi_F(\underline{X})$ a cone condition, and the polynomials $f_i, g_i$, $1 \leq i \leq l$, the cone data.

We define 
$$
V_{F, \psi} = \{ \underline{x} \in \cO_F^m; \psi(\underline{x}) \}. 
$$
If $h_0, h_1, \dots, h_k$ are polynomials in $\underline{X}$ with rational coefficients, we define the cone integral in $k$ complex variables $\underline{s}=(s_1, \cdots, s_k) \in \C^k$  with respect to $\psi$ by 
$$
Z_\psi(\underline{s}; F) = \int_{V_{F, \psi}} |h_0(\underline{x})| \cdot |h_1(\underline{x})|^{s_1} \cdots |h_k(\underline{x})|^{s_k} \cdot |d\underline{x}|. 
$$
Our first goal here is to find an explicit formula for $Z_\psi$ for $p$ outside a finite set of primes.  In this section, following the method of \cite{duSG} closely we will find an explicit formula for our multivariable cone integral which depends on the numerical invariants of a resolution. 

\ 

Let $(\mY_\Q, h_\Q)$ be a resolution of the polynomial $\Phi = \prod_i h_i. \prod_j f_j g_j$ over $\Q$, and assume that the prime $p$ is such that $(\mY_\Q, h_\Q)$ has good reduction mod $p$, and $ \Phi\not\equiv 0 \mod p$. Let $(\mY, h)$ be the resolution of $\Phi$ over  $F$ obtained by base extension. Then $(\mY, h)$ has good reduction mod $\fq$. 

\

Let $a \in \overline{\mY}(\overline{F})$. Since $\overline{\mY}$ is a closed subscheme of $\widetilde{\mY}$, $a$ is a closed point of $\widetilde{\mY}$. Let 
$$
T_a = \{ i \in T, a \in \overline{\mE}_i\} = \{ i \in T, a \in \widetilde{\mE}_i\}. 
$$
Let $r = |T_a|$ and write $T_a = \{ i_1, \cdots, i_r\}$. Then in the local ring $\cO_{\widetilde{\mY}, a}$ we write 
$$
\Phi \circ \tilde{h} = u c_1^{N_{i_1}} \dots c_r^{N_{i_r}}
$$
where $c_j \in \cO_{\widetilde{\mY}, a}$ generates the ideal of $\widetilde{\mE}_{i_j}$ and $u$ a unit in $\cO_{\widetilde{\mY}, a}$. Since $f_i, g_i, h_i$ divide $\Phi$, we can also write 
$$
f_i \circ \tilde{h} = u(f_i) c_1^{N_{i_1}(f_i)} \dots c_r^{N_{i_r}(f_i)}
$$
$$
g_i \circ \tilde{h} = u(g_i) c_1^{N_{i_1}(g_i)} \dots c_r^{N_{i_r}(g_i)}
$$
$$
h_i \circ \tilde{h} = u(h_i) c_1^{N_{i_1}(h_i)} \dots c_r^{N_{i_r}(h_i)}. 
$$
We define vectors $\underline{w}_j$, $1 \leq j \leq r$, by 
$$
\underline{w}_j = (N_{i_j}(h_1), \dots, N_{i_j}(h_k)) \in \N^k. 
$$

Define an integral $J_{a, \psi} (\underline{s}, F)$ by the following expression: 
$$
J_{a, \psi} (\underline{s}; F) = \int_{\theta^{-1}(a) \cap h^{-1}(V_{F, \psi})} |h_0 \circ h | \cdot |h_1 \circ h|^{s_1} \cdots |h_k \circ h|^{s_k} \cdot |h^*(dx_1 \wedge \cdots \wedge dx_m)|. 
$$
Here the function $\theta$ is defined as follows: Let $H = \{ b \in \mY(F), h(b) \in \cO_F^m\}$. A point $b \in H \subset \mY(F)$ can be represented by its coordinates $(x_1, \cdots, x_m, y_0, \cdots, y_k) \in F^m \times \bP_\mX^k(F)$ where $(x_1, \cdots, x_m) \in \cO_F^m$ and $(y_0, \dots, y_k)$ are homogeneous coordinates that are chosen to satisfy $\min_{i} v_F(y_i) =0$. We define $\theta(b) =(\overline{x_1}, \cdots, \overline{x_m}, \overline{y_0}, \cdots, \overline{y_k}) \in \overline{\mY}(\overline{F})$. 
The next step is to calculate each integral $J_{a, \psi}$. We have
$$
J_{a, \psi}(\underline{s}; F) = \int_{\theta^{-1}(a) \cap h^{-1}(V_{F, \psi})} |c_1|^{\underline{w}_1 \cdot \underline{s} + N_{i_1}(h_0) + \nu_{i_1}-1} \cdots |c_r|^{\underline{w}_r \cdot \underline{s} + N_{i_r}(h_0) + \nu_{i_r}-1} \, | dc_1 \wedge \dots \wedge dc_m|. 
$$ 
Since $\overline{c}_1, \dots, \overline{c}_m$ are in the maximal ideal of $\cO_{\overline{\mY}, a}$, we have that $c_1(b), \dots, c_m(b) \in \fq$ for all $b \in \theta^{-1}(a)$, and the map $c: \theta^{-1}(a) \to \fq^m$ given by 
$$
b \mapsto (c_1(b), \dots, c_m(b)). 
$$
is a bijection. Consequently, 
$$
J_{a, \psi}(\underline{s}; F) = \int_{V_{F, \psi}'} |c_1|^{\underline{w}_1 \cdot \underline{s} + N_{i_1}(h_0) + \nu_{i_1}-1} \cdots |c_r|^{\underline{w}_r \cdot \underline{s} + N_{i_r}(h_0) + \nu_{i_r}-1} \, | dc_1 \wedge \dots \wedge dc_m|
$$
where $V_{F, \psi}'$ is the set of all $y = (y_1, \dots, y_m) \in \fq^m$ such that for each $i$ satisfying $1 \leq i \leq l$
$$
\sum_{j=1}^r N_{i_j}(f_i) v_F(y_j) \leq \sum_{j=1}^r N_{i_j}(g_i) v_F(y_j). 
$$
Let $\underline{A}_{j, a} = w_j$ and $B_{j, a} = N_{i_j}(h_0)+ \nu_{i_j}$ for $1 \leq j \leq r$ and $\underline{A}_{j, a} = \underline{0}$ and $B_{j, a} = 1$ for $j > r$. Then 
$$
J_{a, \psi}(\underline{s}; F) = \sum_{(k_1, \dots, k_m) \in \Lambda} q^{-\sum_{j=1}^m k_j (\underline{A}_{j, a}.\underline{s} + B_{j, a} -1)} (q^{-k_1} - q^{-k_1 -1}) \dots (q^{-k_m}-q^{-k_m-1})
$$
$$
= (1-q^{-1})^m \sum_{(k_1, \dots, k_m) \in \Lambda} q^{-\sum_{j=1}^m k_j (\underline{A}_{j, a}.\underline{s} + B_{j, a})}, 
$$
where 
$$
\Lambda = \left\{(k_1, \dots, k_m) \in \N^m; \sum_{j=1}^r N_{i_j}(f_i) k_j \leq \sum_{j=1}^r N_{i_j}(g_i) k_j, 1 \leq i \leq l \right\}. 
$$
The set $\Lambda$ is the intersection of $\N^m$ with a rational polyhedral cone $C$ in $\R^m$. Write this cone as a disjoint union of simplicial cones $C_1, \dots, C_t$ with 
$$
C_i = \{\alpha_1 v_{i1} + \dots + \alpha_{m_i} v_{im_i}; \alpha_j \in \R_{> 0}, 1 \leq j \leq m_i \}
$$
where $\{v_{i1}, \dots, v_{im_i}\}$ is a linearly independent set of vectors in $\R^m$.

Then $\Lambda$ is the disjoint union of the following sets
$$
\Lambda_i = \{l_1 v_{i1} + \dots + l_{m_i} v_{im_i}; l_j \in \N, 1 \leq j \leq m_i\}.  
$$
Now $v_{jk} = (q_{jk1}, \dots, q_{jkm}) \in \R_{>0}^m$ for $1 \leq k \leq m_j$. Hence 
$$
J_{a, \psi}(\underline{s}; F) = (1-q^{-1})^m \sum_{i=1}^t \prod_{u=1}^{m_i} \frac{q^{-\underline{A}_{i, u, a}. \underline{s} - B_{i, u, a}} }{1- q^{-\underline{A}_{i, u, a}. \underline{s} - B_{i, u, a}}}
$$
with $\underline{A}_{i, u, a} = \sum_{j=1}^m q_{iuj} \underline{A}_{j, a}$ and $B_{i, u, a} = \sum_{j=1}^m q_{iuj} B_{j, a}$.  

For each $I \subset T$ define 
$$
c_{F, I} = \left| \{ a \in \overline\mY (\overline{F}); a \in \overline\mE_i \text{ if and only if } i \in I \}\right|, 
$$
and put $\underline{A}_{i, u, I} = \underline{A}_{i, u, a}$ and $B_{i, u, I} = B_{i, u, a}$ for any $a \in \{ x \in \overline\mY (\overline{F}); x \in \overline\mE_i \text{ if and only if } i \in I \}$.

Clearly, 
$$
Z_\psi(\underline{s}; F) = \sum_{a \in \overline{\mY}(\overline{F})} J_{a, \psi}(\underline{s}; F). 
$$
Putting everything together 
$$
Z_\psi(\underline{s}; F) = (1-q^{-1})^m \sum_{I \subset T} c_{F, I} \sum_{i=1}^{t_I} \prod_{u=1}^{m_i} \frac{q^{-\underline{A}_{i, u, I}. \underline{s} - B_{i, u, I}} }{1- q^{-\underline{A}_{i, u, I}. \underline{s} - B_{i, u, I}}}. 
$$
The absolute convergence of the integral is guaranteed if 
$$
\underline{A}_{i, u, I}. \Re \underline{s} + B_{i, u, I} > 0
$$
for all $I \subset T$, $1 \leq i \leq t$, and $1 \leq u \leq m_i$, where 
$$
\Re \underline{s} = (\Re s_1, \dots, \Re s_k). 
$$
We note that the domain of the absolute convergence depends only on the geometry of our data, and not on the particular choice of the field $F$. 

\

As in \cite{duSG}, we derive another expression for the integral.  Set 

$$
\overline{D_T} = \left\{(x_1, \dots, x_t) \in \R^t_{\geq 0} ; \sum_{j=1}^t N_j(f_i) x_j \leq \sum_{j=1}^t N_j(g_i) x_j, 1 \leq i \leq l \right\}
$$
where $t = |T|$. This is a closed cone. This cone is a disjoint union of open simplicial pieces called $R_k$, $0 \leq k \leq w$. We assume that the fundamental region for the lattice points of $R_k$ has no lattice points in its interior. We will assume that $R_0=(0, \dots, 0)$ and that $R_1, \dots, R_q$ are all the open one dimensional edges of the cone $\overline{D_T}$. Write 
$$
R_k = \{ \alpha \underline{e}_k = \alpha (q_{k1}, \dots, q_{kt}); \alpha > 0\}. 
$$
For any $0 \leq k \leq w$, there is a subset $M_k \subset \{1, \dots, q\}$ such that 
$$
R_k = \left\{\sum_{j \in M_k} \alpha_j \underline{e}_j, \forall j \in M_k \right\}.
$$
Let $m_k : = |M_k| \leq t$. For each $I \subset T$ set 
$$
D_I = \left\{(k_1, \dots, k_t) \in \overline{D_T}; k_i > 0, \forall i \in I, k_i = 0, \forall i \in T\setminus I  \right\}
$$
$$
\Delta_I = D_I \cap \N^t. 
$$
We also set $\overline{D_T} = \overline{\Delta_T}$.  For each $I \subset T$, there is a subset $W_I \subset \{0, \dots, w\}$ such that 
$$
D_I = \bigcup_{k \in W_I} R_k. 
$$

Suppose $a \in \overline\mY(\overline{F})$ is such that $a \in \overline\mE_i$ if and only if $i \in I$. Then we have 
$$
J_{a, \psi}(\underline{s}; F) = p^{-(m-|I|)} \int_{V_F'} \prod_{i \in I} |z_i|^{N_i(h_1)s_1 + \dots + N_i(h_k)s_k + N_i(h_0) + \nu_i -1} \prod_{i \in I} |dz_i|
$$
with $V_F'$ the set of all $(z_i)_{i \in I} \in \fq^{|I|}$ satisfying for $1 \leq j \leq l$
$$
\sum_{i \in I} N_i(f_j) v_F(z_i) \leq \sum_{i \in I} N_i(g_j) v_F(z_i). 
$$
Then 
$$
J_{a, \psi}(\underline{s}; F) = p^{-(m-|I|)} (1-p^{-1})^{|I|} \sum_{(k_1, \dots, k_t) \in \Delta_I} q^{-\sum_{j=1}^t k_j(N_i(h_1)s_1 + \dots + N_i(h_k)s_k + N_i(h_0) + \nu_i)}
$$
$$
= \sum_{k \in W_I}  p^{-(m-|I|)} (1-p^{-1})^{|I|} \sum_{(k_1, \dots, k_t) \in R_k \cap \N^t} q^{-\sum_{j=1}^t k_j(N_j(h_1)s_1 + \dots + N_j(h_k)s_k + N_j(h_0) + \nu_j)}
$$
as $D_I = \cup_{k \in W_I} R_k$. As
$$
R_k \cap \N^t = \left\{ \sum_{j \in M_k} \alpha_j \underline{e}_j; \alpha_j \in \N, \forall j \in M_k \right\}
$$
we have 
$$
J_{a, \psi}(\underline{s}; F) =  \sum_{k \in W_I} p^{-(m-|I|)} (1-p^{-1})^{|I|} \prod_{j \in M_k} \frac{q^{-(\underline{A_j}.\underline{s} + B_j)}}{1- q^{-(\underline{A_j}.\underline{s} + B_j)}}
$$
with 
$$
\underline{A}_j = \sum_{i=1}^t q_{ji} \underline{N}_i
$$
$$
B_j = \sum_{i=1}^t q_{ji} (N_i(h_0) + \nu_i),  
$$
and 
$$
\underline{N}_i = (N_i(h_1), \dots, N_i(h_k)). 
$$
So if we set $c_{F, k} = c_{F, I}$ and $I_k = I$ if $k \in W_I$, for every non-archimedean local field $F$ where the resolution has good reduction we have 
$$
Z_\psi(\underline{s}; F) = \sum_{k=0}^w (q-1)^{|I_k|} q^{-m} c_{F, k} \prod_{j \in M_k} \frac{q^{-(\underline{A_j}.\underline{s} + B_j)}}{1- q^{-(\underline{A_j}.\underline{s} + B_j)}}. 
$$
In the situation where the resolution is not necessarily of good reduction, following the argument of Proposition 3.3 of \cite{duSG} one proves that there exists a finite set $B_F$ such that for every $b \in B_F$ there is an associated subset $I_b \subset T$ and an integer $e_b$ such that 
\begin{equation}\label{ramified}
Z_\psi(\underline{s}; F) = \sum_{b \in B_F} \sum_{k \in W_{I_b}} (q-1)^{|I_b|} q^{-m}  \prod_{j \in M_k} \frac{q^{-e_b(\underline{A_j}.\underline{s} + B_j)}}{1- q^{-(\underline{A_j}.\underline{s} + B_j)}}. 
\end{equation}

\subsection{Application to some volume computations}\label{alphanarrow}
Let $F$ be a finite extension of $\Q_p$ with ring of integers $\cO_F$ and $|.|_F$ its normalized absolute value. We fix a uniformizer $\varpi_F$ for $F$. Let $q$ be the size of the residue field of $F$. For $\underline{x} = (x_1, \cdots, x_n) \in (F^\times)^n$, and $\underline{\alpha} = (\alpha_1, \cdots, \alpha_n) \in \R^n$, we define $v_F(\underline{x}) = (v_F x_1, \dots, v_F x_n)$, and $|\underline{x}|_F^{\underline{\alpha}} = \prod_i |x_i|_F^{\alpha_i}$. We define $vol_F$ and $vol_{F^n}$, to be the normalized Haar measure on $F$, and on $F^n$, respectively.  If $\underline{k} = (k_1, \dots, k_n) \in \Z^n$, and $\alpha \in F$ is nonzero, we set $\alpha^{\underline{k}} = (\alpha^{k_1}, \dots, \alpha^{k_n})$; in particular, $\varpi_F^{\underline{k}} = (\varpi_F^{k_1}, \dots, \varpi_F^{k_n})$. 

\

Let $\underline{X} = (X_1, \cdots, X_n)$ and $\underline{Y} = (Y_1, \cdots, Y_m)$, and let $f_i, g_i \in \Z[\underline{X}; \underline{Y}]$, $1 \leq i \leq k$, be polynomials. For each $\underline{x} \in \cO_F^n$, define a set 
$$
V_F(\underline{x})= \{ \underline{y} \in \cO_F^m; v_F(f_i(\underline{x}; \underline{y}))\leq v_F(g_i(\underline{x}; \underline{y})) , 1 \leq i \leq k\}. 
$$
We will assume that $V_F(\underline{x})$ is $F$-\emph{round} in that it is invariant under the action of units of the local field, i.e. $V_F(\underline{x}) = V_F(\underline{x}')$ if $v_F(\underline{x}) = v_F(\underline{x}')$.  With abuse of language, when we say $V$, we mean the assignment that takes an extension $F$ of $\Q_p$ and an element $\underline{x} \in \cO_F^n$, and returns the set $V_F(\underline{x})$. We will call $V$ \emph{round} if for all $F$, $V_F(\underline{x})$ is $F$-round. 
\begin{defn}
Let $\underline{\alpha} = (\alpha_1, \cdots, \alpha_n) \in \R^n$, $\ell \in \N$, and $P \in \R[X_1, \dots, X_n]$ with positive coefficients. We say $V$ is $(\ell, \underline{\alpha}, P, F)$-narrow, if for all 
$\underline{x} \in (\cO_F\setminus \{0\})^n$ we have 
$$
vol_{F^m} (V_F(\underline{x})) \leq P(v_F(\underline{x})) q^{-\ell} |\underline{x}|_F^{\underline{\alpha}}. 
$$
\end{defn}
Now here is the theorem: 
\begin{theorem}\label{thm:alphanarrow}
Suppose there is $\underline{\alpha} = (\alpha_1, \cdots, \alpha_n) \in \R^n$, $\ell \in \N$, $P \in \R[X_1, \dots, X_n]$ with positive coefficients, and an infinite set of primes $\mathcal{P}$ such that for all $p \in \mathcal{P}$ the set $V$ is $(\ell, \underline{\alpha}, P, \Q_p)$-narrow. Then $V$ is $(\ell, \underline{\alpha}, P, \Q_p)$-narrow for almost all primes $p$.  
\end{theorem}

In the statement of the theorem ``almost all" means all but possibly finitely many. 

\begin{proof} Let $F=\Q_p$ for $p \in \mathcal{P}$.  In order to prove the theorem, we consider the following integral: 
$$
Z_V(\underline{s}) = \int_{\cO_F^n} \, vol_{F^m}(V_F(\underline{x}))|\underline{x}|^{\underline{s}} \, dx 
$$
$$
= (1-p^{-1})^n \sum_{\underline{k}\in \Z_{\geq 0}^m} vol_{F^m}(V_F(\underline{\varpi}_F^{\underline{k}}))p^{-|\underline{k}|} p^{-\underline{k}.\underline{s}}. 
$$

On the other hand, we write 
$$
Z_V(\underline{s}) = \int_{\cO_F^{m+n};  v(f_i(\underline{x}; \underline{y}))_F \leq v(g_i(\underline{x}; \underline{y}))_F , 1 \leq i \leq k}|x_1|^{s_1} \dots |x_n|^{s_n} \, |d\underline{x}|\, |d \underline{y}|.  
$$
This is a multivariable cone integral. 

\

Since the set $\mathcal{P}$ is infinite, we may assume that $p$ is good in the sense of \S \ref{good-reduction}.  By \S\ref{multicone} we have 
$$
Z_V(\underline{s}) = \sum_{k=0}^w (p-1)^{|I_k|} p^{-m-n} c_{F, k} \prod_{j \in M_k} \frac{p^{-(\underline{A_j}.\underline{s} + B_j)}}{1- p^{-(\underline{A_j}.\underline{s} + B_j)}}
$$
with non-negative integer vectors $\underline{A}_j$ and non-negative integers $B_j$. Regrouping terms gives 
$$
Z_V(\underline{s}) = \sum_{\underline{k}} p^{-\underline{k}.\underline{s}} \sum_{i=0}^w (p-1)^{|I_i|}p^{-m-n} c_{F, i} \left(\prod_{j \in M_i} \sum_{\alpha_j=1}^{+\infty}\right)_{\sum_j \alpha_j \underline{A}_j = \underline{k}} p^{-\alpha_j B_j}
$$
where the notation 
$$
\left(\prod_{j \in M_i} \sum_{\alpha_j=1}^{+\infty}\right)_{\sum_j \alpha_j \underline{A}_j = \underline{k}}
$$
means we have only considered those $\alpha_j$'s that satisfy $\sum_j \alpha_j \underline{A}_j = \underline{k}$. Comparing the two expressions for $Z_V$ gives
$$
 vol_{F^m}(V_F(\underline{\varpi}_F^{\underline{k}})) = (1-p^{-1})^{-n}p^{|\underline{k}|}\sum_{i=0}^w (p-1)^{|I_i|}p^{-m-n} c_{F, i} \left(\prod_{j \in M_i} \sum_{\alpha_j=1}^{+\infty}\right)_{\sum_j \alpha_j \underline{A}_j = \underline{k}} p^{-\alpha_j B_j}
$$
$$
= \sum_{i=0}^w c_{F, i} (1-p^{-1})^{-n}p^{|\underline{k}|} (p-1)^{|I_i|}p^{-m-n}  \left(\prod_{j \in M_i} \sum_{\alpha_j=1}^{+\infty}\right)_{\sum_j \alpha_j \underline{A}_j = \underline{k}} p^{-\alpha_j B_j}. 
$$
We note that if $|I_i|>m+n$, then $c_{F, i} = 0$. As a result we may write 
$$
vol_{F^m}(V_F(\underline{\varpi}_F^{\underline{k}})) = \sum_{i=0}^w c_{F, i}  (1-p^{-1})^{-n}p^{|\underline{k}|} (p-1)^{|I_i|}p^{-m-n}  P_{i, \underline{k}} (p^{-1})
$$
with $P_{i, \underline{k}}(X)$ a polynomial with positive integral coefficients which depends only on $i$ and $\underline{k}$, and not on the choice of the field $F$. Further, the number of terms of $P_{i, \underline{k}}$ depends on $\underline{k}$ in a polynomial fashion. In particular there are no cancellations between the terms. These observations imply that $V$ is $(\underline{\alpha}, F)$-narrow if and only if for each $i=0, \dots, w$, we have some polynomial with positive coefficients $P$ such that 
$$
c_{F, i}  (1-p^{-1})^{-n}p^{|\underline{k}|} (p-1)^{|I_i|}p^{-m-n}  P_{i, \underline{k}} (p^{-1})
\leq P(k_1, \dots, k_n) p^{-\ell}p^{-\alpha_1 k_1 - \dots - \alpha_n k_n}. 
$$
This is true if and only if 
$$
c_{F, i} p^{|\underline{k}|} p^{|I_i|-m-n}  P_{i, \underline{k}} (p^{-1})
\leq P(k_1, \dots, k_n)p^{-\ell}p^{-\alpha_1 k_1 - \dots - \alpha_n k_n}. 
$$
Proposition 4.9 combined with Proposition 4.13 of \cite{duSG} implies that, after letting $p$ become larger in $\mathcal{P}$, this inequality is true if and only if 
$$
p^{m+n - |I_i|} p^{|\underline{k}|} p^{|I_i|-m-n}  P_{i, \underline{k}} (p^{-1})
\leq P(k_1, \dots, k_n)p^{-\ell}q^{-\alpha_1 k_1 - \dots - \alpha_n k_n}, 
$$
which is equivalent to 
$$
p^{|\underline{k}|}   P_{i, \underline{k}} (p^{-1})
\leq P(k_1, \dots, k_n)p^{-\ell}p^{-\alpha_1 k_1 - \dots - \alpha_n k_n}. 
$$
Since $\mathcal{P}$ is infinite,  we can let $p \to \infty$, and as a result an inequality of this nature is valid if and only if it is true for degree reasons.  The theorem now follows. 
\end{proof}

\begin{rem}
Here is a variation of the above theorem which may be useful in other contexts. There is a finite set $S$ of primes such that every $p \notin S$ has the following property:  If there is $\alpha \in \R^n$, $\ell \in \N$, and $P \in \R[X_1, \dots, X_n]$ such $V$ is $(\ell, \underline{\alpha}, P, F)$-narrow for every $F$ finite extension of $\Q_p$, then for all $q \notin S$, $V$ is $(\ell, \underline{\alpha}, E)$-narrow for every $E$ finite extension of $\Q_q$. 
\end{rem}

\section{The proof of Theorem \ref{thm:main}}\label{proof:thmmain}
\subsection{Tauberian theorem}\label{section tauberian}

We will use the Tauberian theorem of \cite{C-T}, Appendix A, in the following
form:

\begin{theorem}\label{tauberian application}
Let
$$
F(s) = \sum_{n=1}^\infty \frac{a_n}{n^s}
$$
be a Dirichlet series with an Euler product
$$
F(s) = \prod_{p} F_p(s).
$$
Suppose each Euler factor is of the form
$$
F_p(s) = 1 + \sum_{l \geq 1} \frac{a_l(p)}{p^{ls}}
$$
where $a_1(p) = k$, a positive integer independent of $p$, and $a_l(p)$
are non-negative real numbers. Suppose there is a $\delta_0$ satisfying $\frac{1}{2}
\leq \delta_0 <1$ such that for $\sigma > \delta_0$ we have
$$
\sum_p \sum_{l \geq 2} \frac{a_l(p)}{p^\sigma} < + \infty.
$$
Then there is a polynomial $P$ of degree $k-1$ such that for all
$\epsilon >0$
$$
\sum_{n \leq B} a_n = BP(\log B) + O_\epsilon (B^{\delta_0 +
\epsilon})
$$
as $B \to \infty$.
\end{theorem}

\subsection{Outline of the proof of Theorem \ref{thm:main}}
\label{outline-thm2}
If $p$ is unramified in $K$, we write 
$$
p \cO_K = \fp_1 \fp_2 \dots \fp_r, 
$$
where each $\fp_i$ is a prime ideal in $\cO_K$, and let 
$$
f_i = f(\fp_i/p)
$$
denote the residue degree of the prime $\fp_i$.

Then 
$$
\cO_K \otimes_\Z \Z_p = \prod_i  \cO_{\fp_i}
$$
where $\cO_{\fp_i}$ is the ring of integers of the completion of $K$ at the prime 
$\fp_i$, and the isomorphism class of $\cO_K \otimes_\Z \Z_p$ is determined by the multi-set $f_p = \{f_1, \cdots, f_r\}$, called the {\em type} of $p$. The type of a prime is always a partition of $n$.   We typically write the type of an unramified prime $p$ in the form $f_p= 1^v 2^w r_1^{e_1} \cdots r_k^{e_k}$, where $1 < 2 < r_1 < \cdots < r_k$ are the distinct residue degrees, and $v, w, e_1, \cdots, e_k$ are the number of times each of these appears. 

\

The starting point of the proof of the theorem is the following proposition: 

\begin{proposition}\label{prop:1}
If $p$ is an unramified prime of type $f_p = 1^v 2^w r_1^{e_1} \cdots r_k^{e_k}$, then 
$$
a_1(p) = w + {v \choose 2};  
$$
in particular $a_1(p)$ depends only on the type $f_p$. 
\end{proposition}
We will present the proof of this proposition in Section \ref{proof:prop1}.  Given a partition $f$ as above, we let 
$$
a(f) = w + {v \choose 2}. 
$$
Then we observe that the condition that $p$ has type $1^u 2^w r_1^{e_1} \cdots r_k^{e_k}$ is Chebotarev condition in $G=\Gal(E/\Q)$ in the sense that there are a number of conjugacy classes $\cC_i \subset G$, $1 \leq i \leq t$, such that $p$ has type $1^u 2^w r_1^{e_1} \cdots r_k^{e_k}$ if and only if 
$$
\left(\frac{E/\Q}{p} \right) = \cC_i
$$
for some $i$. Here $\left(\frac{E/\Q}{p} \right)$ is the Frobenius conjugacy class of $p$ in $G$.  Next we use the following fact: 

\begin{proposition}\label{prop:2}
Let $L/K$ be a Galois extension of number fields with Galois group $H = \Gal(L/K)$. Let 
$\cC \subset H$ be a conjugacy class and define 
$$
F_\cC (s) = \prod_{p \, \text{unramified} \atop \left(\frac{L/K}{p}\right) = C} (1- N(p)^{-s})^{-1}. 
$$
Then $F_\cC(s)$ converges absolutely for $\Re s >1$. Furthermore, $F_\cC(s)^{|H|}$ has an analytic continuation to a meromorphic function on an open set containing $\Re s \geq 1$ with a unique pole of order $|\cC|$ at $s=1$. 
\end{proposition} 
We will present the proof of this proposition in Section \ref{proof:prop2}. Now suppose a partition $f$ of $n$ is given. On the one hand $f$ can be type of a prime $p$, and on the other hand $p$ determines a conjugacy class in $S_n$. It is a well-known fact that if $p$ has type $f$ in $K/\Q$, then $\left(\frac{E/\Q}{p} \right)$ has cycle type $f$.  Given a type $f$, we define $b(f)$ be the number of elements of $G$ of cycle type $f$ in $S_n$. Combining everything done so far one concludes that the function $f(s)$ in the statement of Theorem~\ref{thm:main} has a pole at $s=1$ of order 
\begin{equation}\label{equation}
r : = \frac{1}{|G|} \sum_{f \text{ type }} a(f) b(f).  
\end{equation}
We finally have the following statement: 
\begin{lemma}[B. Srinivasan] \label{lem:bhama}
We have $r=r_2$. 
\end{lemma}
\begin{proof}[Proof of Lemma] We define a function $\alpha$ on $G$ as follows. If $g$ is of cycle decomposition type $f$, we set $\alpha(g) = a(f)$. We note that the expression on the right is equal to $\langle \alpha, \psi \rangle$ where $\psi$ is the trivial character of $G$, and $\langle, \rangle$ is the inner product on the space of class functions of $G$. The function $\alpha$ is character of the permutation representation $\pi$ of $G$ on the set of $2$-element subsets of $\{1, 2, \dots, n\}$. In fact, if $g$ is of type $f$ as above, then it is clear that it fixes ${u \choose 2} + w$ $2$-element sets. Then the expression on the right hand side of \eqref{equation} is equal to the multiplicity of the trivial representation in $\pi$. For every orbit of $G$ on the set of $2$-element subsets of $\{1, 2, \dots, n\}$ we get a copy of the trivial representation in $\pi$, and these are the only copies of the trivial representation in $\pi$. It is easily seen that if $G$ is transitive the number of such orbits is equal to $r_2$. 
\end{proof}
Theorem~\ref{thm:main} now follows from a  standard Tauberian argument.

\subsection{Proof of Proposition \ref{prop:1}}\label{proof:prop1}
We first give an overview of the proof of Proposition \ref{prop:1}.  A result of \cite{GSS} shows that determining $a_1(p)$ is equivalent to a counting problem about certain lower-triangular matrices.  By Lemma 5.18 of \cite{Br} $\cO_p:=\cO_K\otimes_\Z \Z_p$ is a $\Z_p$-module of rank $n$. By choosing a special type of basis for $\cO_p$ and then applying elementary row operations the lower-triangular matrices we consider will be of a relatively simple form.   We then break up the overall computation of $a_1(p)$ into a few parts depending on the type of $p$.  The proof of Proposition \ref{prop:1} depends on the following lemmas.

\begin{lemma}\label{impossible} 
Let $L/\Q_p$ be an extension of degree $n$. If $n>2$, the ring of integers $\cO_L$ of $L$ does not have any multiplicatively closed sublattices of index $p$ that are $\Z_p$ modules of rank $n$. 
\end{lemma}

This result shows that in order to determine $a_1(p)$ in general, we need only determine primes of a restricted type.

\begin{lemma}\label{a1preduce}
Let $p$ be a prime of type $f_p = 1^v 2^w r_1^{e_1} \cdots r_k^{e_k}$, and let $q$ be a prime of type $f_q = 1^v 2^w$. Then $a_1(p) = a_1(q)$. 
\end{lemma}

We will determine $a_1(p)$ for primes of this type by considering primes of type $1^v$ and primes of type $2^w$ separately.  The next lemma follows directly from \cite{Liu} Proposition 1.1.

\begin{lemma}\label{LemmaLiu}
Let $p$ be a prime of type $f_p = 1^v$.  Then $a_1(p) = \binom{v}{2}$.
\end{lemma}

\begin{lemma}\label{a1pairs}
Let $p$ be a prime of type $f_p = 2^w$.  Then $a_1(p) = w$.
\end{lemma}
The proof of Proposition \ref{prop:1} will follow from combining these results in the following way.
\begin{lemma}\label{a1pvw}
Let $p$ be a prime of type $f_p = 1^v 2^w$.  Then $a_1(p) = \binom{v}{2} + w$.
\end{lemma}
  
We now explain how to interpret $a_1(p)$ in terms of a counting problem about lower-triangular matrices.  The first observation is that $a_1(p)$ depends only on $\cO_p$ and not on $K$.  We choose any ordered basis of this ring, $\{v_1,\ldots, v_n\}$ and represent a subring $L$ of $\cO_p$ by a matrix $M$ where the $i$th column corresponds to $v_i$ and $L$ is generated by the rows of $M$.  The entries of this matrix are in $\Z_p$.  By elementary linear algebra, a version of Gauss-Jordan elimination over $\Z_p$, we are free to suppose that $M$ is lower triangular. Multiplying a row of $M$ by a unit in $\Z_p$ does not change the subring generated by $M$. Therefore, we may suppose that the $(i,i)$ entry of $M$ is equal to $p^{k_i}$ for some $k_i \ge 0$. 

Let $\M(p)$ denote the set of all lower triangular matrices whose rows generate a subring of $\cO_p$ with respect to this ordered basis.  We can now present a slight modification of a proposition of Grunewald, Segal and Smith \cite{GSS}.

\begin{proposition}\label{GSS}
For every prime $p$,
\[\eta_{K,p}(s) = (1- p^{-1})^{-n} \int_{M\in \M(p)} |x_{11}|^{s-n}  |x_{22}|^{s-(n-1)} \cdots  |x_{nn}|^{s-1} |dv|,\]
where $|dv|$ is the additive Haar measure of the $p$-adic lower triangular matrices.  
\end{proposition}

The index of a subring $L \subseteq \cO_p$ is the determinant of any matrix $M \in \M(p)$ generating $L$.  By definition, $a_1(p)$ is equal to the $p^{-s}$ coefficient of the integral in this proposition.  We therefore need only consider matrices $M\in \M(p)$ where exactly one $x_{ii}$ is equal to $p$ and all others are equal to $1$.  

Suppose the rows of $M$ generate a subring of $\cO_p$ of index $p$ and suppose that $x_{jj} = 1$ for some $j$.  By adding multiples of the $j$th row of $M$ to its other rows we can set each of the nondiagonal entries in column $j$ to $0$ without changing the subring generated by this matrix.  In fact, by applying a version of Gauss-Jordan elimination we can simultaneously accomplish this for each column which has its diagonal entry equal to $1$.   This gives a matrix that is diagonal except for a single column that may have nonzero entries below the diagonal.  We give an example below:
\[
\left(\begin{array}{ccccc}
1 & 0& 0& 0 & 0\\
0& 1 &0 & 0& 0\\
0& 0& p & 0& 0\\
0& 0& a_1 & 1 &0  \\
0&0 & a_2 & 0 & 1 
\end{array} \right).
\]

Suppose the rows of $M$ generate a subring of $\cO_p$ of index $p,\ x_{jj} = p$ for some $j$, and every other column of $M$ has a single $1$ on the diagonal and is $0$ otherwise.  Let $\{a_0, a_1, a_2, \ldots, a_{p-1}\}$ be some choice of representatives for $\Z_p / p \Z_p$ with $a_0 = 0$ and $a_1 = 1$.  By adding multiples of row $j$ to the rows below it, we may suppose that the entries $x_{j+1,j}, x_{j+2,j},\ldots, x_{n,j}$ are all elements of $\{a_0,\ldots, a_{p-1}\}$.  These representatives are uniquely defined by the subring, but the elements of a matrix generating this subring can be changed by an arbitrary element of $p\Z_p$.  We note that the normalized volume of $p\Z_p$ is $p^{-1}$.

This reduction gives a map from subrings of $\cO_p$ of index $p$ given by a matrix $M$ with $x_{jj} = p$ and all other diagonal entries equal to $1$ to tuples $(x_{j+1,j}, x_{j+2,j},\ldots, x_{n,j})$ where each $x_{i,j} \in \{a_0,\ldots, a_{p-1}\}$.  Let $a_1(p,j)$ be the size of the image of this map.  In the case $j=n$, if the matrix $M$ with diagonal entries all equal to $1$ except for $x_{n,n} = p$ and all other entries equal to $0$ generates a subring of $\cO_p$ of index $p$, then we define $a_1(p,n) = 1$.  Otherwise, $a_1(p,n) = 0$.  This description along with Proposition \ref{GSS} shows the following.

\begin{lemma}\label{a1sum}
We have $a_1(p) = \sum_{j=1}^n a_1(p,j)$.
\end{lemma}

The particular basis that we choose for $\cO_p$ has a major effect on the multiplication of rows of the matrix generating a subring.  Our next goal is to pick a convenient basis for this module.

Suppose that $p$ is an unramified prime of type $f_p = 1^v 2^w r_1^{e_1} \cdots r_k^{e_k}$ where the $r_i$ are distinct and greater than $2$.  Each residue degree $r_i$ that occurs contributes $r_i$ basis elements.  We choose these basis elements for $\cO_p / f \cO_p$ to be $1, y, y^2,\ldots, y^{r_i-1}$, where $f(y)$ is an irreducible polynomial of degree $r_i$ over $\Z_p$.  We get $e_i$ such groups of $r_i$ basis elements for each $r_i$, including $w$ blocks of two basis elements $\{1,y\}$ coming from primes of residue degree $2$, and $v$ basis elements $\{1\}$ corresponding to primes of residue degree $1$.  We choose these basis elements to be orthogonal to each other unless they correspond to the same irreducible polynomial.  

The ordering of the basis elements has a large effect on the form of the lower triangular matrices in $\M(p)$.  We order this basis so that elements corresponding to a single irreducible polynomial are given left to right by increasing powers of $y$.  The $e_i$ sets of $r_i$ columns corresponding  to the primes of residue degree $r_i$ are ordered so that they occur in adjacent blocks.  We order these groups of $e_i$ blocks of $r_i$ columns from left to right by decreasing values of $r_i$, except that we switch the positions of the block of $v$ columns corresponding to primes of residue degree $1$, and the $w$ pairs of columns corresponding to primes of residue degree $2$.  We give an example for a lower triangular matrix corresponding to a prime of type $1^2 2^1 3^1$.  The first three columns correspond to basis elements corresponding to an irreducible cubic, followed by two columns corresponding to linear polynomials, and finally by a pair of columns from an irreducible quadratic.  In the picture below variable names are chosen to emphasize the grouping of columns:
\[
\left(\begin{smallmatrix}
a_{1,1} & 0& 0& 0 & 0 & 0 & 0 \\
a_{2,1}& a_{2,2} &0 & 0& 0  & 0 & 0\\
a_{3,1}& a_{3,2}& a_{3,3} & 0& 0 & 0 & 0\\
a_{4,1}& a_{4,2}& a_{4,3}  & b_{4,4} &0  & 0 & 0 \\
a_{5,1}&a_{5,2} & a_{5,3}  & b_{5,4} & b_{5,5}  & 0 & 0 \\ 
a_{6,1}& a_{6,2}& a_{6,3}  & b_{6,4} &b_{6,5} & c_{6,6} &  0 \\
a_{7,1}& a_{7,2} & a_{7,3}  & b_{7,4} & b_{7,5} & c_{7,6} & c_{7,7} \\ 
\end{smallmatrix} \right).
\]

We now briefly explain how to take the product of two rows of such a matrix.  A row vector corresponds to a linear combination of basis elements.  We can take two vectors, take the product of the corresponding elements in $\cO_p$ and then express the result as a linear combination of our chosen basis.  We denote the product corresponding to rows $v$ and $w$ by $v\circ w$.

We now give the proof of Lemma \ref{impossible} on the non-existence of certain kinds of multiplicatively closed sublattices.
\begin{proof}[Proof of Lemma \ref{impossible}]
Let $R$ be a multiplicatively closed sublattice of $\cO_L$ of index $p$. Then clearly $p \cO_L \subset R$, and consequently 
$$
p \cO_L \subset R \subset \cO_L. 
$$
This means $(R / p \cO_L) \subset \left(\cO_L / p \cO_L\right)$. Now $\cO_L/p\cO_L$ is a field of order $p^n$, and $R/p\cO_L$ is a subring, not necessarily with a multiplicative identity, of $\cO_L/p\cO_L$. It is also clear that $R/p\cO_L$ is multiplicatively closed.  Any multiplicatively closed subset of a finite field does contain the identity element because the multiplicative group of the field is cyclic, so $R/p\cO_L$ is also a field.

Since the index is of $R$ in $\cO_L$ is $p$ the number of elements of $R/p\cO_L$ is $p^{n-1}$. Thus if $\F_{p^k}$ is the finite field with $p^k$ elements we have $\F_{p^{n-1}} \subset \F_{p^n}$. This implies either $n-1=0$ or $n-1$ divides $n$. In the first case we get $n=1$ and in the second case we get $n=2$. Any larger value of $n$ gives a contradiction.
\end{proof} 

\begin{cor}
Let $p$ be a prime of type $f_p = r$ with $r\ge 3$.  Then $a_1(p) = 0$.
\end{cor}

These previous two lemmas allow us to compute $a_1(p)$ by considering a much smaller class of lower triangular matrices.

\begin{proof}[Proof of Lemma \ref{a1preduce}]
We choose the ordered basis of $\cO_p$ described above.  Suppose that column $j$ corresponds to a basis element coming from a prime of residue degree $k>2$.  We claim that the diagonal element of this column must be equal to $1$.

We argue by contradiction.  Suppose that $x_{jj} = p$.  By row-reducing we may suppose that the only nonzero elements of this matrix off the diagonal are in column $j$. Basis elements that do not correspond to the same irreducible polynomial are orthogonal. Suppose that the columns corresponding to the same irreducible polynomial as the basis element of column $j$ are labeled by $c_1, \ldots, c_k$ and let $v_1,\ldots, v_k$ be the rows containing the diagonal entries of these columns. The only nonzero entries of the vector $v_i \circ v_j$ are in positions corresponding to the columns $c_1,\ldots, c_k$.  Therefore, $v_i \circ v_j$ is a linear combination of the rows $v_1,\ldots, v_k$.  Taking the span of these rows and projecting onto the coordinates corresponding to the columns $c_1,\ldots, c_k$ gives a multiplicatively closed sublattice of a ring corresponding to a degree $k$ extension of $\Q_p$, which is impossible by the argument of Lemma \ref{impossible}.

Therefore every column corresponding to a basis element coming from a prime of residue degree greater than $2$ has its diagonal entry equal to $1$ and does not contribute to $a_1(p)$.
\end{proof}

\begin{proof}[Proof of Lemma \ref{a1pairs}]
A subring of $\cO_p$ of index $p$ is generated by a lower triangular matrix $M$ with exactly one diagonal element equal to $p$ and all others equal to $0$.  We choose the basis of $\cO_p$ so that columns occur in pairs with each pair corresponding to two basis elements $\{1,y\}$ of $\cO_p / f\cO_p$ where $f(y)$ is an irreducible quadratic polynomial over $\F_p$ and the column corresponding to $1$ occurs first.  When $p\neq 2$ we can choose $f(y) = y^2-b$ with $b$ a positive integer which is not a square modulo $p$.  We focus on this case but note that for $p=2$ we can take $f(y)= y^2+y+1$ and the rest of the argument is similar.  Basis elements occurring in distinct pairs are orthogonal to each other.

We will first show that it is not possible that the column with diagonal entry $p$ corresponds to a basis element $1$ for some quadratic polynomial.  Suppose that it is and let the row which contains this diagonal element be $v_1$.  Let $v_2$ be the row which has diagonal element in the column corresponding to the basis element $y$ for the same polynomial.  Suppose the entry in row $v_2$ in the column with diagonal entry $p$ is $a \in \Z_p$.

We will now give a first example of an argument that will be important throughout the rest of this section.  Suppose $M$ spans a sublattice of index $p$ and has diagonal entries equal to $1$ except for a single column in which the corresponding entry is $p$.  We note that all vectors in the lattice spanned by $M$ that are zero except in this entry must lie in $p\Z_p$ since otherwise we could row reduce $M$ and see that the index of this lattice is actually $1$.  We will use this fact to show that certain columns cannot have the single diagonal entry equal to $p$.

We see that $v_2 \circ v_2$ has two nonzero entries: $2a$ in the column corresponding to $y$ and a $b + a^2$ corresponding to $1$, since $y^2$ is $b$ modulo $f(y)$.  Since $M$ generates a multiplicatively closed sublattice, and all other entries in the column with diagonal entry in the row $v_2$ are $0$, and so $v_2 \circ v_2 - 2a v_2$ must be in the row span of $v_1$. So there must exist some $\alpha_1\in \Z_p$ such that 
\[p \alpha_1 = b+a^2 -2 a^2 = b- a^2.\]
This implies that $b-a^2 \in p\Z_p$, contradicting the fact that $b$ is a nonsquare modulo $p$. Therefore we may suppose that for each column corresponding to $1$ for a quadratic polynomial, the diagonal entry is $1$.

There are $w$ columns which correspond to basis elements $y$ for distinct irreducible quadratic polynomials.  We will show that if the diagonal element of such a column is equal to $p$ then all other entries of this column are in $p\Z_p$.  Applying elementary row operations together with Lemma \ref{a1sum} completes the proof.

We suppose that row $v_1$ has its diagonal entry equal to $p$ and that this column corresponds to a basis element $y$ for some irreducible quadratic polynomial.  Let $v_2$ denote the row with diagonal entry corresponding to the basis element $1$ for the same quadratic polynomial.  Note that $v_2$ is above $v_1$ in this matrix and has a single nonzero entry equal to $1$.  We will show that it is not possible for there to be a row $u$ with an entry that is a unit in the column with diagonal entry $p$.  

Suppose that there is such a row with an entry $a \in U_p$ in this column and consider $u \circ v_1$.  This has a single nonzero entry equal to $a$ in the column corresponding to the diagonal entry $p$.  The argument above shows that such a matrix actually generates $\cO_p$ and not a subring of index $p$, which is a contradiction.  We have shown that there are no units in the column with diagonal entry $p$, completing the proof.

\end{proof}

\begin{proof}[Proof of Lemma \ref{a1pvw}]
We continue with the notation of the previous proof.  Again we consider $p \neq 2$ and note that when $p =2$ we choose $f(y) = y^2+y+1$ for our irreducible quadratic polynomials and the argument is very similar.

We choose the basis elements of $\cO_p$ so that the first $v$ columns correspond to primes of residue degree $1$ and the last $2w$ columns occur in pairs and correspond to primes of residue degree $2$.  The proof of the previous lemma shows that matrices with diagonal entry equal to $p$ in a column corresponding to a prime of residue degree $2$ contribute $w$ to $a_1(p)$.  We now focus only on the entries of the columns of this matrix which correspond to primes of residue degree $1$.

Suppose $x_{jj} = p$ and that this column corresponds to a prime of residue degree $1$.  Since $L$ is a subring and not just a multiplicative sublattice, it must contain the identity element of $\cO_p$ and we see that there must be some entry in this column that is a unit.  In fact, we will show that there must be a unique entry in this column that is a unit.  Each of the $v-j$ rows directly below this diagonal entry can contain any unit in $1+p\Z_p$, but no other units can occur.  Applying Lemma \ref{a1sum} shows that $a_1(p) = w + \sum_{j=1}^v (v-j) = w+ \binom{v}{2}$, completing the proof.

We first note that we cannot have two units in rows corresponding to primes of degree $1$ in the column with diagonal entry equal to $p$.  If we did, taking $v_1 \circ v_2$ for these two rows would give a vector with a single nonzero entry which is a unit in the column with diagonal entry $p$.  This is a contradiction.

Suppose there is a row with diagonal entry corresponding to an irreducible quadratic polynomial which has a unit entry in the column with diagonal entry $p$.  Let $v_1$ be the row corresponding to the basis element $1$ for this polynomial and $v_2$ be the row corresponding to the basis element $y$.  Suppose the entry in the column with diagonal entry $p$ is $a$ in row $v_1$ and $c$ in row $v_2$.  By assumption, at least one of $a,c$ is a unit.  We show that this is a contradiction.  

 We see that $v_1 \circ v_1 - v_1$ has an entry of $a^2-a$ in the column with diagonal entry $p$ and every other entry of this vector is zero.  So either $a \in p\Z_p$ or $a \in 1 + p\Z_p$.  We see that $v_2 \circ v_2 - b v_1$ has an entry $c^2 - ab$ in the column with diagonal entry $p$ and every other entry is zero.  If $a \in 1 + p\Z_p$ then since $b$ is not a square modulo $p$, we get a contradiction.  If $a \in p\Z_p$ then we have $c^2 \in p\Z_p$, which is also a contradiction.

\end{proof}

Combining Lemma \ref{a1preduce} and Lemma \ref{a1pvw} completes the proof of Proposition \ref{prop:1}.

\subsection{Proof of Proposition \ref{prop:2}}\label{proof:prop2}

To fix notation we give a quick review of basic class field theory \cite{Ne}. Let $K$ be a number field, and let $\cI_K$ be the free group generated by the finite primes of $K$. 
There is a natural map $\iota: K^\times \to \cI_K$. A modulus, called a cycle in \cite{Ne}, is a finite formal product of primes of $K$ with non-negative exponents $\prod_\p \p^{n_\p}$. 
If $\m = \prod_\p \p^{n_\p}$ is a modulus, and $x \in K$, we write $x \equiv 1 \mod \m$ to mean:
\begin{itemize}
\item For each finite $\p | \m$, $x \equiv 1 \mod \p^{n_\p}$; 
\item for each real prime $\nu | \m$, we have $x_v >0$. 
\end{itemize}
If $S$ is a finite set of primes, we let $\cI_K^S$ be the subgroup of $\cI$ generated 
by the primes not in $S$. For a modulus $\m$ we let $J_K^\m$ be $J_K^S$ where 
$S$ is the set of finite primes that divide $\m$. Set 
$$
K^\m := \iota^{-1}(\cI_K^\m)
$$
and 
$$
K_1^\m : = \{ x\in K^\m; x \equiv 1 \mod \m\}. 
$$
Let $P_K^\m = \iota(K_1^\m)$ and define 
$$
\CL_K^\m =\cI_K^\m / P_K^\m. 
$$
This class group is finite. A congruence subgroup modulo $\m$ is a subgroup $H^\m$ of 
$\cI^\m_K$ which contains $P_K^\m$. We recall the following two main theorems of class field theory: 

\begin{theorem}[Artin Reciprocity Law]
For $L/K$ an Abelian extension of number fields, there is a modulus $\m$ divisible by all the ramified primes of $L/K$ such that 
the sequence 
$$
1 \to P_K^\m .N_{L/K}(\cI_L^\m) \hookrightarrow \cI_K^\m \to \Gal(L/K) \to 1 
$$
is exact. 
\end{theorem}

\begin{theorem}
 For any congruence subgroup $H^\m$, there is a unique Abelian extension $L/K$ such
 that $L$ is the class field of $K$ of the congruence class group $\cI_K^\m/H^\m$. 
\end{theorem}

We have the following lemma:  

\begin{lemma}
Let $K$ be a number field, $\m$ a modulus, and $H^\m$ a congruence subgroup. If $C$ is a
coset of $\cI_K^\m / H^\m$, we set 
$$
f_C(s) = \prod_{p \in C} (1-N(p)^{-s})^{-1}.
$$
Then $f_C(s)$ is holomorphic for $\Re s > 1$. Furthermore, 
Then $g_C(s) = f_C(s)^r$, $r = |\cI_K^\m/ H^\m|$, has an analytic continuation to an open set 
containing $\Re s =1$ with a unique pole at $s=1$. Assuming GRH, $s=1$ is the only pole for $\Re s>1/2$. 
\end{lemma} 
We do not need the additional convergence provided but assuming GRH to prove Proposition \ref{prop:2}, but include this statement to give a better idea of the analytic behavior of this function.
\begin{proof}
Let $G= \cI_K^\m / H^\m$. Then 
$$
\log g_C(s) = |G| \log f_C(s)  
$$
$$
= - |G| \sum_{p \in C} \log (1-N(p)^{-s})
$$
$$
= |G| \sum_{p \in C} N(p)^{-s} + |G| \sum_{p \in C} \sum_{m \geq 2} \frac{1}{m} N(p)^{-ms}. 
$$
Write 
$$
h(s) = |G| \sum_{p \in C} \sum_{m \geq 2} \frac{1}{m} N(p)^{-ms}. 
$$
This is holomorphic for $\Re s > 1/2$. We then write 
$$
\log g(s) - h(s) =   \sum_{p} \sum_{\chi \in Hom (G, S^1)} \chi(p) \chi(C^{-1}) N(p)^{-s} 
$$
$$
= \sum_{\chi \in Hom(G, S^1)} \chi(C^{-1}) \sum_{p} \chi(p) N(p)^{-s} 
$$
$$
= \sum_{\chi \in Hom(G, S^1)} \chi(C^{-1}) \left( \log \prod_p (1 - \chi(p) N(p)^{-s}) - \sum_p \sum_{m \geq 2}  \frac{1}{m} \chi(p)^m N(p)^{-ms} \right)
$$
$$
=\log \left(\prod_{\chi \in Hom(G, S^1)} L(s, \chi)^{\chi(C^{-1})}\right) + H(s) 
$$
with $H(s)$ a function that is holomorphic for $\Re s > 1/2$. Hence 
$$
g_C(s) = \prod_{\chi \in Hom(G, S^1)} L(s, \chi)^{\chi(C^{-1})} e^{H(s) + h(s)} . 
$$
The lemma now follows from results on zero free regions of $L$-functions, e.g. Ch. 2 of \cite{Mu-Mu}.

\end{proof} 

Next we can prove Proposition \ref{prop:2}:

\begin{proof}[Proof of Proposition \ref{prop:2}]
If $L/K$ is Abelian, this follows from the above lemma and class field theory. In general, let $\sigma \in C$, and let $H = \langle \sigma \rangle$. Let $M= L^H$.  Note that 
$L/M$ is an Abelian Galois extension.  Let 
$$
F_H(s) = \prod_{p \in S} (1- N_M(p)^{-s})^{-1} 
$$
where  $S$ is the set of primes of $L^H$ satisfying 
\begin{itemize}
\item $\left(\frac{L/M}{p}\right) = \sigma$; 
\item $f(p/p \cap \cO_K) = e(p / p \cap \cO_K)=1$. 
\end{itemize}
We will also consider 
$$
F'_H(s) = \prod_{p \in S'} (1- N_M(p)^{-s})^{-1} 
$$
where $S'$ is the set of primes $p$ of $M$ such that $\left(\frac{L/M}{p}\right) = \sigma$.  We know from 
what we proved before that $F'_H(s)^{|H|}$ has a simple pole at $s=1$. By the computations of Ch. V, \S 6 of \cite{Ne} we know that $F'_H(s) / F_H(s)$ is holomorphic for $\Re s > 1/2$. Thus $F_H(s)^{|H|}$ has a simple pole at $s=1$ and otherwise holomorphic in an open set containing  $\Re s \geq 1$. 

\

Next, it follows from the reduction step of the proof of the Chebotarev density theorem, Theorem 6.4 of \cite{Ne}, that 
$$
F_H(s) = \left( \prod_{p \text{ prime of } K \atop \left(\frac{L/K}{p}\right)=C} 
(1-N(p)^{-s})^{-1}\right)^{\frac{|G|}{|C|\cdot |H|}}
$$
$$
=(F_C(s))^{\frac{|G|}{|C| \cdot |H|}}. 
$$
The proposition is now immediate. 
\end{proof}

\subsection{Some remarks on $r_2$}\label{remarks:r2}

Suppose we have a finite group $G$ acting on a finite set $A$. Let $O_1, \dots, O_r$ be the distinct orbits of the action of $G$. Then $G$ has an induced representation on the vector space 
$$
V = \oplus_{a \in A} \C. 
$$
We skip the proof of the following elementary lemma: 
\begin{lemma} We have 
$$
\dim V^G = r. 
$$
\end{lemma}

The lemma has the following consequence: 

\begin{proposition}
We have 
\begin{enumerate}
\item for $n\geq 3$, $r_2(S_n) = r_2(A_n) = 1$; 
\item $r_2(C_n) = r_2(D_n) = \lfloor n/2 \rfloor$. 
\end{enumerate}
\end{proposition}
\begin{proof}
For the first part we show that $A_n$ acts transitively on the two element subsets of $\{1, \dots, n\}$. For this we notice that for three distinct elements $a, b, c$, the even permutation $(a \,\, c)(b \,\, a)$ maps the set $\{a, b\}$ to the set $\{b, c\}$.  

\

For $C_n$ and $D_n$, write  $n=2k$ or $n=2k+1$, depending on the parity of $n$. Suppose $C_n = \langle (1 \,\, 2 \,\, \dots \,\, n) \rangle$.  It is easy to see that for each $1 \leq i \leq k$, the set
$$
O_i =\{ \{a , b\}; 1 \leq a,  b \leq n, b -a \equiv i \mod n\}
$$
is an orbit of the action of $C_n$ on the set of two element subsets of $\{1, \dots, n\}$. Furthermore, these are all the possible orbits. To see the result for $D_n$, we consider the generators $(1 \,\, 2 \,\, \dots \,\, n), \sigma$, with 
$$\sigma = (1 \,\, n) (2 \,\, n-1) \dots (k \,\, k+1).$$ 
We observe that each orbit $O_i$ is invariant under the action of $\sigma$. 
\end{proof}

For the case where $n$ is a prime number, we have the following proposition: 

\begin{proposition}
Let $G$ be a transitive subgroup of $S_p$, $p$ prime. Then one of the following two possibilities occurs:
\begin{enumerate}
\item $G$ is doubly transitive and $r_2(G) =1$;
\item $G$ is solvable in which case $p \, | \, |G|$ and $r_2(G) = \gcd\left(\frac{|G|}{p}
, \frac{p-1}{2}\right)$. 
\end{enumerate}
\end{proposition}

\begin{proof}
A theorem of Burnside \cite{Burnside, Muller} says that a transitive subgroup of $S_p$ is either doubly transitive or solvable.  If the action of $G$ is doubly transitive, then $r_2(G) =1$. If $G$ is solvable, a classical theorem of Galois (\cite{Hu}, p. 163)\footnote{We learned Galois' theorem from a question posted by Chandan Singh Dalawat on {\tt mathoverflow}, and comments by Matt Emerton, Jack Chapman, and Jack Schmidt.} asserts that $G$ contains a unique normal subgroup $C$ of order $p$, and is contained in the normalizer of $C$. Furthermore, $G/C$ is a cyclic group of order dividing $p-1$. Up to conjugation we may assume that $C =\langle (1 \,\, 2 \,\, \dots \,\, p) \rangle$.  The normalizer of $C$ is the split extension of the group $C$ by the cyclic group $Z$ of order $p-1$ consisting of the elements $\sigma_k$, $1 \leq k \leq p-1$ identified by 
$$
\sigma_k(x) \equiv k x \mod p,  
$$
for $x \in \{1 , \dots, n\}$; that the group $Z$ is cyclic is the theorem of the primitive root in elementary number theory. Let $\sigma_g$ be a generator of $Z$. Since $G$ is transitive,  $G$ is equal to $C \ltimes \langle \sigma_g^j \rangle$ for some $ j | p-1$. By the description of orbits of $C$ on the two element subsets of $\{1, \dots, p\}$, we just need to know the number of orbits of $\langle \sigma_g^j, \sigma_g^{\frac{p-1}{2}}\rangle$ on $(\Z/p\Z)^*$. The latter is equal to 
$$
\frac{|(\Z/p\Z)^*|}{|\langle \sigma_g^j, \sigma_g^{\frac{p-1}{2}}\rangle|} = \frac{p-1}{|\langle \sigma_g^{\gcd(j,  \frac{p-1}{2})}\rangle|} = \gcd\left(j , \frac{p-1}{2}\right). 
$$
\end{proof}

\section{The proof of Theorem \ref{mainthm:1}}\label{proof}

\subsection{Outline of the proof of Theorem \ref{mainthm:1}}\label{outline}
Let $d \in \N$, and let $R=\Z^d$ equipped with componentwise addition and multiplication. Namely for $v = (v_1,\dots, v_d), w = (w_1, \dots, w_d) \in \Z^n$, 
we set
$$
v+w = (v_1+w_1, \dots, v_d+w_d),
$$
$$
\beta(v, w):=v \circ w = (v_1 w_1, \dots, v_d w_d).
$$

To emphasize the dependence of $\M_p(\beta)$ from Definition \ref{Mp} on $d$, we write it as $\M_d(p)$.  For $d=2,3,4$, we will give an explicit description of $\M_d(p)$
in Sections \ref{n=3}, \ref{volume estimates for n=4} and
\ref{volume estimates for n=5}.

\
\begin{defn}\label{muk}
 If $\underline{k}=(k_1, \dots, k_d)$ is a $d$-tuple of
 non-negative integers, we set
 $$
\M_d(p; \underline{k}) = \left\{ M=\begin{pmatrix}
p^{k_1} & 0  & \ldots  & 0\\
x_{21} & p^{k_2} & 0 & \vdots\\
\vdots & \vdots & \ddots & 0\\
x_{d1} & \ldots  & x_{d \, d-1} & p^{k_d}
\end{pmatrix} \in \M_d(p)\right\}.
 $$
We define $\mu_p(\underline{k})$ to be the
$\frac{d(d-1)}{2}$-dimensional volume of $\M_d(p; \underline{k})$.
\end{defn}

It is easy to see that
\begin{equation}\label{z-mu}
\zeta_{\Z^d, p}^<(s)= \sum_{\underline{k}=(k_1, \dots, k_d) \atop k_i
\geq 0, \forall i} p^{\sum_{i=1}^d (d-i)k_i} p^{-s \sum_{i=1}^d
k_i} \mu_p(\underline{k}).
\end{equation}

Intuitively what this means is that we have multiplied the rows by
units to make the diagonal entries a $p$-power. We note that this
does not change the lattice generated by the rows.

\

\noindent {\bf Warning.} The volume of $\M_d(p; \underline{k})$ are used to count  {\em subrings} of finite index in $\Z^d$, and {\em orders} of finite index in $\Z^{d+1}$. 
The reader should be careful about the distinction between subrings and orders. 

\

We have the following lemma which is equivalent to Lemma \ref{LemmaLiu} given during the proof of Proposition \ref{prop:1}.
\begin{lemma}\label{n+1 choose 2} We have
$$ a^<_{\Z^d}(p) = {d+1 \choose 2}. $$
\end{lemma}
For a proof see \cite{Liu} Proposition 1.1. The quantity $a^<_{\Z^d}(p)$ is equal to $f_{d+1}(p)$ of that reference.  By Theorem \ref{tauberian application}, Theorem \ref{mainthm:1} is proved if we can show the following statement: 
there is an $\epsilon >0$ such that for $\Re (s) =
\sigma > 1- \epsilon$ we have
$$
\sum_{p} \sum_{k=2}^\infty \frac{a^<_{\Z^d}(p^k)}{p^{k \sigma}}<\infty.
$$
Since by Equation \eqref{z-mu}
$$
a^<_{\Z^d}(p^k)= \sum_{\underline{k}=(k_1, \dots, k_d) \atop \sum_i k_i
= k} p^{\sum_i (d-i)k_i} \mu_p(\underline{k}),
$$
in order to prove the lemma we need to estimate
$\mu_p(\underline{k})$. The relevant computations are performed in Sections \ref{n=3}, \ref{volume estimates for n=4}, and \ref{volume estimates for n=5}. 

\

 The results are stated in Theorems \ref{error-3},
\ref{error-4}, and \ref{error-5}. These theorems form part 1 of Theorem \ref{mainthm:1}. 

\

The proof of part 2 of Theorem \ref{mainthm:1} appears in \S \ref{general n}.

\subsection{General facts about volumes}\label{general facts}

We begin with some lemmas that allow us to bound the volumes of certain sets that arise in our volume computations.  Let $U_p$ denote the set of units of
$\Z_p$ and $v_p(\cdot)$ be the $p$-adic valuation.  Recall that for $\alpha, \beta \in \Z_p$, if $v_p(\alpha) \neq v_p(\beta)$ then $v_p(\alpha - \beta) = \min\{v_p(\alpha),
v_p(\beta)\}$. 

\begin{proposition}\label{xy-z}
For fixed $y,z\in \Z_p,\ k\ge 0$, the volume of $x\in \Z_p$ such
that $v_p(x y - z)\ge k$ is at most $p^{-(k-v_p(y))}$.
\end{proposition}
\begin{proof}
We first note that for $y = 1$, the volume of $x$ such that $v_p(x
- z) \ge k$ is $p^{-k}$, since we are just fixing the first $k$
digits in the $p$-adic expansion of $x$ to coincide with those of
$z$.  Similarly, for any unit $u\in U_p$ the volume
of $x$ such that $v_p(ux - z) \ge k$ is $p^{-k}$.

We see that if $v_p(z) < k$ and $v_p(y) > v_p(z)$,
then clearly $v_p(xy-z) = v_p(z) < k$ for any value of $x$.  If
$v_p(z) \ge k$, then $v_p(xy-z)\ge k$ if and only if $v_p(xy) \ge
k$ which holds if and only if $v_p(x) \ge k-v_p(y)$.  This holds
on a set of volume at most $p^{-(k-v_p(y))}$ if $k\ge v_p(y)$ and
on a set of volume $1$ if $v_p(y) \ge k$.

Now if $v_p(z) < k$ and $v_p(y) \le v_p(z)$ then we can write $y =
p^{v_p(y)} u$ for some unique unit $u\in U_p$, and $z = p^{v_p(y)}
z'$ for some unique $z' \in \Z_p$.  We have $v_p(xy-z) \ge k$ if
and only if $v_p(xu - z') \ge k-v_p(y)$, which holds on a set of
volume at most $p^{-(k-v_p(y))}$.
\end{proof}

\begin{proposition}\label{k+1}
For fixed $z \in \Z_p$, the combined volume of $x, y \in \Z_p^2$
such that $v_p(xy-z) \ge k$ is at most $(k+1) p^{-k}$.
\end{proposition}

\begin{proof}
If $v_p(y) \geq k$, then there are two cases. Either $v_p(z) \geq
k$ in which case any $x$ will work, or $v_p(z) < k$ in which case
no $x$ works. So assume $0 \leq v_p(y) < k$. Then given $y$ with $
l = v_p(y)$ we need $x$ such that $x \in p^{-l}(p^k \Z_p + z)$. So
the total volume is
$$
\sum_{l=0}^{k-1} p^{-l} \vol (p^{-l}(p^k \Z_p + z)) \le kp^{-k}.
$$
\end{proof}

\begin{proposition}\label{k+1 xy-z}
For any fixed $z\in \Z_p$, the combined volume of $x, y \in
\Z_p^2$ such that $v_p(x(y-z)) \ge k$ is at most $(k+1) p^{-k}$.
\end{proposition}

\begin{proof}
This proposition is very similar to the previous one.  We have
$v_p(x) \ge k$ on a set of volume $p^{-k}$.  Suppose that this
does not hold and set $v_p(x) = m$.  We see that for any fixed $z$
the volume of $y$ such that $v_p(y-z) \ge k-m$ is $p^{-(k-m)}$.
Summing over the $k$ possible values of $m$ gives the result.
\end{proof}

\begin{proposition}\label{zk2}
  Suppose $z \in \Z_p$, $k, l \geq 0$ are given. Then the volume of $x \in \Z_p$ such that
  $$
  v_p(x(x-p^l) - z) \geq k
  $$
  is bounded by $2 p^{-\lceil k/2\rceil}$.
\end{proposition}
\begin{proof}
If there is no such $x$ then the volume is zero and there is
nothing to prove. Assume that the volume is nonzero. For simplicity of notation, let
$y = p^l$. If $v_p(t) \ge k$ and $v_p(x(x-y) - z) \geq k$, then $x+t$ also satisfies
the same inequality.

Given $y$ and $z$ modulo $p^k$, we must determine the number of $x$ modulo $p^k$ such that
$x(x-y) - z \equiv 0 \mod p^k$. If this number is $N$, the volume
of our domain is $N\cdot p^{-k}$. Suppose $X, X+u$ are
both solutions of the congruence
$$
x(x-y) \equiv z \mod p^k.
$$
This implies that $u$ satisfies the congruence
$$
u^2 + u(2X - y) \equiv 0 \mod p^k.
$$
We count the number of nonzero solutions $u$ of this congruence equation.

If $2X -y \equiv 0 \mod p^k$, then $u^2 \equiv 0 \mod p^k$.  This implies any solution $u$ is of the form
$$
a_{\lceil\frac{k}{2}\rceil} p^{\lceil\frac{k}{2}\rceil} + a_{r+1} p^{r+1} + \dots + a_{k-1} p^{k-1}.
$$
There are at most  $p^{k - \lceil{k/2}\rceil}$ choices for $u$.  If not, then we write $2X - y \equiv p^s q \mod p^k$ with $s<k$ and $(q,p) =1$.

We write $u = p^r m \mod p^k$.  By assumption, $(m,p) = 1$ and $r<k$.  Since 
\begin{equation}\label{usquared}
u\left(u + (2X-y)\right) \equiv 0 \mod p^k,
\end{equation}
we have $u + (2X-y) \equiv 0 \mod p^{k-r}$.  

If $2r \ge k$, then $r \ge \lceil \frac{k}{2} \rceil$, and as above there are at most  $p^{k - \lceil{k/2}\rceil}$ choices for $u$.

If $2r < k$, then $s=r$ and Equation \ref{usquared} implies that $u$ and $2X-y$ match up in the first $k-r \ge \lceil \frac{k}{2} \rceil$ digits of their $p$-adic expansions.  This gives at most $p^{k- \lceil \frac{k}{2} \rceil} \le p^{\lceil \frac{k}{2} \rceil}$ choices for $u$.  Multiplication by $p^{-k}$ gives the result.
\end{proof}

We point out that in the most general possible case it is not
possible to improve this result by more than  a factor of $2$.
Suppose $l \ge \lceil k/2\rceil$.  Then $v_p(x) + v_p(x-p^l) \ge
k$ if and only if $v_p(x) \ge \lceil k/2 \rceil$, which holds on a
set of volume at most $p^{-\lceil k/2\rceil}$.  However, in some
cases we can say something stronger.

\begin{proposition}\label{k-l z}
  Suppose $z \in \Z_p$, $k, l \geq 0$ are given. Then there is a constant $C$, which for odd $p$ may be taken to be
  $6$, such that the volume of $x \in \Z_p$ satisfying
  $$
  v_p(x(x-p^l) - z) \geq k
  $$
  is bounded by $C p^{-(k-l)}$ except when $p=2$ and $v_2(z) = 2l-2 <k$. In this exceptional
  situation:
  \begin{enumerate}
\item If $v_2(z+ 2^{2l-2}) \geq k$, the volume is bounded by
$2^{-\lceil k/2 \rceil}$, and this is the best bound possible.\\
\item If $v_2(z+ 2^{2l-2}) < k$ is odd, the volume is zero.\\
\item If $v_2(z+ 2^{2l-2}) < k$, the volume is bounded by
$$
8\left|z+ 2^{2l-2}\right|_2^{-1/2} 2^{-k}, 
$$  
where $|\, . \, |_2$ is the $2$-adic absolute value on $\Q_2$. 
\end{enumerate}
\end{proposition}
 \begin{proof} The proposition will have no content unless $l <
 k$. First we consider the case where $p$ is odd.
  We recognize two basic cases: \\

1. If $v_p(z) \geq k$, then we have $v_p(x(x-p^l)) \geq k$.  We consider two cases, when $v_p(x) = l$ and when $v_p(x) \ne l$. In the first case $v_p(x-p^l) \geq k-l$, and in the second case we have $v_p(x) \geq k-l$. In either case the volume is bounded by $p^{-(k-l)}$.  \\

2. If $v_p(z) < k$, then our inequality can be valid only when
$v_p(x(x-p^l)) = v_p(z)$. Since $v_p(z) < k$, we write $z = \zeta
p^u$ with $u <k$. We are looking for solutions to
  $$
  v_p(x(x-p^l)- \zeta p^u) \geq k
  $$
 that satisfy $v_p(x) + v_p(x-p^l) = u$. 
  \begin{itemize}
  \item If $v_p(x) > l$, then we must have $v_p(x) + l = u$, and as a result $u-l > l$ which means $u > 2l$. Write
  $x = \epsilon p^{u-l}$. Then we need
  \[
  v_p(\epsilon p^{u-l}(\epsilon p^{u-l}-p^l) - \zeta p^u) \geq k.
  \] 
  This implies
  $v_p(\epsilon (\epsilon p^{u-2l}-1)-\zeta) \geq k-u$. This is a
  quadratic equation in $\epsilon$ with at most two solutions
  modulo $p$.   Hensel's lemma says that the volume of $\epsilon$ satisfying
   this last inequality is at most $2 p^{-(k-u)}$. The volume for $x$
   is then at most $2p^{-(u-l)} \cdot p^{-(k-u)} = 2p^{-(k-l)}$.

  \item (*) If $v_p(x) < l$, then $2 v_p(x) = u$,
  which means $u$ is even and $u < 2l$. Write
  $x = \epsilon p^{u/2}$. Then we need $v_p(\epsilon p^{u/2}
  (\epsilon p^{u/2} - p^l) - \zeta p^u) \geq k$ which gives
  $v_p(\epsilon(\epsilon - p^{l-u/2})- \zeta) \geq k-u$.
  By Hensel's lemma the volume of such $\epsilon$ is at most
  $2 p^{-(k-u)}$. The volume of $x$ is then bounded by
  $2p^{-(k-u)}\cdot p^{-u/2} = 2p^{-k + u/2} < 2p^{-k + l}$ which is what we want.

  \item If $v_p(x) = l$, then $x = \epsilon p^l$, and we have
  $
  2 l + v_p(\epsilon -1) = u.
  $
  This means $u \geq 2l$. Then we need $v_p(\epsilon(\epsilon -1) - \zeta p^{u-2l}) \geq k-2l$. An application
  of Hensel's lemma then says that the volume of $\epsilon$ satisfying this inequality is at most $2 p^{-(k-2l)}$. Since $x = p^l \epsilon$, the volume of $x$ is at most $2p^{-(k-l)}$.
\end{itemize}

Now we examine the situation for $p=2$. Except for the step marked
(*) every other step of the proof works verbatim. The argument (*)
can be adjusted as follows. We let $r= l - \frac{u}{2}$ and $s =
k-u$. Then $ r\geq 1$ and we are trying to determine the volume of
$\epsilon \in U_p$ such that
$$
v_2(\epsilon(\epsilon - 2^r) - \zeta) \geq s.
$$
for a given unit $\zeta$. Rewrite this inequality as
$$
v_2((\epsilon - 2^{r-1})^2 - (\zeta + 2^{2r-2}) ) \geq s.
$$
First we consider the situation for $r \geq 2$. In this case both
$\epsilon - 2^{r-1}$ and $\zeta + 2^{2r-2}$ are still units, and
without loss of generality we may assume that our inequality has
the form
$$
v_2(\epsilon^2 - \zeta) \geq s
$$
with $\epsilon, \zeta$ units. Fix an $\epsilon$ that satisfies the
inequality, and we determine for what values of $\tau$, $\epsilon
+ \tau$ also satisfies the inequality. The volume of such $\tau$
is the volume of $\epsilon$. We have
$$
v_2((\epsilon + \tau)^2 - \zeta) = v_2((\epsilon^2-\zeta) +
\tau(\tau + 2 \epsilon)).
$$
This implies that
$$
v_2(\tau(\tau + 2 \epsilon)) \geq s.
$$
This immediately implies that $v_2(\tau) \geq s-1$ or $v_2(\tau +
2\epsilon) \geq s-1$. Consequently the volume of $\epsilon$ is
bounded by $2 \cdot 2^{-(s-1)} = 4 \cdot 2^{-(k-u)}$. The rest of the
argument works as before.

\

 Now we consider the case where $r=1$. In
this case the inequality becomes
$$
v_2((\epsilon -1)^2 - (\zeta +1)) \geq s.
$$
There are two cases to consider: 

\

\noindent \emph{Case I.} $v_2(\zeta + 1) \geq s$. In this case we see that $v_2(\epsilon -1) \geq \lceil s/2 \rceil$ and
as a result the volume is $2^{-\lceil s/2 \rceil}$. The volume of $x$ is then seen to be bounded by $2^{-\lceil k/2\rceil }$. 

\

\noindent \emph{Case II.} $v_2(\zeta+1) < s$. We have $2 v_2(\epsilon -1) = v_2(\zeta +1)$, so we can write $\zeta+1 = \gamma 2^{2t}$, with $\gamma$ a unit. Then we have $v_2(\epsilon -1)
= t$, and write $\epsilon -1 = \omega 2^t$. This implies
$$
v_2(\omega^2 - \gamma) \geq s-2t.
$$
As above, the volume of such $\omega$ is bounded by
$4 \cdot 2^{-s+2t}$. The volume of $\epsilon$ then is bounded by
$4 \cdot 2^{-s+t}$. The volume of $x$ is then bounded by
$4 \cdot 2^{-k+l} \cdot 2^t$.
\end{proof}

\subsection{Orders of $\Z^3$}\label{n=3}
\subsubsection{Volume estimates} 
First we give a description of $\M_2(p)$.
\begin{lemma}\label{mp2}
The set $\M_2(p)$ is the collection of matrices
\[M=\left(
\begin{array}{cc}
x_{11} & 0 \\
x_{21} & x_{22}\\
\end{array} \right),\]
with entries in $\Z_p$ such that
$$
v_p(x_{21}(x_{21}-x_{22})) \geq v_p(x_{11}).
$$
\end{lemma}
\begin{proof}  Let $v_1$ and $v_2$ be the first and the second rows
of $M$ respectively. Then since entries are in $\Z_p$ it is clear
that $v_1 \circ v_1$ and $v_1 \circ v_2$ are integral linear
combinations of $v_1, v_2$. Now we need $v_2\circ
v_2=\alpha_1 v_1+\alpha_2 v_2$ with $\alpha_1, \alpha_2 \in \Z_p$.
So $x_{22}^2=\alpha_2 x_{22}$, which implies $\alpha_2=x_{22}$.
Then $\alpha_1 x_{11} + x_{22} x_{21} = x_{21}^2$, and $\alpha_1 =
x_{11}^{-1} (x_{21}^2-x_{21}x_{22})$. Therefore $\alpha_1$ is in
$\Z_p$ if and only if $v_p(x_{11})\le v_p(x_{21}^2-x_{21}x_{22})$.
\end{proof}

We note that the sublattice corresponding to a matrix $M$ as above has finite index if and only if $\det M \ne 0$. 

\subsubsection{Orders} 
We now prove the following theorem: 

\begin{theorem}\label{error-3}
There is a polynomial $P_3$ of degree $2$ such that for all
$\epsilon >0$
$$
N_3(B) = B P_3(\log B) + O(B^{\frac{1}{2}+ \epsilon})
$$
as $B \to\infty$. \end{theorem}
\begin{proof} By Theorem \ref{tauberian application} and Lemma \ref{n+1 choose 2}, it suffices to prove the following statement:  If $\sigma > \frac{1}{2}$ the series
\begin{equation}\label{expression-3}
\sum_p \sum_{k+l \geq 2} p^k p^{-k\sigma - l\sigma}\mu_p(k,l)
\end{equation} 
converges. Here $\mu_p(k, l)$ is as in Definition \ref{muk}. 

\

We divide the series \eqref{expression-3} into three subseries: \\

\emph{Case I}. $k \geq 0, l \geq 2$. Then by Proposition \ref{zk2}
$$
\mu_p(k,l) \leq 2 p^{-k/2}.
$$
Our subseries is then majorized by
$$
\sum_p \sum_{k \geq 0} \sum_{l \geq 2} p^{k/2} p^{-k\sigma -
l\sigma}
$$
which converges for $\sigma > \frac{1}{2}$. \\

\emph{Case II.} $k \geq 2, l=0$. Then by the proof of Proposition \ref{zk2}
$$
\mu_p(k,0) \leq 2 p^{-k}
$$
and as a result our subseries is majorized by
$$
\sum_p \sum_{k \geq 2}  p^{-k\sigma}
$$
which converges for $\sigma > \frac{1}{2}$. \\

\emph{Case III.} $k=1,l=1$. By Proposition \ref{zk2}
$$
\mu_p(1,1) \leq 2 p^{-1}
$$
and our subseries is majorized by
$$
\sum_p p^{-2 \sigma}.
$$
This converges for $\sigma > \frac{1}{2}$.

\

For the second assertion in the statement of the theorem we observe that
$$
f_3(k) = N_3(k)- N_3(k-1).
$$
\end{proof}

\subsection{Orders of $\Z^4$}\label{volume estimates for n=4}
\subsubsection{Volume estimates}
\begin{lemma}\label{mp3}
The domain $\M_3(p)$ is the collection of $3 \times 3$ lower
triangular matrices
$$
\begin{pmatrix} x_{11} \\ x_{21} & x_{22} \\ x_{31} & x_{32} & x_{33} \end{pmatrix}
$$
with entries in $\Z_p$ such that the following inequalities hold:
\begin{eqnarray*}
\mbox{[4-1]} & & v_p(x_{11})\le v_p(x_{21}^2-x_{21}x_{22})\\
\mbox{[4-2]} & & v_p(x_{11})\le v_p(x_{21}(x_{31}-x_{32}))\\
\mbox{[4-3]} & & v_p(x_{22})\le v_p(x_{32}^2-x_{32}x_{33})\\
\mbox{[4-4]} & & v_p(x_{11}) + v_p(x_{22}) \le
v_p(x_{22}(x_{31}^2- x_{31} x_{33})- x_{21}(x_{32}^2-x_{32}
x_{33})).
\end{eqnarray*}
\end{lemma}
\begin{proof}
We want to determine the conditions on matrices 
\[
M=\left(
\begin{array}{ccc}
x_{11} & 0 & 0 \\
x_{21} & x_{22} & 0\\
x_{31} & x_{32} & x_{33}\\
\end{array} \right),
\]
such that $x_{11},x_{21},x_{22}, x_{31}, x_{32}, x_{33} \in \Z_p$
and for $1\le i,j\le 3$, there exist $\alpha_1,\alpha_2,\alpha_3 \in\Z_p$ with $v_i\circ v_j=\alpha_1 v_1+\alpha_2 v_2 +
\alpha_3 v_3$, where $v_i$ is the $i^{th}$
row of the matrix $M$.

The condition that $v_2\circ v_2=\alpha_1 v_1+\alpha_2 v_2$ gives
the same condition that we had for the case $n=3$.  That is,
$v_p(x_{11})\le v_p(x_{21}^2-x_{21}x_{22})$.

We have \[v_2\circ v_3 = (x_{21}x_{31}, x_{22} x_{32},0) =\alpha_1
v_1 +\alpha_2 v_2 +\alpha_3 v_3.\]  Clearly $\alpha_3=0$. We have
$\alpha_2 x_{22}= x_{32} x_{22}$, so $\alpha_2=x_{32}$.  So we
have $\alpha_1 x_{11} + x_{32} x_{21} = x_{21} x_{31}$.  This
implies
\[\alpha_1 = x_{11}^{-1} (x_{21} x_{31}-x_{21}x_{32}).\] Therefore
$v_p(x_{11})\le v_p(x_{21}(x_{31}-x_{32}))$.

Next consider \[v_3\circ v_3 =(x_{31}^2,x_{32}^2, x_{33}^2)=
\alpha_1 v_1 +\alpha_2 v_2 +\alpha_3 v_3.\] We must have $\alpha_3
=x_{33}$. So $\alpha_2 x_{22} + x_{33} x_{32}= x_{32}^2$.  This
implies \[\alpha_2 = x_{22}^{-1} (x_{32}^2-x_{32} x_{33}).\]
Therefore $v_p(x_{22})\le v_p(x_{32}^2-x_{32}x_{33})$.

We also have $\alpha_1 x_{11} + x_{22}^{-1} (x_{32}^2-x_{32}
x_{33}) x_{21} + x_{33} x_{31} = x_{31}^2$.  This implies
\begin{eqnarray*}
\alpha_1 & = & x_{11}^{-1} (x_{31}^2- x_{31} x_{33} - x_{22}^{-1}
x_{21} (x_{32}^2-x_{32} x_{33}) )\\
& = & x_{11}^{-1} x_{22}^{-1}(x_{22}(x_{31}^2- x_{31} x_{33})-
x_{21} (x_{32}^2-x_{32} x_{33})).
\end{eqnarray*}
  So $v_p(x_{11}) + v_p(x_{22}) \le v_p(x_{22}(x_{31}^2- x_{31} x_{33})-
x_{21}(x_{32}^2-x_{32} x_{33}))$.

\end{proof}

Suppose that $v_p(x_{11}) = k,\ v_p(x_{22}) = l$ and $v_p(x_{33})
= r$.  By multiplying by appropriate units, we can suppose that
$x_{11} = p^k,\ x_{22} = p^l$ and $x_{33} = p^r$. Note that this
does not change the lattice generated by the rows. Then we can define  $\mu_p(k; l; r)$ as in Definition \ref{muk}. 

\begin{proposition}\label{generic case}
Suppose that $k,l, r \ge 0$. Then  
\begin{equation}\label{main-ineq}
\mu_p(k;l;r)\leq 8 p^{-7k/6} p^{-l/6}.
\end{equation}
\end{proposition}

\begin{proof}
We divide the proof into three steps.  We give two different bounds on $\mu_p(k;l;r)$ and then take an average. \\

\emph{Step I.} By Proposition \ref{zk2} the volume of $x_{32}$
satisfying inequality $[4-3]$ is at most $2p^{-l/2}$. By
Proposition \ref{zk2} the volume of $x_{21}$ satisfying inequality
$[4-1]$ is at most $2p^{-k/2}$, and for fixed $x_{21}, x_{32}$,
Proposition \ref{zk2} implies that the volume of $x_{31}$ satisfying
inequality $[4-4]$ is at most $2 p^{-k/2}$. Multiplication gives:
$$
\mu_p(k;l;r) \leq 8 p^{-k} p^{-l/2}.
$$

\emph{Step II.} By one of the steps of the proof of Proposition \ref{k-l z} the volume of $x_{21}$
satisfying inequality $[4-1]$ is at most $2 p^{-k+l}$. By
Proposition \ref{zk2} the volume of $x_{32}$ satisfying inequality
$[4-3]$ is at most $2 p^{-l/2}$. By Proposition \ref{zk2} the
volume of $x_{31}$ satisfying inequality $[4-4]$ is at most
$2p^{-k/2}$. Multiplication gives
$$
\mu_p(k;l;r) \leq 8 p^{-3k/2} p^{l/2}.
$$

\emph{Step III.} We now consider an appropriate average. The idea
is that if $\mu \leq A$ and $\mu \leq B$, with $\mu, A, B >0$,
then for all $m, n$ positive integers
$$
\mu \leq (A^m B^n)^{\frac{1}{m+n}}.
$$

The bounds from Steps I and II give
\begin{align*}
\mu_p & \leq \left\{\left( 8p^{-k} p^{-l/2} \right)^2 \left(8
p^{-3k/2} p^{l/2}\right)\right\}^{1/3} \\
& = 8 p^{-7k/6} p^{-l/6}.
\end{align*}
\end{proof}

\begin{rem}
This is not the best possible bound one can prove. In fact using a
more complicated argument similar to the proof of \emph{Step I} of
Theorem \ref{n=5} we can prove a bound of $Cp^{-9k/8}p^{-l/2}$ in
\emph{Step I} of the above theorem. This leads to the bound $\mu_p
\leq C p^{-5k/4}p^{-l/2}$ after averaging. This however will not
improve the bound in Theorem \ref{error-4} unless one has
an analogue of Theorem \ref{t=1} for $r=1$. Such a theorem is easy
to prove, but the resulting estimate would still not be as good as the one obtained in \cite{Liu}. For this reason we decided to include
only the simplest non-trivial estimate.
\end{rem}

\begin{proposition}\label{r=0}
Let $p$ be odd. If $r=0$ and $k, l \geq 1$, then
  $$
 \mu_p(k;l;0) \leq  24 p^{-3k/2 -l}.
  $$
\end{proposition}
\begin{proof}
Proposition \ref{zk2} implies that inequality $[4-1]$ holds on a set of $x_{21}$ of volume at most $2p^{-\lceil k/2\rceil}$. Proposition \ref{k-l z} implies that inequality $[4-3]$ holds on a set of $x_{32}$ of volume at most $2p^{-l}$.  For fixed $x_{21}, x_{32}$, Proposition \ref{k-l z} implies that inequality $[4-4]$ holds on a set of $x_{31}$ of volume at most $6p^{-k}$.

We see that our total volume is bounded by $24 p^{-k-l-\lceil
k/2\rceil}$.
\end{proof}

\begin{proposition}\label{l or k}
Let $p$ be odd. Then
$$
\mu_p(0; l; 0) \leq 2 p^{-l}
$$
and 
$$
\mu_p(k; 0; 0) \leq 3 p^{-2k}. 
$$
\end{proposition}
\begin{proof}
  If $k=r=0$, then inequality $[4-3]$ and Proposition \ref{k-l z} give the result. Now suppose $l=r=0$. Then we have
  $$
  v_p(x_{21}) + v_p(x_{21}-1) \ge k
  $$
  which determines two possibilities for $x_{21}$: \\

  1. $v_p(x_{21})  \geq k$. In this case inequality $[4-4]$ says
  $$
  v_p(x_{31}) + v_p(x_{31}-1) \ge k. 
  $$
  The volume of such $x_{31}$ is $2p^{-k}$. As a result the whole volume is at most $2 p^{-2k}$. \\

  2. $v_p(x_{21}) =0$ and $v_p(x_{21}-1) \ge k$. Then inequality $[4-2]$ gives
  $$
  v_p(x_{31}-x_{32}) \ge k
  $$
  and the two dimensional volume of $(x_{31}, x_{32})$ satisfying this inequality is at most $p^{-k}$. This gives a bound on the entire volume of $p^{-2k}$. \\

  Adding up gives the result.

\end{proof}

\subsubsection{Orders}\label{convergence for n=4}
In this section we prove the following theorem:
\begin{theorem}\label{error-4}
There is a polynomial $P_4$ of degree $5$ such that for all
$\epsilon >0$
$$
N_4(B) = B P_4(\log B) + O(B^{\frac{11}{12}+ \epsilon})
$$
as $B \to\infty$.\end{theorem}
\begin{proof} By Theorem \ref{tauberian application} it suffices to prove the following statement:  the expression
\begin{equation}\label{expression}
\sum_p \sum_{k + l + r \geq 2} p^{2k+l - k\sigma - l\sigma -
r\sigma} \mu_p(k;l;r)
\end{equation}
converges whenever $\sigma > \frac{11}{12}$.

\

We write the sum \eqref{expression} as
$$
\sum_{k + l + r \geq 2} 2^{2k+l - k\sigma - l\sigma - r\sigma}
\mu_2(k;l;r) + \sum_{p \text{ odd}}\sum_{k + l + r \geq 2} p^{2k+l
- k\sigma - l\sigma - r\sigma} \mu_p(k;l;r).
$$
By Proposition \ref{generic case} the first piece is majorized by
$$
\sum_{k , l , r \geq 0} 2^{2k+l - k\sigma - l\sigma - r\sigma}
2^{-7k/6}2^{-l/6}
$$
which converges for $\sigma > 5/6$.

\

We now consider the second piece of the sum. We consider three
cases. \\

\emph{Case I.} $r \geq 2$. By Proposition \ref{generic case} the
relevant sum is bounded by
$$
\sum_{p \text{ odd}}\sum_{r \geq 2}\sum_{k,l \geq 0}  p^{2k+l -
k\sigma - l\sigma - r\sigma} p^{-7k/6}p^{-l/6}=\sum_{p \text{
odd}}\sum_{r \geq 2}\sum_{k,l \geq 0}
p^{(\frac{5}{6}-\sigma)(k+l) - r\sigma}. 
$$
This sum is equal to
$$
\sum_{p \text{ odd}}\sum_{r \geq 2}\sum_{m \geq 0} (m+1)
p^{(\frac{5}{6}-\sigma)m - r\sigma}.
$$
This sum is converges for $\sigma >
\frac{5}{6}$. \\

\emph{Case II.} $r=1$. From the previous computation the
corresponding sum converges if the sum

$$
\sum_{p \text{ odd}}\sum_{m \geq 1} p^{(\frac{5}{6}-\sigma)m -
\sigma}
$$
converges. If $ \sigma > \frac{5}{6}$, the series converges if the series
$$
\sum_{p \text{ odd}} p^{(\frac{5}{6}-\sigma) - \sigma}
$$
converges. The latter converges for $\sigma > 11/12$. \\

\emph{Case III.} $r=0$. We write the corresponding sum as
$$
\sum_{p \text{ odd}}\sum_{k+l\geq 2 } p^{2k+l
- k\sigma - l\sigma} \mu_p(k;l;0)= \sum_{p \text{
odd}}\sum_{l \geq 2}p^{l
- l\sigma } \mu_p(0;l;0)
$$
$$
+\sum_{p \text{ odd}}\sum_{k \geq
2 }p^{2k
- k\sigma } \mu_p(k;0;0)+
\sum_{p \text{ odd}}\sum_{k,l \geq 1}p^{2k+l
- k\sigma - l\sigma} \mu_p(k;l;0).
$$

By Proposition \ref{l or k} we have
$$
\sum_{p \text{ odd}}\sum_{l \geq 2 }p^{l
- l\sigma } \mu_p(0;l;0) \ll \sum_{p \text{
odd}}\sum_{l \geq 2} p^{-l \sigma}
$$
and this is convergent for $\sigma > 1/2$. Again by Proposition \ref{l or k}
$$
\sum_{p \text{ odd}}\sum_{k \geq 2 }p^{2k
- k\sigma } \mu_p(k;0;0) \ll \sum_{k \geq
2}p^{-k \sigma}
$$
which converges for $\sigma > 1/2$. Finally by Proposition \ref{r=0}
$$
\sum_{p \text{ odd}}\sum_{k,l \geq 1} p^{2k+l
- k\sigma - l\sigma} \mu_p(k;l;0) \ll \sum_{p \text{
odd}} \sum_{k,l \geq 1} p^{(\frac{1}{2}-\sigma)k - l \sigma}.
$$
If $\sigma > \frac{1}{2}$ this last series converges if the series
$$
\sum_{p \text{ odd}}p^{(\frac{1}{2}-\sigma) -  \sigma}
$$
converges. This last series converges for $\sigma > \frac{3}{4}$.
\end{proof}

\begin{rem}
The bounds obtained by Liu \cite{Liu} for $f_3(k)$ and $f_4(k)$ are better than what we have obtained here. Liu proves $f_3(k) = O(k^{1/3})$ and $f_4(k) = O_\epsilon(k^{1/2+\epsilon})$. 
\end{rem}

\subsection{Orders of $\Z^5$}\label{volume estimates for n=5}

\subsubsection{Volume estimates} We will begin with the set
of inequalities defining our region of integration.

\begin{lemma} $\M_4(p)$ is the collection of matrices with entries
in $\Z_p$
$$
\begin{pmatrix} x_{11} \\ x_{21} & x_{22} \\ x_{31} & x_{32} & x_{33} \\ x_{41} & x_{42} & x_{43} &  x_{44} \end{pmatrix}
$$
whose entries satisfy:
\begin{eqnarray*}
\mbox{[5-1]} & & v_p(x_{11})\le v_p(x_{21}^2-x_{21}x_{22})\\
\mbox{[5-2]} & & v_p(x_{11})\le v_p(x_{21}(x_{31}-x_{32}))\\
\mbox{[5-3]} & & v_p(x_{22})\le v_p(x_{32}^2-x_{32}x_{33})\\
\mbox{[5-4]} & & v_p(x_{11}) + v_p(x_{22}) \le
v_p(x_{22}(x_{31}^2-
x_{31} x_{33})- x_{21}(x_{32}^2-x_{32} x_{33}))\\
\mbox{[5-5]} & & v_p(x_{11}) \le v_p(x_{21} (x_{41}-x_{42}))\\
\mbox{[5-6]} & & v_p(x_{22}) \le v_p(x_{32} (x_{42}-x_{43}))\\
\mbox{[5-7]}& & v_p(x_{11}) + v_p(x_{22}) \le v_p(x_{22}x_{31}(x_{41}-  x_{43})- x_{21}x_{32} (x_{42}- x_{43}))\\
\mbox{[5-8]} & & v_p(x_{33})\le v_p(x_{43}^2-x_{43}x_{44})\\
\mbox{[5-9]} & & v_p(x_{22}) + v_p(x_{33}) \le  v_p(x_{33}x_{42}(x_{42}-  x_{44})- x_{32}x_{43} (x_{43}- x_{44}))\\
\mbox{[5-10]} & & v_p(x_{11}) + v_p(x_{22}) + v_p(x_{33}) \le  v_p(x_{22} x_{33} x_{41}(x_{41}-  x_{44}) -x_{22} x_{31} x_{43}(x_{43}-  x_{44})\\
& &  - x_{21} x_{33} x_{42}(x_{42}-  x_{44}) + x_{21} x_{32}
x_{43}(x_{43}-  x_{44})).
\end{eqnarray*}
\end{lemma}
The proof of this lemma is very similar to the proof of Lemma
\ref{mp3}.

By multiplying by appropriate units, we can suppose that $x_{11} =
p^k,\ x_{22} = p^l,\ x_{33} = p^r$ and $x_{44} = p^t$. We define $\mu_p(k; l; r; t)$ as in Definition \ref{muk}. 

\

We start with a  lemma: 

\begin{lemma}\label{lemmamu} Let $p$ be a prime. Then there is a polynomial with positive coefficients $R \in \R[x]$ such that 
$$
\mu_p(k;l;r;t) \leq R(k) p^{-2k -l}. 
$$
\end{lemma}
\begin{proof} In this proof we will suppress the dependence of $R(k)$ on $k$, and will simply write $R$. The value of the polynomial $R$ does not affect the convergence of the sum we consider, so we do not compute it.
The key to our argument will be that
once our other variables are fixed, there are several different
bounds available to us for the volume of $x_{31}$ such that
inequalities $[5-4]$ and $[5-10]$ hold.  

More specifically, we use Proposition \ref{zk2} to give a bound on the volume of the possible set of $x_{32}$, then give a bound on the set of possible $x_{43}$.  Once these two values are fixed we again use Proposition \ref{zk2} to give a bound on the set of $x_{42}$, which then bounds the set of possible $x_{21}$.  Finally, we combine a few different possible bounds for the set of $x_{31}$ so that these inequalities simultaneously hold.

Proposition \ref{zk2} implies that inequality $[5-3]$ holds on a
set of $x_{32}$ of volume at most $2 p^{-l/2}$.

Suppose that $v_p(x_{43}(x_{43}-x_{44})) = r + z$.  Inequality
$[5-8]$ implies that $z \ge 0$.  This inequality holds on a set of
$x_{43}$ of volume at most $2p^{-r/2-z/2}$.  Fix such an $x_{43}$.

Now for fixed $x_{32}, x_{43}$, Proposition \ref{zk2} implies that
inequality $[5-9]$ holds on a set of $x_{42}$ of volume at most $2
p^{-l/2}$.

We now consider inequality $[5-5]$.  For fixed $x_{42}$
Proposition \ref{k+1 xy-z} implies that the total volume of
$x_{21}, x_{41}$ such that this inequality holds is at most $(k+1)
p^{-k}$.

Finally, we consider $x_{31}$.  We begin with inequality $[5-10]$.  For
fixed values of $x_{21}, x_{32}, x_{41}, x_{42}, x_{43}$, we can write this
as
\[k +l + r \le v_p(x_{31} x_{22} y - \tau),\]
where $y, \tau \in \Z_p$ with $v_p(y) = r + z$.  We see that this
holds on a set of $x_{31}$ of volume at most $p^{-(k-z)}$.

Consider inequality $[5-4]$.  By Proposition \ref{zk2}, this holds
on a set of $x_{31}$ of volume at most $2p^{-k/2}$.

Using $2 p^{-(k-z)}$ as our bound for the volume of $x_{31}$ gives a bound on our total volume of
\[R_1 p^{-2k-l-(r-z)/2},\]
for some polynomial $R_1$.  This is enough for our result if $r \ge z$.  Suppose that this is not the case.

By the proof of Proposition \ref{k-l z}, we see that the total volume of $x_{31}$ such that
\[v_p(x_{31}(x_{31}-x_{33}) - z) \ge k,\]
is at most $6 p^{-(k-r)}$ unless $p=2$, $v_p(x_{31}) = r-1$ and
$v_p(z) = 2r-2 < k$.  If we are not in this exceptional situation
the total volume is at most
 $R_2 p^{-2k-l -(z/2-r/2)}$. Since $r < z$, this is at most $R p^{-2k-l}$, completing the proof.

Suppose that we are in the situation where  $p=2$, $v_p(x_{31}) =
r-1$ and $v_p(z) = 2r-2 < k$.

First suppose that $v_p(x_{31}) \neq v_p(x_{32})$.  Then
$v_p(x_{31}-x_{32}) \le v_p(x_{31}) = r-1$.  Inequality $[5-2]$
now holds on a set of $x_{21}$ of volume at most $p^{-(k-r)}$.
Using this bound for the volume of $x_{21},\ 2p^{-l/2}$ for the
volume of $x_{32}$ and $2p^{-k/2}$ for the volume of $x_{31}$,
gives the total bound
\[ R_3 p^{-2k - l - (z-r)/2},\]
which is at most $A p^{-2k-l}$ for some polynomial $A$, since $z \ge r$.

Now suppose $v_p(x_{32}) = v_p(x_{31}) = r-1$.  Then
$v_p(x_{32}(x_{32}-x_{33})) = 2r-2$, and we must have $v_p(x_{21})
= l$.  Now consider inequality $[5-7]$.  We write $x_{21} = \alpha
p^l,\ x_{31} = \beta p^{r-1}$, and $x_{32} = \gamma p^{r-1}$ for
units $\alpha, \beta, \gamma$.  Factoring out $p^{l+r-1}$, the
inequality is now
\[v_p(\beta x_{41} - \alpha \gamma x_{42} + (\alpha \gamma - \beta) x_{43}) \ge k-r+1.\]
For fixed values of $x_{21}, x_{31}, x_{32}, x_{42}, x_{43}$, this holds on a set of $x_{41}$ of volume at most $p^{-(k-r)}$.  Using $2p^{-k/2}$ as our bound for $x_{21}$ and $x_{31}$, this gives total bound
\[R_4 p^{-2k - l - (z-r)/2},\]
which is at most $R p^{-2k-l}$, completing the proof. 
\end{proof}

\begin{proposition}\label{n=5 p=2} Let $p$ be any prime.
Suppose that $k,l,r,t\ge 0$.  Then for a polynomial $A \in \R[x]$ with positive coefficients we have
$$
\mu_p(k;l;r;t) \leq A(k) p^{-(2+ \frac{1}{34})k - (1+\frac{1}{34})l -
\frac{r}{17} + \frac{16t}{17}}.
$$
\end{proposition}

\begin{proof} 
The value of the polynomial $A$ does not affect the convergence of the sum we will consider so we do not compute it.  For example in the collection of equations \eqref{first-ineq}, \eqref{second-ineq}, and \eqref{third-ineq} the polynomials $A$ will not be the same. 

\

We have two steps: \\

\emph{Step I}.
Here we show that the following three inequalities
hold:
\begin{eqnarray}
\label{first-ineq} \mu_p(k;l;r;t) & \le & A p^{-3k/2 - 3l/2  + t}\\
\label{second-ineq} \mu_p(k;l;r;t) & \le & A p^{-2k - l - r + 3t}\\
\label{third-ineq} \mu_p(k;l;r;t) & \le & A p^{-5k/2 - l + r +
3t}.
\end{eqnarray}

We proceed as follows. Inequality $[5-1]$ holds on a set $x_{21}$
of volume at most the minimum of $2p^{-k/2}$ and $2 p^{-(k-l)}$. Inequality
$[5-3]$ holds on a set $x_{32}$ of volume at most the minimum of $2p^{-l/2}$ and
$2p^{-(l-r)}$. Inequality $[5-8]$ holds on a set of $x_{43}$ of
volume at most $2p^{-(r-t)}$.

When $p \neq 2$, we can use Proposition \ref{k-l z} for the
remaining three variables (See the proof of Theorem \ref{n=5} for
details). For $p = 2$, some care is required.  By Proposition
\ref{zk2} we always have the following. For any fixed $x_{21}$ and
$x_{32}$ inequality $[5-4]$ holds on a set of $x_{31}$ of volume
at most $2p^{-k/2}$. For any fixed $x_{32}, x_{43}$ inequality
$[5-9]$ holds on a set of $x_{42}$ of volume at most $2 p^{-l/2}$.
For any fixed $x_{21}, x_{31},x_{32},x_{42},x_{43}$ inequality
$[5-10]$ holds on a set of $x_{41}$ of volume at most $2p^{-k/2}$.

Inequality \eqref{first-ineq} follows from taking $2p^{-k/2}$
for the volume of $x_{21}, x_{31}, x_{41}$, taking $2 p^{-(l-r)}$
for the volume of $x_{32}$, taking $2p^{-l/2}$ for the volume of
$x_{42}$, and taking $2 p^{-(r-t)}$ for the volume of $x_{43}$.

For inequality \eqref{second-ineq} we take $2p^{-k/2}$ as our
bound for the volume of $x_{21}$ and $x_{31},\ 2p^{-l/2}$ as the
bound for $x_{32}$ and $x_{42}$, and $2 p^{-(r-t)}$ as the bound
for the volume of $x_{43}$.  We must now show that when all other
variables are fixed, the total volume of $x_{41}$ satisfying our
inequalities is at most $A p^{-(k-2t)}$.  

Suppose we are not in the special case in which we cannot apply Proposition \ref{k-l z}. We
have that the volume of $x_{41}$ satisfying inequality $[5-10]$ is
at most $6 p^{-(k-t)}$, completing this case.

We can write inequality $[5-10]$ as
\begin{eqnarray*}
 & & v_p(x_{11}) + v_p(x_{22}) + v_p(x_{33}) \le  v_p(x_{22} x_{33} x_{41}(x_{41}-  x_{44}) - (x_{22} x_{31} x_{43}(x_{43}-  x_{44})\\
& &  + x_{21} (x_{33} x_{42}(x_{42}-  x_{44}) - x_{32}
x_{43}(x_{43}-  x_{44})))).
\end{eqnarray*}

Inequality $[5-8]$ implies that we can write
$x_{43}(x_{43}-x_{44}) = p^r \alpha$, with $\alpha \in \Z_p$.
Inequality $[5-9]$ implies that we can write
\[x_{33} x_{42}(x_{42}-  x_{44}) - x_{32} x_{43}(x_{43}-  x_{44}) = p^{l+r} \beta,\]
with $\beta \in \Z_p$.

Our inequality is now
\[k \le v_p(x_{41}(x_{41}-x_{44}) - (x_{31} \alpha + x_{21} \beta)).\]
We can apply Proposition \ref{k-l z}, giving our bound, unless $v_p(x_{41}) = t-1$ and $v_p(x_{31}\alpha + x_{21} \beta) = 2t-2$.

First suppose that $v_p(x_{21}) \le 2t$.  Then for fixed $x_{21},
x_{42}$, inequality $[5-5]$ holds on a set of $x_{41}$ of volume
at most $p^{-(k-2t)}$, which completes this case.  Now suppose
that $v_p(x_{31}) \le 2t$.  Proposition \ref{xy-z} now implies
that for fixed $x_{21}, x_{31}, x_{32}, x_{42}, x_{43}$,
inequality $[5-7]$ holds on a set of $x_{41}$ of volume at most
$p^{-(k- v_p(x_{31}))} \le p^{-(k-2t)}$.  This is enough for our
bound, so we suppose that $v_p(x_{21}) \ge 2t$ and $v_p(x_{31})
\ge 2t$.  This implies that $v_p(x_{31}\alpha + x_{21} \beta) \ge
2t > 2t-2$, so we can apply Proposition \ref{k-l z}, completing
this case.

Inequality \eqref{third-ineq} will be proved in a few steps. First
we suppose that we are in the case where we can apply Proposition
\ref{k-l z} to inequality $[5-4]$ and conclude that the volume of
$x_{31}$ satisfying this inequality is at most $6 p^{-(k-r)}$. As
above, we see that either one of $x_{21}, x_{31}$ has valuation at
most $2t$, giving a bound of $ p^{-(k-2t)}$, or both have
valuation at least $2t$, in which case we can apply Proposition
\ref{k-l z} and conclude that the total volume of $x_{41}$ is at
most $6 p^{-(k-t)}$.  Using $2p^{-k/2}$ as our bound for $x_{21},\
2p^{-l/2}$ as our bound for $x_{32}$ and $x_{42}$, and
$2p^{-(r-t)}$ as our bound for $x_{43}$, we get total volume
\[A p^{-5k/2 - l +3t},\]
completing this case.

Now suppose that we are in the case where we cannot apply
Proposition \ref{k-l z} to inequality $[5-4]$.  Then $v_p(x_{31})
= r-1$.  We now consider two subcases.  First suppose that
$v_p(x_{31}) \neq v_p(x_{32})$.  Then inequality $[5-2]$ implies
that $v_p(x_{21}) \ge k- v_p(x_{31}) > k-r$, which holds on a set
of $x_{21}$ of volume at most $p^{-(k-r)}$.  We use $2p^{-k/2}$ as
the bound for the volume of $x_{31}$ satisfying inequality
$[5-4]$. Now using the same argument given above, the volume of
$x_{41}$ satisfying these inequalities is at most $6 p^{-(k-2t)}$.
Combining these estimates gives total volume bounded by
\[A p^{-5k/2 - l +3t},\]
completing this case.

Finally, suppose that $v_p(x_{31}) = v_p(x_{32}) = r-1$.  Now for
fixed $x_{32}, x_{43}$, the total volume of $x_{42}$ satisfying
inequality $[5-6]$ is at most $p^{-(l-r)}$.  We use $2p^{-(k-l)}$
as the bound on the volume of $x_{21}$ satisfying inequality
$[5-1]$, $2p^{-k/2}$ as the bound on the volume of $x_{31},\
2p^{-(r-l)}$ as our bound on the volume of $x_{32}$, and
$2p^{-(r-t)}$ as the bound on the volume of $x_{43}$.  Using the
same argument given above, we can use $6 p^{-(k-2t)}$ as our bound
on the volume of $x_{41}$.  This gives total bound
\[A p^{-5k/2 - l + r + 3t},\]
completing \emph{Step I}.

\

\emph{Step II}. Here we consider an appropriate average of the
previous inequalities to prove the theorem. The constants attached to these inequalities do not affect the convergence of the sums we consider so we will suppress them. By Lemma \ref{lemmamu} and Step I we have
$$
\mu_p \leq p^{-2k -l},
$$
$$
\mu_p\leq p^{-3k/2 - 3l/2 + t},
$$
$$
\mu_p\leq p^{-2k - l - r + 3t},
$$
and
$$
\mu_p \leq  p^{-5k/2 - l + r + 3t}.
$$
This means for all $n \geq 1$
\begin{align*}
\mu_p \leq &
\Big\{\left(p^{-3k/2 - 3l/2 + t}\right) \left(p^{-2k - l - r + 3t} \right)^3  \left(p^{-5k/2 - l + r + 3t} \right)^2 \left(p^{-2k-l}\right)^n\Big\}^{1/(n+6)} \\
= & p^{-(2+ \frac{1}{2(n+6)})k - (1+\frac{1}{2(n+6)})l - \frac{r}{n+6} +
\frac{16t}{n+6}}.
\end{align*}
Setting $n=11$ gives the result.
\end{proof}

We now state several results for odd primes $p$.

\begin{proposition}\label{n=5} Let $p$ be odd. Suppose that $k,l,r,t\ge 0$.  Then there is a polynomial $B \in \R[x]$ with positive coefficients such that $$
\mu_p(k; l; r; t) \leq B(k) p^{-(2 + \frac{1}{20})k-(1+
\frac{1}{20})l - \frac{r}{20} + \frac{9t}{20}}.
$$
\end{proposition}
\begin{proof}
We have two steps: \\

\emph{Step I}. Here we show that the following three inequalities
hold:
$$
\mu_p(k;l;r;t) \leq B p^{-2k - 3l/2 - r + 3t},
$$
$$
\mu_p(k;l;r;t) \leq B p^{-3k - l + r + 3t},
$$
and
\begin{equation}\label{ineq-useful}
\mu_p(k;l;r;t) \leq B p^{-5k/2-3l/2+3t}.
\end{equation}

We will use \eqref{ineq-useful} in the proof of Theorem \ref{t=1}.
We proceed as follows. Inequality $[5-1]$ holds on a set $x_{21}$
of volume at most the minimum of $2p^{-k/2}$ and $2 p^{-(k-l)}$. Inequality
$[5-3]$ holds on a set $x_{32}$ of volume at most the minimum of $2p^{-l/2}$ and
$2p^{-(l-r)}$. Inequality $[5-8]$ holds on a set of $x_{43}$ of
volume at most $2p^{-(r-t)}$. For any fixed $x_{21}$ and $x_{32}$
inequality $[5-4]$ holds on a set of $x_{31}$ of volume at most the minimum of
$2p^{-k/2}$ and $6 p^{-(k-r)}$. For any fixed $x_{32}, x_{43}$
inequality $[5-9]$ holds on a set of $x_{42}$ of volume at most
$6p^{-(l-t)}$. For any fixed $x_{21}, x_{31},x_{32},x_{42},x_{43}$
inequality $[5-10]$ holds on a set of $x_{41}$ of volume at most
$6p^{-(k-t)}$. Hence the total volume is bounded by
$$
B p^{-(k-t)}\cdot p^{-(l-t)}\cdot p^{-(r-t)} \cdot p^{-k/2} \cdot p^{-l/2} \cdot p^{-k/2},
$$
by
$$
B
p^{-(k-t)}\cdot p^{-(l-t)}\cdot p^{-(r-t)}\cdot p^{-(k-l)} \cdot p^{-(l-r)} \cdot p^{-(k-r)},
$$
and by
$$
B p^{-(k-t)} \cdot p^{-(l-t)} \cdot p^{-(r-t)} \cdot p^{-k/2} \cdot p^{-l/2} \cdot p^{-(k-r)}.
$$
Simplification gives the result. \\

\emph{Step II}. Here we consider an appropriate average of the
previous inequalities to prove the theorem. As constants play no
role we ignore them. By Lemma \ref{lemmamu} and Step I we have
$$
\mu_p \leq p^{-2k -l},
$$
$$
\mu_p\leq p^{-2k - 3l/2 - r + 3t},
$$
and
$$
\mu_p \leq  p^{-3k - l + r + 3t}.
$$
This means for all $n \geq 1$
\begin{align*}
\mu_p \leq &  \Big\{\left(p^{-2k - 3l/2 - r + 3t}\right)^2
\left(p^{-3k - l + r + 3t} \right) \left(p^{-2k
-l}\right)^n\Big\}^{1/(n+3)} \\
= & p^{-(2+ \frac{1}{n+3})k - (1+\frac{1}{n+3})l - \frac{r}{n+3} +
\frac{9t}{n+3}}.
\end{align*}
Setting $n=17$ gives the result.
\end{proof}

\begin{proposition}\label{t=0}
Let $p$ be odd. Then for any $k$, $l$, $r$, with $k+l+r \ge
2$, we have 
\[ \mu_p(k; l; r; 0) \leq C p^{-(2+ \frac{1}{7})k-(1 + \frac{1}{7})l - \frac{r}{7}-\frac{8}{7}}\]
for some constant $C>0$.
\end{proposition}

\begin{proof} We have two basic steps: \\

\emph{Step I.} Here we will show that $\mu_p \leq C p^{-2k - l
-2}$ whenever $k + l + r \geq 2$. We first note that Proposition
\ref{k-l z} implies that inequality $[5-8]$ holds on a set of
$x_{43}$ of volume at most $2p^{-(r-t)} = 2 p^{-r}$. Inequality
$[5-3]$ holds on a set of $x_{32}$ of volume at most $2p^{-\lceil
l/2\rceil}$.

Proposition \ref{zk2} implies that inequality $[5-1]$ holds on a
set of $x_{21}$ of volume at most $2p^{-\lceil k/2\rceil}$.  For
fixed $x_{21}, x_{32}$, Proposition \ref{zk2} implies that the
total volume of $x_{31}$ satisfying inequality $[5-4]$ is at most
$2p^{-\lceil k/2\rceil}$.

For fixed $x_{21}, x_{31}, x_{32}, x_{42}, x_{43}$, inequality
$[5-10]$ can be written as
\[k+l+r \le v_p(x_{22} x_{33} x_{41}(x_{41}-x_{44}) - z),\]
for some $z\in \Z_p$.  Proposition \ref{k-l z} implies that this
holds on a set of $x_{41}$ of volume at most $6 p^{-k}$.

Therefore, our total volume is at most
\[C p^{-k - 2\lceil k/2 \rceil -l -\lceil l/2 \rceil-r},\]
for some $C > 0$.  If $r + \lceil l/2\rceil \ge 2$, we are done.  Therefore,
suppose $r = 0$ and $l \in \{0,1,2\}$ or $r = 1$ and $l = 0$.

First suppose $r =0$.  Then Proposition \ref{k-l z} implies that
inequality $[5-3]$ holds on a set of $x_{32}$ of volume at most
$2p^{-l}$.  For fixed $x_{21}, x_{32}$, Proposition \ref{k-l z}
implies that inequality $[5-4]$ holds on a set of $x_{31}$ of
volume at most $6 p^{-k}$.  Using the above bounds for $x_{42}$
and $x_{41}$, our total volume is now bounded by
\[C p^{-2k - \lceil k/2 \rceil - 2l}.\]
Since $k+l \ge 2$, we have $\lceil k/2 \rceil + l \ge 2$ unless $l
= 0$ and $k = 2$.  In this case, we use $2p^{-k}$ as a bound for
the volume of $x_{21}$ satisfying inequality $[5-1]$, which
completes this case.

Now suppose $r = 1$ and $l = 0$.  Proposition \ref{k-l z} implies
that the volume of $x_{21}$ satisfying inequality $[5-1]$ is at
most $2p^{-k}$.  For fixed $x_{21}, x_{32}$, Proposition \ref{zk2}
implies that the total volume of $x_{31}$ satisfying inequality
$[5-4]$ is at most $2p^{-\lceil k/2\rceil}$.  We use the same
bounds for the volume of $x_{43}$ and $x_{41}$.  Our total volume
is now bounded by
\[C p^{-2k - \lceil k/2\rceil - 1}.\]
Since $k+l+r \ge 2$, we have $k \ge 1$ and our bound is at most $C
p^{-2k-2}$, completing the proof.

\emph{Step II.} This step is very similar to the last step of the
proof of Theorem \ref{n=5}. We have by the above and the second
step of the proof of Theorem \ref{n=5}
$$
\mu_p \leq p^{-2k -l-2},
$$
$$
\mu_p\leq p^{-2k - 3l/2 - r },
$$
and
$$
\mu_p \leq  p^{-3k - l + r}.
$$
This means for all $n \geq 1$
\begin{align*}
\mu_p \leq &  \Big\{\left(p^{-2k - 3l/2 - r }\right)^2
\left(p^{-3k - l + r } \right) \left(p^{-2k
-l-2}\right)^n\Big\}^{1/(n+3)} \\
= & p^{-(2+ \frac{1}{n+3})k - (1+\frac{1}{n+3})l - \frac{r}{n+3} -
\frac{2n}{n+3}}.
\end{align*}
Setting $n=4$ gives the result.

\end{proof}

We can similarly handle the case where $t=1$.

\begin{proposition}\label{t=1}
Let $p$ be odd. Then for any $k$, $l$, $r$ with $k+l+r \ge
1$, we have 
\[ \mu_p(k; l; r; 1) \leq D p^{-(2+ \frac{1}{18})k-(1+ \frac{1}{9})l- \frac{r}{9} - \frac{1}{9}},\]
for some constant $D > 0$.
\end{proposition}

\begin{proof} We have two main steps: \\

\emph{Step I.} Here we will show that the volume is bounded by
\[D p^{-2k-l-1}.\]

We recall that inequality $[5-1]$ holds on a set of $x_{21}$ of
volume at most the minimum of $2 p^{-\lceil k/2 \rceil}$ and
$2p^{-(k-l)}$. Similarly, inequality $[5-3]$ holds on set $x_{32}$
of volume at most the minimum of $2 p^{-\lceil l/2 \rceil}$ and
$2p^{-(l-r)}$. We also have that inequality $[5-8]$ holds on a set
of $x_{43}$ of volume at most the minimum of $2 p^{-\lceil r/2
\rceil}$ and $2p^{-(r-t)} = 2 p^{-(r-1)}$.

For any fixed values of $x_{21}, x_{32}$, we see that inequality
$[5-4]$ holds on a set of $x_{31}$ of volume at most the maximum
of $2 p^{-\lceil k/2 \rceil}$ and $6p^{-(k-r)}$.  For any fixed
values of $x_{32}, x_{43}$, we see that inequality $[5-9]$ holds
on a set of $x_{42}$ of volume at most the maximum of $2
p^{-\lceil l/2 \rceil}$ and $6 p^{-(l-1)}$.  For any fixed values
of $x_{21}, x_{31}, x_{32}, x_{42}, x_{43}$, we can write
inequality $[5-10]$ as $k\le v_p(x_{41}(x_{41}-x_{44}) - z)$, for
some $z\in \Z_p$.  This holds on a set of $x_{41}$ of volume at
most the maximum of $2 p^{-\lceil k/2 \rceil}$ and $6p^{-(k-1)}$.

We now combine these inequalities to get bounds on the total
volume satisfying inequalities $[5-1]$ through $[5-10]$.  Note
that if $k-l \ge \lceil k/2\rceil$ and $l-r \ge \lceil l/2
\rceil$, then $k-r \ge \lceil k/2 \rceil$.  By using $2 p^{-\lceil
k/2 \rceil}$ as the bound for the volume of $x_{21}$ and $x_{31}$,
and $2p^{-(l-r)}$ as the bound for $x_{32}$, we see that our total
volume is bounded by
\[D p^{-k-2l - 2 \lceil k/2 \rceil +3}.\]
Therefore, we are done if $l \ge 4$, or if $l \ge 3$ and $k$ is odd.  Suppose that this is not the case.

Suppose that $l \le 3$.  Using $2p^{-\lceil l/2\rceil}$ instead of $2p^{-(l-r)}$ as our bound for the volume of $x_{32}$, our total bound is now
\[D p^{-k-2 \lceil k/2 \rceil - l - \lceil l/2\rceil - r + 3}.\]
Therefore, we are done if $\lceil l/2 \rceil + r \ge 4$, or
$\lceil l/2 \rceil + r \ge 3$ and $k$ is odd.  Suppose that these
conditions do not hold.

First suppose that $l = 3$.  Then $r \le 1$.  We can use
$2p^{-\lceil r/2 \rceil}$ as a bound for the total volume of
$x_{43}$ satisfying inequality $[5-8]$ instead of $2p^{-(r-1)}$.
We use $2p^{-(l-r)}$ as our bound for the volume of $x_{32}$
satisfying inequality $[5-3]$. We see that our total volume is
bounded by
\[D p^{-k - 2 \lceil k/2 \rceil - 3 -3 + r - \lceil r/2 \rceil + 2} = D p^{-k - 2 \lceil k/2 \rceil - 4 + r - \lceil r/2 \rceil}.\]
Since $r \le 1$, this is at most $D p^{-2k-l-1}$, completing this case.

Now suppose that $l \le 2$.  For fixed $x_{32}, x_{43}$,
Proposition \ref{zk2} implies that the total volume of $x_{42}$
satisfying inequality $[5-9]$ is at most $2p^{-\lceil l/2\rceil}$.
We use this bound instead of $6 p^{-(l-1)}$.  Our total volume is
now bounded by
\[D p^{-k - 2 \lceil k/2 \rceil - 2 \lceil l/2 \rceil - r +2},\]
and we are done unless $r \le 2$.  In this case $\lceil r/2 \rceil
\ge r-1$, so we use $2 p^{-\lceil r/2\rceil}$ as our bound for the
volume of $x_{43}$ satisfying inequality $[5-8]$.  Now our bound
is
\[D p^{-k - 2 \lceil k/2 \rceil - 2 \lceil l/2 \rceil - \lceil r/2\rceil +1}.\]

First suppose $r = 2$.  Then if $l$ is odd or $k$ is odd, we are
done.  If $l = 0$, then we can use $2 p^{-k}$ as our bound for the
volume of $x_{21}$ satisfying inequality $[5-1]$, giving
\[D p^{-2k - \lceil k/2 \rceil},\]
as our bound.  Therefore, we are done unless $k=0$.  In this case,
$k=l=0$, we have that the total volume is at most the total volume
of $x_{43}$ satisfying inequality $[5-9]$, which is at most
$2p^{-1}$, which completes this case.

Now suppose $r=l = 2$.  This is the most difficult case to
consider.  If $k$ is odd then $2 \lceil k/2 \rceil = k+1$, and we
are done.  If $k \ge 6$, then we can use $2 p^{-(k-l)}$ as our
bound for $x_{21}$, which is enough to complete this case.  If $k
= 0$, then we use $1$ as our bound for $x_{41}$ instead of $6
p^{-(k-1)}$, and our total bound is $D p^{-l-1}$, completing this
case.  We now must consider $k = 2$ and $k =4$.

First suppose $k=2$.  We need a bound of $D p^{-7}$.  Using $2
p^{-\lceil k/2 \rceil}$ as our bound for $x_{21}, x_{31}, x_{41},\
2p^{-\lceil l/2\rceil}$ as our bound for $x_{32}$ and $x_{42}$,
and $2 p^{-\lceil r/2\rceil}$ as our bound for $x_{43}$, we get a
bound of $D p^{-6}$.  Since $l = k  = 2$ inequality $[5-1]$
becomes $2 v_p(x_{21}) \ge 2$ and inequality $[5-3]$ becomes $2
v_p(x_{32}) \ge 2$.  If either of these variables has valuation
greater than $1$, then we will have the upper bound that we need.
Therefore, we need only consider the case where $v_p(x_{21}) =
v_p(x_{32}) = 1$.  Inequality $[5-2]$ now implies that $v_p(x_{31}
- x_{32}) \ge 1$.  Therefore, $v_p(x_{31}) \ge 1$, and we note
that if $v_p(x_{32}) \ge 2$, we will have our bound.  Therefore we
suppose that $v_p(x_{31}) = 1$.  Finally, we consider inequality
$[5-4]$.  We have $v_p(x_{22}(x_{31}^2- x_{31} x_{33})) = 4 =
k+l$, but $v_p(x_{21}(x_{32}^2- x_{32} x_{33})) = 3 < k+l$, so
this case cannot occur.

When $k = 4$ we will argue similarly.  We need a bound of $D
p^{-11}$.  Using $2 p^{-\lceil k/2 \rceil}$ as our bound for
$x_{21}$ and $x_{31}$, $6 p^{-(k-1)}$ as our bound for $x_{41},\
2p^{-\lceil l/2\rceil}$ as our bound for $x_{32}$ and $x_{42}$,
and $2 p^{-\lceil r/2\rceil}$ as our bound for $x_{43}$, we get a
bound of $D p^{-10}$.   Since $l = r  = 2$ inequality $[5-8]$
becomes $2 v_p(x_{43}) \ge 2$ and inequality $[5-3]$ becomes $2
v_p(x_{32}) \ge 2$.  If either of these variables has valuation
greater than $1$, then we will have the bound that we need.
Therefore, we need only consider the case where $v_p(x_{43}) =
v_p(x_{32}) = 1$.  Inequality $[5-6]$ now implies that $v_p(x_{42}
- x_{43}) \ge 1$.  Therefore, $v_p(x_{42}) \ge 1$, and we note
that if $v_p(x_{42}) \ge 2$, we will have our bound.  Therefore we
suppose that $v_p(x_{42}) = 1$.  Finally, we consider inequality
$[5-9]$.  We have $v_p(x_{33}(x_{42}^2- x_{42} x_{43})) = 4 =
l+r$, but $v_p(x_{32}(x_{43}^2- x_{43} x_{44})) = 3 < l+r$, so
this case cannot occur.

Next suppose $l\le 2$ and $r = 1$.  We have the bound
\[D p^{-k - 2 \lceil k/2 \rceil - 2 \lceil l/2 \rceil}.\]
If $l = 1$, we are done.  Suppose $l = 2$.  Then we can use $2
p^{-l}$ as our bound for the volume of $x_{32}$ satisfying
inequality $[5-3]$, and we are done.  If $l = 0$, then we can use
$2p^{-k}$ as the bound for $x_{21}$ satisfying inequality $[5-1]$,
and our bound is
\[D p^{-2k -\lceil k/2\rceil},\]
which completes this case unless $k=0$.  If $k = l = 0$ and $r = t
= 1$, then our total volume is at most the volume of $x_{43}$
satisfying inequality $[5-8]$, which is $2 p^{-1}$, and we are
done.

Finally, suppose $r = 0$ and $l\le 2$.  We can use $2p^{-l}$ as
our bound for the volume of $x_{32}$ satisfying inequality
$[5-3]$, and for fixed $x_{21}, x_{32}$, we use $6 p^{-k}$ as our
bound for the volume of $x_{31}$ satisfying inequality $[5-4]$. We
also use $2 p^{-\lceil k/2\rceil}$ as our bound for the volume of
$x_{41}$ satisfying inequality $[5-10]$.  Our total volume is now
bounded by
\[ D p^{-2k - \lceil k/2 \rceil - l - \lceil l/2 \rceil}.\]
Since $k + l + r \ge 1$, we are done.

\emph{Step II.} Again we do an averaging. We have the inequalities
$$
\mu_p \leq p^{-2k -l-1},
$$
$$
\mu_p\leq p^{-2k - 3l/2 - r + 3},
$$
and
$$
\mu_p \leq p^{-5k/2-3l/2+3}.
$$
The last two inequalities are from \emph{Step II} of the proof of
Theorem \ref{n=5} for $t=1$. This means for all $n \geq 1$
\begin{align*}
\mu_p \leq &  \Big\{\left(p^{-2k - 3l/2 - r +3 }\right)
\left(p^{-5k/2-3l/2+3} \right) \left(p^{-2k
-l-1}\right)^n\Big\}^{1/(n+2)} \\
= & p^{-(2+ \frac{1}{2(n+2)})k - (1+\frac{1}{n+2})l -
\frac{r}{n+2} + \frac{6-n}{n+2}}.
\end{align*}
We set $n=7$ to get the result.
\end{proof}

\begin{rem}
The case by case analysis of the small values of parameters in the
proofs of Theorems \ref{t=0} and \ref{t=1} can be avoided if
instead one uses the results of \cite{Liu} for $f_n(p^k)$ for
small $k$. In \cite{Liu} these values are worked out for $k$ up to
$5$. This is not sufficient for our purposes, but computing the
missing data is not difficult using the results of Liu. Here we
chose instead to present the above elementary treatment to make
the argument self-contained.
\end{rem}

\begin{rem}
The choices of the parameter $n$ in the proofs of Theorems \ref{n=5 p=2}, \ref{n=5}, \ref{t=0}, and \ref{t=1} are made to optimize the error estimate in Theorem \ref{error-5}. 
\end{rem}

\subsubsection{Orders}\label{convergence for n=5}

In this section we prove the following theorem:

\begin{theorem}\label{error-5}
There is a polynomial $P_5$ of degree $9$ such that for all
$\epsilon >0$
$$
N_5(B) = B P_5( \log B) + O(B^{\frac{33}{34}+ \epsilon})
$$
as $B \to \infty$. \end{theorem}
\begin{proof}
By Theorem \ref{tauberian application}, it suffices to prove the following statement: for $\sigma > \frac{33}{34}$ the expression

 $$
 \sum_{p}\sum_{m \geq 2} \frac{a^<_{\Z^4}(p^m)}{p^{m\sigma}}
 $$
  converges.

In our analysis we will ignore all constants as they will have no
bearing on convergence. We write

$$
\sum_p \sum_{m \geq 2} \frac{a^<_{\Z^4}(p^m)}{p^{m\sigma}} = \sum_{m \geq 2}\frac{a^<_{\Z^4}(2^m)}{2^{m\sigma}} + \sum_{p \text{ odd }}\sum_{m \geq 2} \frac{a^<_{\Z^4}(p^m)}{p^{m\sigma}}.
$$

If we use Proposition \ref{n=5 p=2} we see very easily that the first
piece converges for $\sigma > \frac{33}{34}$. So we concentrate on
the sum corresponding to the odd primes.  We will show that for $m \geq 2$ and $p$ odd we have 
\begin{equation}\label{narrow-special}
a_{\Z^4}^<(p^m) \leq A(m) p^{-1+ \frac{19}{20} m}
\end{equation}
for a polynomial $A(m)$.

It is clear that this will be sufficient for the proof of the theorem. In order to prove \eqref{narrow-special} we write 
$$
a_{\Z^4}^<(p^m) = \sum_{k+l+r+t=m} p^{3k+2l+r}\mu_p(k;l;r;t) 
$$
$$
= \sum_{t=2}^m \sum_{k+l+r=m-t} p^{3k+2l+r}\mu_p(k;l;r;t) 
$$
$$
+ \sum_{k+l+r = m-1, t=1} p^{3k+2l+r}\mu_p(k;l;r;t) 
$$
$$
+ \sum_{k+l+r=m, t=0}p^{3k+2l+r}\mu_p(k;l;r;t) 
$$
$$
\leq \sum_{t=2}^m \sum_{k+l+r=m-t} p^{3k+2l+r} p^{-(2+1/20)k - (1+1/20)l - r/20+ 9t/20}
$$
$$
+ \sum_{k+l+r = m-1} p^{3k+2l+r}p^{-(2+1/18)k - (1+1/9)l - r/9-1/9}
$$
$$
+ \sum_{k+l+r = m, t=0} p^{3k+2l+r}p^{-(2+1/7)k - (1+1/7)l - r/7-8/7} 
$$
by Propositions \ref{n=5}, \ref{t=0}, \ref{t=1}, after ignoring some polynomials in terms of $k, l, r, t$ as coefficients. Next, 
$$
a_{\Z^4}^<(p^m) \leq \sum_{t=2}^m p^{9t/20} p^{(1-1/20)(m-t)} \sum_{k+l+r=m-t} 1 
$$
$$
+ p^{-1/9} p^{(1-1/18)(m-1)} \sum_{k+l+r=m-1} 1 + p^{-8/7}p^{(1-1/7)m} \sum_{k+l+r=m} 1
$$
$$
\leq p^{-1+(1-1/20)m} + p^{-19/18 + (1-1/18)m} + p^{-8/7 + (1-1/7)m}
$$
after ignoring some polynomials. Now the result follows. 
\end{proof}

The following statement is a consequence of the inequality \eqref{narrow-special}: 

\begin{cor}
For each $\epsilon >0$
$$
f(k) \ll_\epsilon k^{\frac{33}{34} + \epsilon} \prod_{p | k} p^{-1}. 
$$
If $k$ is odd, then for each $\epsilon > 0$, 
$$
f(k) \ll_\epsilon k^{\frac{19}{20} + \epsilon} \prod_{p | k} p^{-1}. 
$$
\end{cor}

\begin{rem}
Using Proposition \ref{n=5 p=2} for odd primes instead of Proposition
\ref{n=5} in the proof of Theorem \ref{error-5} would have
produced a weaker error term.
\end{rem}

\subsection{Orders of $\Z^d$ for $d>5$}\label{general n}
In this section we prove part 2 of Theorem \ref{mainthm:1}. The idea is to find non-trivial volume bounds for $\M_5(p)$, and then use an inductive argument to obtain bounds for $\M_d(p)$ for $d>5$. 

\

We begin by defining $\M_5(p)$.

\begin{lemma}\label{M5}
$\M_5(p)$ is the collection of $5 \times 5$ lower triangular matrices with entries in $\Z_p$
\begin{equation*}
\begin{pmatrix} x_{11} \\
                \ x_{21} & x_{22} \\
                \ x_{31} & x_{32} & x_{33} \\
                \ x_{41} & x_{42} & x_{43} & x_{44} \\
                \ x_{51} & x_{52} & x_{53} & x_{54} & x_{55} \end{pmatrix}
\end{equation*}

whose entries satisfy:

\begin{enumerate}
\item [{[6-1]}] $v_p(x_{11}) \leq v_p(x_{21}(x_{21} - x_{22}))$
\item [{[6-2]}] $v_p(x_{11}) \leq v_p(x_{21}(x_{31} - x_{32}))$
\item [{[6-3]}] $v_p(x_{22}) \leq v_p(x_{32}(x_{32} - x_{33}))$
\item [{[6-4]}] $v_p(x_{11}) + v_p(x_{22}) \leq v_p(x_{22}x_{31}(x_{31} - x_{33}) - x_{21}x_{32}(x_{32} - x_{33}))$
\item [{[6-5]}] $v_p(x_{11}) \leq v_p(x_{21}(x_{41} - x_{42}))$
\item [{[6-6]}] $v_p(x_{22}) \leq v_p(x_{32}(x_{42} - x_{43}))$
\item [{[6-7]}] $v_p(x_{11}) + v_p(x_{22}) \leq v_p(x_{22}x_{31}(x_{41} - x_{43}) - x_{21}x_{32}(x_{42} - x_{43}))$
\item [{[6-8]}] $v_p(x_{33}) \leq v_p(x_{43}(x_{43} - x_{44}))$
\item [{[6-9]}] $v_p(x_{22}) + v_p(x_{33}) \leq v_p(x_{33}x_{42}(x_{42} - x_{44}) - x_{32}x_{43}(x_{43} - x_{44}))$
\item [{[6-10]}] $v_p(x_{11}) + v_p(x_{22}) + x_{33} \leq v_p(x_{22}x_{33}x_{41}(x_{41} - x_{44}) - x_{22}x_{31}x_{43}(x_{43} - x_{44}) - x_{21}x_{33}x_{42}(x_{42} - x_{44}) + x_{21}x_{32}x_{43}(x_{43} - x_{44}))$
\item [{[6-11]}] $v_p(x_{11}) \leq v_p(x_{21}(x_{51} - x_{52}))$
\item [{[6-12]}] $v_p(x_{22}) \leq v_p(x_{32}(x_{52} - x_{53}))$
\item [{[6-13]}] $v_p(x_{11}) + v_p(x_{22}) \leq v_p(x_{22}x_{31}(x_{51} - x_{33}) - x_{21}x_{32}(x_{52} - x_{53}))$
\item [{[6-14]}] $v_p(x_{33}) \leq v_p(x_{43}(x_{53} - x_{54}))$
\item [{[6-15]}] $v_p(x_{22}) + v_p(x_{33}) \leq v_p(x_{33}x_{42}(x_{52} - x_{54}) - x_{32}x_{43}(x_{53} - x_{54}))$
\item [{[6-16]}] $v_p(x_{11}) + v_p(x_{22}) + x_{33} \leq v_p(x_{22}x_{33}x_{41}(x_{51} - x_{54}) - x_{22}x_{31}x_{43}(x_{53} - x_{54}) - x_{21}x_{33}x_{42}(x_{52} - x_{54}) + x_{21}x_{32}x_{43}(x_{53} - x_{54}))$
\item [{[6-17]}] $v_p(x_{44}) \leq v_p(x_{54}(x_{54} - x_{55}))$
\item [{[6-15]}] $v_p(x_{33}) + v_p(x_{44}) \leq v_p(x_{44}x_{53}(x_{53} - x_{5}) - x_{43}x_{54}(x_{54} - x_{55})$
\item [{[6-19]}] $v_p(x_{22}) + v_p(x_{33}) + x_{44} \leq v_p(x_{33}x_{44}x_{52}(x_{52} - x_{55}) - x_{33}x_{42}x_{54}(x_{54} - x_{55}) - x_{32}x_{44}x_{53}(x_{53} - x_{55}) + x_{32}x_{43}x_{54}(x_{54} - x_{55}))$
\item [{[6-20]}] $v_p(x_{11}) + v_p(x_{22}) + v_p(x_{33}) + v_p(x_{44}) \leq v_p(x_{22}x_{33}x_{44}x_{51}(x_{51} - x_{55}) - x_{22}x_{33}x_{41}x_{54}(x_{54} - x_{55}) - x_{22}x_{31}x_{44}x_{53}(x_{53} - x_{55}) + x_{22}x_{31}x_{43}x_{54}(x_{54} - x_{55}) - x_{21}x_{33}x_{44}x_{52}(x_{52} - x_{55}) - x_{21}x_{33}x_{42}x_{54}(x_{54} - x_{55}) - x_{21}x_{32}x_{44}x_{53}(x_{53} - x_{55}) + x_{21}x_{32}x_{43}x_{54}(x_{54} - x_{55}))$
\end{enumerate}
\end{lemma}

We omit the proof. 

\

As usual after multiplying by appropriate units, we can assume that $x_{11} = p^{k_1}$, $x_{22} = p^{k_2}$, $x_{33} = p^{k_3}$, $x_{44} = p^{k_4}$, and $x_{55} = p^{k_5}$.

\

We now give a  bound for $\mu_p(k_1, k_2, k_3, k_4, k_5)$.

\begin{proposition}\label{bound6}
For odd prime p,
\begin{equation*}
\mu_p(k_1, k_2, k_3, k_4, k_5) \leq c \cdot p^{-(\frac{5}{2} + \frac{1}{6})k_1 - (\frac{3}{2} + \frac{1}{6})k_2 - (\frac{1}{2} + \frac{1}{6})k_3 - (\frac{1}{2} - \frac{2}{6})k_4 + \frac{2}{6} k_5}
\end{equation*}
where $c$ is a polynomial in $k_1, \ldots, k_5$.
\end{proposition}

\begin{proof}  First we show the following three inequalities:
\begin{align}
\mu_p & \leq c_1 \cdot p^{-3 k_1 - \frac{3}{2} k_2 - k_3 + k_5} & =: A \\
\mu_p & \leq c_2 \cdot p^{-\frac{5}{2} k_1 - \frac{3}{2} k_2 - \frac{1}{2} k_3 - \frac{1}{2} k_4} & =: B \\
\mu_p & \leq c_3 \cdot p^{-\frac{5}{2} k_1 - 2 k_2 - \frac{1}{2} k_3} & =: C
\end{align}
where $c_1, c_2, c_3$ are polynomials in $k_1, \ldots, k_5$.

To show (1), we see that inequality [6-1] holds on a set of $x_{21}$ of volume at most $2p^{-k_1/2}$ by Proposition \ref{zk2}.
We see that [6-4] holds on a set of $x_{31}$ of volume at most $2p^{-k_1/2}$ by Proposition \ref{zk2}.
The combined volume of $x_{41}$ and $x_{54}$ satisfying [6-16] is at most $(k_1 + 1) p^{-k_1}$ by Proposition \ref{k+1 xy-z}.
The volume of $x_{51}$ satisfying [6-20] is at most $6p^{-k_1 + k_5}$ by Proposition \ref{k-l z}.
The volume of $x_{32}$ satisfying [6-3] is at most $2p^{-k_2/2}$ by Proposition \ref{zk2}.
The volume of $x_{42}$ satisfying [6-9] is at most $2p^{-k_2/2}$ by Proposition \ref{zk2}.
The volume of $x_{52}$ satisfying [6-19] is at most $2p^{-k_2/2}$ by Proposition \ref{zk2}.
The volume of $x_{43}$ satisfying [6-8] is at most $2p^{-k_3/2}$ by Proposition \ref{zk2}.
The volume of $x_{53}$ satisfying [6-18] is at most $2p^{-k_3/2}$ by Proposition \ref{zk2}.
Multiplication gives $$\mu_p \leq c_1 \cdot p^{-3 k_1 - \frac{3}{2} k_2 - k_3 + k_5} = A.$$

To show (2), we see that inequality [6-1] holds on a set of $x_{21}$ of volume at most $2p^{-k_1/2}$ by Proposition \ref{zk2}.
The combined volume of $x_{31}$ and $x_{43}$ satisfying [6-7] is at most $(k_1 + 1) p^{-k_1}$ by Proposition 7.
The volume of $x_{41}$ satisfying [6-10] is at most $2p^{-k_1/2}$ by Proposition \ref{zk2}.
The volume of $x_{51}$ satisfying [6-20] is at most $2p^{-k_1/2}$ by Proposition \ref{zk2}.
The volume of $x_{32}$ satisfying [6-3] is at most $2p^{-k_2/2}$ by Proposition \ref{zk2}.
The volume of $x_{42}$ satisfying [6-9] is at most $2p^{-k_2/2}$ by Proposition \ref{zk2}.
The volume of $x_{52}$ satisfying [6-19] is at most $2p^{-k_2/2}$ by Proposition \ref{zk2}.
The volume of $x_{53}$ satisfying [6-18] is at most $2p^{-k_3/2}$ by Proposition \ref{zk2}.
The volume of $x_{54}$ satisfying [6-17] is at most $2p^{-k_4/2}$ by Proposition \ref{zk2}.
Multiplication gives $$\mu_p \leq c \cdot p^{-\frac{5}{2} k_1 - \frac{3}{2} k_2 - \frac{1}{2} k_3 - \frac{1}{2} k_4} = B.$$

To show (3), we see that inequality [6-1] holds on a set of $x_{21}$ of volume at most $2p^{-k_1/2}$ by Proposition \ref{zk2}.
The volume of $x_{31}$ satisfying [6-4] is at most $2p^{-k_1/2}$ by Proposition \ref{zk2}.
The combined volume of $x_{41}$ and $x_{54}$ satisfying [6-16] is at most $(k_1 + 1) p^{-k_1}$ by Proposition \ref{k+1 xy-z}.
The volume of $x_{51}$ satisfying [6-20] is at most $2p^{-k_1/2}$ by Proposition \ref{zk2}.
The combined volume of $x_{32}$ and $x_{43}$ satisfying [6-6] is at most $(k_2 + 1) p^{-k_2}$ by Proposition  \ref{k+1 xy-z}.
The volume of $x_{42}$ satisfying [6-9] is at most $2p^{-k_2/2}$ by Proposition \ref{zk2}.
The volume of $x_{52}$ satisfying [6-19] is at most $2p^{-k_2/2}$ by Proposition \ref{zk2}.
The volume of $x_{53}$ satisfying [6-18] is at most $2p^{-k_3/2}$ by Proposition \ref{zk2}.
Multiplication gives $$\mu_p \leq c \cdot p^{-\frac{5}{2} k_1 - 2 k_2 - \frac{1}{2} k_3} = C.$$

Lastly, we note that $\mu_p \le \min\left\{A,B,C\right\}$ implies that $$\mu_p \leq (ABC)^{1/3} = c \cdot p^{-(\frac{5}{2} + \frac{1}{6})k_1 - (\frac{3}{2} + \frac{1}{6})k_2 - (\frac{1}{2} + \frac{1}{6})k_3 - (\frac{1}{2} - \frac{2}{6})k_4 + \frac{2}{6} k_5}$$
giving the result.
\end{proof}

\begin{proposition}\label{gen}
Suppose $n \geq 5$.  Then there is $C \in \R[k_1,\ldots, k_5]$ such that
$$\mu_p(k_1, \ldots, k_d) \leq C p^{-A_d(p) - \sum_{j=6}^{d} (d-j) \left\lceil \frac{k_j}{2} \right\rceil}$$
with
$$A_d(p) = \left(\frac{d} {2} + \frac{1} {6}\right)k_1 + \left(\frac{d-2} {2} + \frac{1} {6}\right)k_2 + \left(\frac{d-4} {2} + \frac{1} {6}\right)k_3 + \left(\frac{d-4} {2} - \frac{1} {6}\right)k_4 + \left(\frac{d-5} {2} - \frac{2} {6}\right)k_5$$
for p odd, and
$$A_d(p) = \left(\frac{d} {2} + \frac{1} {34}\right)k_1 + \left(\frac{d-2} {2} + \frac{1} {34}\right)k_2 + \left(\frac{d-4} {2} + \frac{1} {17}\right)k_3 + \left(\frac{d-4} {2} - \frac{16} {17}\right)k_4 + \left(\frac{d-5} {2}\right)k_5$$
for $p = 2$.
\end{proposition}

\begin{proof}
The proof is by induction on $d$. Since $C$ will not affect the convergence of the sums we consider we do not compute it.  The lemma will follow from
Theorem \ref{n=5 p=2} and Theorem \ref{M5} if we show that
\begin{equation}\label{ineq-mu}
\mu_p(k_1; \dots; k_d) \leq 2^{d-1} p^{- \sum_{j=1}^{d-1}\left\lceil
\frac{k_j}{2}\right\rceil} \mu_p(k_1, \dots, k_{d-1}).
\end{equation}
In order to see this inequality  observe that if
\[M=\left(
\begin{array}{cccc}
p^{k_1} & 0  & \ldots  & 0\\
x_{21} & p^{k_2} & 0 & \vdots\\
\vdots & \vdots & \ddots & 0\\
x_{d1} & \ldots  & \ldots & p^{k_d}\\
\end{array} \right) \in \M_d(p; k_1, \dots, k_d) \]
then for the matrix obtained by removing the last row
\[M'=\left(
\begin{array}{cccc}
p^{k_1} & 0  & \ldots  & 0\\
x_{21} & p^{k_2} & 0 & \vdots\\
\vdots & \vdots & \ddots & 0\\
x_{d-1\, 1} & \ldots  & \ldots & p^{k_{d-1}}\\
\end{array} \right) \in \M_{d-1}(p; k_1, \dots, k_{d-1}). \]
The inequality \eqref{ineq-mu} will follow if we show that the
fibers of the map $M \mapsto M'$ have volume bounded by
$$
 2^{d-1} p^{-
\sum_{j=1}^{d-1} \left\lceil\frac{k_j}{2}\right\rceil}.
$$
As usual we set
$$
v_j = (x_{j1}, \dots, x_{jj}, 0, \dots, 0).
$$
Suppose $v_1, \dots, v_{d-1}$ are the rows of $M'$. We now bound
the volume of the set of vectors $v_d$ with $x_{dd} = p^{k_d}$
such that
$$
v_d \circ v_d = c_1 v_1 + \dots + c_d v_d
$$
with $c_i \in \Z_p$. It is clear that $c_d = x_{dd}$. We then see
that for $1 \leq j \leq d-1$
$$
x_{dj}^2 - x_{dd}x_{dj} = c_j x_{jj} + \sum_{k=j+1}^{d-1} c_k
x_{kj}.
$$
If $c_k, x_{kj}$ are given for $j+1 \leq k \leq d$, then the existence of such a 
a $c_j$ is equivalent to
$$
v_p\left(x_{dj}^2 - x_{dd}x_{dj} - \sum_{k=j+1}^{d-1} c_k
x_{kj}\right) \geq k_j.
$$
Proposition \ref{zk2} implies that the volume of $x_{dj}$ is
bounded by $2p^{-\lceil k_j/2\rceil}$. Induction will give the
result.

\end{proof}

We can now prove part 2 of Theorem \ref{mainthm:1}: 

\begin{proof}
We will prove this theorem for $\Z^{d+1}$.  We will show that the abscissa of convergence of $\zeta_{\Z^d}^<(s)$ is less than or equal to $\frac{d-1}{2} - \frac{1}{6}$.  Recall 
$$\zeta_{\Z^d}^<(s) = \prod\limits_p \sum\limits_{k_1, \ldots, k_d \geq 0} p^{\sum_{j=1}^d (d-j)k_j} p^{-s \sum_{j=1}^d k_j} \mu_p(k_1, \ldots, k_d).$$
It is not hard to see that by Lemma \ref{gen} the factor corresponding to $p = 2$ converges for $\sigma = \Re (s) > \frac{d-1}{2} - \frac{1}{6}$.  For the remainder of this proof we will write $\sum_p$ for $\sum_{p \textmd{ odd}}$.  It remains to prove the convergence of the series
\begin{align*}
& \sum\limits_{p} \sum\limits_{k_1 + \ldots + k_d \geq 1} p^{\sum_{j=1}^d (d-j)k_j} p^{-\sigma \sum_{j=1}^d k_j} \mu_p(k_1, \ldots, k_d) \\
= & \sum\limits_{p} \sum\limits_{k_1 + \ldots + k_d = 1} p^{\sum_{j=1}^d (d-j)k_j} p^{-\sigma \sum_{j=1}^d k_j} \mu_p(k_1, \ldots, k_d) \\ & + \sum\limits_{p \textmd{ odd}} \sum\limits_{k_1 + \ldots + k_d \geq 2}p^{\sum_{j=1}^d (d-j)k_j} p^{-\sigma \sum_{j=1}^d k_j} \mu_p(k_1, \ldots, k_d).
\end{align*}
By Lemma  \ref{n+1 choose 2} $$\sum\limits_{k_1 + \ldots + k_d = 1} p^{\sum_{j=1}^d (d-j)k_j} p^{-\sigma \sum_{j=1}^d k_j} \mu_p(k_1, \ldots, k_d) = \binom{d+1}{2} p^{-\sigma}.$$
and $\sum_p \binom{d+1}{2} p^{-\sigma}$ converges for all $\sigma > 1$.  By Theorem \ref{bound6} we see that the other summand is bounded by
\begin{align*}
\sum\limits_{p} & \sum\limits_{k_1 + \ldots + k_d \geq 2} p^{\sum_{j=1}^d (d-j)k_j} p^{-\sigma \sum_{j=1}^d k_j} \mu_p(k_1, \ldots, k_d) \\ & \leq \sum\limits_{p} \sum\limits_{k_1 + \ldots + k_d \geq 2} p^{\sum_{j+1}^d (d-j)k_j} p^{-\sigma \sum_{j=1}^d k_j} p^{-A_d - \sum_{j=5}^d (d-j) \left\lceil \frac{k_j}{2} \right\rceil} \\
& \leq \sum\limits_p \sum\limits_{k_1 + \ldots + k_d \geq 2} p^{B_d + \frac{1}{2} \sum_{j=5}^d (d-j) k_j} p^{-\sigma \sum_{j=1}^d k_j}
\end{align*}
where
$$B_d = \left( \frac{d}{2} - 1 - \frac{1}{6} \right)(k_1 + k_2 + k_3) + \left(\frac{d}{2} - 2 + \frac{1}{6}\right)k_4 + \left(\frac{d}{2} - 2 -\frac{1}{6}\right)k_5.$$

Our series is now bounded by
$$\sum\limits_p \sum\limits_{k_1 + \ldots + k_d \geq 2} p^{\left( \frac{d}{2} - 1 - \frac{1}{6} - \sigma \right) \sum_{j=1}^{d-1} k_j} p^{-\sigma k_d}$$ $$= \sum\limits_p \sum\limits_{m + k_d \geq 2} C_d(m) p^{\left( \frac{d}{2} - 1 - \frac{1}{6} - \sigma \right) m} p^{-\sigma k_d}$$
where $C_d(m)$ is the number of solutions to $\sum_{j=1}^{d-1} k_j = m$ for $m \geq 0$.  Since $C_d(m)$ is a polynomial in $m$, this series converges if and only if
$$\sum\limits_p \sum\limits_{m + k_d \geq 2} p^{\left( \frac{d}{2} - 1 - \frac{1}{6} - \sigma \right) m} p^{-\sigma k_d}$$
converges.  The subseries consisting of $m=0, k_d \geq 2$ converges if $\sigma > \frac{1}{2}$.  If $k_d = 0, m \geq 2$, the series converges for $\sigma > \frac{d-1}{2} - \frac{1}{6}$.  If $m, k_d \geq 1$ then the series converges if $\sigma > \frac{d}{4} - \frac{1}{12}$. The theorem is now immediate. 
\end{proof}

We state the following corollary of the proof for future reference. 
\begin{cor}\label{53}
Let $d \geq 6$. There is a polynomial $D$ such that for all primes $p$ and all natural numbers $l$ we have 
$$
a_{\Z^d}^{1, <}(p^l) \leq D(l) p^{(\frac{d}{2} - \frac{5}{3})l}.
$$
Consequently, for each $\epsilon > 0$, we have 
$$
a_{\Z^d}^{1, <}(k) \ll_\epsilon k^{\frac{d}{2} - \frac{5}{3}+ \epsilon}.
$$
\end{cor}

\section{The proof of Theorems \ref{basicmain} and \ref{majorthm}}\label{general}
In this section we present a proof of Theorems \ref{conj:3} and \ref{n>5bound} which finish the proof of our main result, Theorem \ref{majorthm}. 
Let $K/\Q$ be an arbitrary extension of degree $n$ which we assume to be $K = \Q(\alpha)$ for $\alpha$ a root of an irreducible polynomial $f(x) \in \Z[x]$. We want to find a finite set $S$ of primes and $\sigma_0(n) \in \R$ such that the double series 
$$
\sum_{p \not\in S}\sum_{k \geq 2} \frac{a^{1, <}(p^k)}{p^{k \sigma}}
$$
converges for $\sigma > \sigma_0(n)$. We show that $\sigma_0(5) = 19/20$ works, and for $n>5$, $\sigma_0(n) = n/2 - 7/6$ works.  

\
We choose an integral basis for $K/\Q$ which we will fix throughout; in particular, this basis provides an integral basis for $K \otimes \Q_p$ over $\Q_p$. 
By equation \eqref{coneintegral1} we have 

\begin{equation}
\zeta_{\cO_K \otimes \Z_p, p}^{1,<}(s) = (1-p^{-1})^{-n} \int_{\M^1_p(K)} |x_{11}|^{s-n} |x_{22}|^{s-n+1} \cdots |x_{nn}|^{s-1}\, dM. 
\end{equation}
where we have written $\M_p^1(K)$ instead of the relevant $\M_p^1(\beta)$. 
\begin{defn}\label{muk1}
 If $\underline{k}=(k_1, \dots, k_n)$ is a $n$-tuple of
 non-negative integers, we set
 $$
\M_p^1(K; \underline{k}) = \left\{ M=\begin{pmatrix}
p^{k_1} & 0  & \ldots  & 0\\
x_{21} & p^{k_2} & 0 & \vdots\\
\vdots & \vdots & \ddots & 0\\
x_{n1} & \ldots  & x_{n \, n-1} & p^{k_n}
\end{pmatrix} \in \M_p^1(K)\right\}.
 $$
We define $\mu_p^1(K, \underline{k})$ to be the
$\frac{n(n-1)}{2}$-dimensional volume of $\M_p^1(K; \underline{k})$.
\end{defn}
The basic observation is that $\M_p^1(K, \underline{k})$ is given by a cone condition. The set $\M_p(K; \underline{k})$ is given by cone conditions. To define the set $\M_p^1(K, \underline{k})$, we have to add the condition that the sublattice generated by the rows contains the identity element. 
Let $ e \in \Z_p^n$ be the image of the identity element of $\cO_K \otimes \Z_p$ under the identification of the latter with $\Z_p^n$. 
Write 
$$
M=\begin{pmatrix}
p^{k_1} & 0  & \ldots  & 0\\
x_{21} & p^{k_2} & 0 & \vdots\\
\vdots & \vdots & \ddots & 0\\
x_{n1} & \ldots  & x_{n \, n-1} & p^{k_n}
\end{pmatrix}
$$
and let the rows of the matrix $M$ be $v_1, \dots, v_n$. Then $M \in \M_p^1(K; \underline{k})$ if there are $\alpha_1, \dots, \alpha_n \in \Z_p^n$ such that 
$ \sum_{i} \alpha_i v_i = e. $
This is equivalent to saying 
$$
(\alpha_1, \cdots, \alpha_n) M = e, 
$$
or what is the same 
$$
e.M^{-1} \in \Z_p^n. 
$$
Since $M$ is a lower triangular matrix, this last statement is equivalent to a collection of $p$-adic inequalities of the form considered in \S \ref{alphanarrow}. 

\
 
Let $S$ be a large finite set of primes containing all primes lying above $2$ and all ramified primes; after enlarging $S$ if necessary we may assume that any $p \not\in S$ is good in the sense of \S \ref{good-reduction}. Let $\mathcal{P}$ be the set of primes $p \not\in S$ which are split in the number field $K$. Clearly, $\mathcal{P}$ is an infinite set of primes. 

\ 
 
It is easy to see that
\begin{equation}\label{z-mu-1}
\zeta_{\cO_K \otimes \Z_p, p}^{1,<}(s)= \sum_{\underline{k}=(k_1, \dots, k_n) \atop k_i
\geq 0, \forall i} p^{\sum_{i=1}^n (n-i)k_i} p^{-s \sum_{i=1}^n
k_i} \mu_p^1(K, \underline{k}).
\end{equation}

Let $p\in \mathcal{P}$. For each $m$
$$
a^{1, <}_{\cO_K \otimes \Z_p}(p^m) = \sum_{\underline{k}=(k_1, \dots, k_n) \atop k_i\geq 0, \forall i, \sum_i k_i = m
} p^{\sum_{i=1}^n (n-i)k_i}  \mu_p^1(K, \underline{k}). 
$$

First we consider $n=5$. We start with the observation that by equation \eqref{narrow-special} for $m \geq 0$ 
$$
a^{1, <}_{\Z_p^5}(p^m) = a^<_{\Z_p^4}(p^m) \leq A(m) p^{-1 + 19m/20} 
$$
for a polynomial $A(m)$.  On the other hand, since 
$$
a^{1, <}_{\Z_p^5}(p^m) = \sum_{k+l+r+t+u =m} p^{4k+3l+2r+t} \mu_p^1(k; l; r; t; u) 
$$
we have 
$$
p^{4k+3l+2r+t} \mu_p^1(k; l; r; t; u)  \leq A(k, l, r, t, u) p^{-1 + 19(k+l+r+t+u)/20}
$$
whenever $k+l+r+t+u \geq 2$, for some polynomial $A(k, l, r, t, u)$.  Thus, 
$$
\mu_p^1(k, l, r, t, u) \leq A(k, l, r, t, u) p^{-1}p^{-(3+1/20)k - (2+1/20)l - (1+1/20)r - t/20 +19u/20}. 
$$
In the terminology of \S \ref{alphanarrow} this means that $\M_p^1(K)$ is $(1, \underline{\alpha}, A)$-narrow with 
$$
\underline\alpha = (3 + 1/20, 2+1/20, 1+1/20, 1/20, - 19/20) 
$$
and some polynomial $A$. Now Theorem \ref{thm:alphanarrow} implies that there is a finite set $T$ of primes such that  for $p \not\in T$ the set $\M_p^1(K)$ is $(1, \alpha, A)$-narrow.  Reversing the process, we get 
\begin{equation}\label{main-bound} 
a_{\cO_K}^{1, <}(p^m) \leq B(m) p^{-1 + 19m/20}
\end{equation}
for some polynomial $B(m)$. Clearly this implies that 
$$
\sum_{p \not\in T}\sum_{m \geq 2} \frac{a_{\cO_K}^{1, <}(p^m)}{p^{m \sigma}}
$$
converges for $ \sigma > 19/20$.  This shows that $\sigma_0(5) = 19/20$ works. The proof of the statement that $\sigma_0(n) = n/2 - 7/6$ works for $n \geq 6$ follows the same reasoning, except that we use Corollary \ref{53}. This finishes the proof of the theorem. 

\

The following corollary is immediate from equation \eqref{main-bound}. This is an improvement of Theorem 8.1 of \cite{Br}. 

\begin{cor}\label{81}
For any quintic field $K$ and any prime number $p$ we have 
$$
\sum_{m \geq 1}  \frac{a_{\cO_K}^{1, <}(p^m)}{p^{2 m}}  = O\left(\frac{1}{p^{2 + \frac{1}{20}}}\right). 
$$
\end{cor}

As in the introduction we set 
$$
a^{1, <}(n, m) = \max_{K/\Q \text{ extension of degree }n} a_{\cO_K}^{1, <}(m). 
$$
We have the following corollary: 

\begin{cor}\label{improvement}
We have
$$
\limsup_{m\to \infty} \frac{\log a^{1, <}(5, m)}{\log m} \leq \frac{19}{20}.   
$$
For $n \geq 6$, we have 
$$
\limsup_{m\to \infty} \frac{\log a^{1, <}(n, m)}{\log m} \leq \frac{n}{2} - \frac{5}{3}. 
$$
In particular, 
$$
\limsup_{n \to \infty} \frac{1}{n} \limsup_{m\to \infty} \frac{\log a^{1, <}(n, m)}{\log m} \leq \frac{1}{2}. 
$$

\end{cor}


\begin{thebibliography}{99}

\bibitem[B1]{B1} M. Bhargava, {\em Higher composition laws. IV. The parametrization of quintic rings}, Ann. of Math. (2) 167 (2008), no. 1, 53-94.

\bibitem[B2]{B} M. Bhargava, {\em The density of discriminants
of quintic rings and fields}, Ann. of Math. (2) 172 (2010), no. 3, 1559--1591.

\bibitem[Br]{Br} J. F. Brakenhoff, {\em Counting problems for number rings}, Thesis (Ph.D.)-Leiden University. 2009. 116 pp. 
Available at \url{http://www.math.leidenuniv.nl/scripties/proefschrift-brakenhoff.pdf}

\bibitem[Bu]{Burnside} W. Burnside, {\em Theory of Groups of Finite Order}, Cambridge Univ.
Press, 2nd edition, 1911.

\bibitem[CT]{C-T} A. Chambert-Loir and Yu. Tschinkel, {\em Fonctions z\^{e}ta des hauteurs des espaces fibr\'{e}s},
Rational points on algebraic varieties, Progr. Math., vol. 199, Birkh\"{a}user, Basel, 2001, 71–-115.

\bibitem[DW]{DW} B. Datskovsky and D. J. Wright. {\em The adelic zeta function associated with the space of binary cubic forms. II: Local theory}. J. Reine Angew. Math., 367:27–75, 1986.

\bibitem[Dn1]{Denef} J. Denef, {\em The rationality of the Poincar\'e
series associated to the $p$-adic points on a variety}. Invent. Math. 77 (1984), 1-23.

\bibitem[Dn2] {Denef-degree} Denef, J., {\it On the degree of Igusa's local zeta function}. Amer. J. Math. 109 (1987), no. 6, 991--1008. 

\bibitem[DM]{DM}  J. D. Dixon and B. Mortimer, {\em Permutation groups}, Graduate Texts in Mathematics, 163. Springer-Verlag, New York, 1996. xii+346 pp.

\bibitem[dSG1]{duSG} M. du Sautoy and F. Grunewald, {\em Analytic properties of
zeta functions and subgroup growth}. Ann. of Math. (2) 152 (2000),
no. 3, 793-833.

\bibitem[dSG2]{duSG2} M. du Sautoy and  F.  Grunewald, {\em Zeta functions of groups and rings}. International Congress of Mathematicians. Vol. II, 131-149, Eur. Math. Soc., ZŸ\"urich, 2006.

\bibitem[dST]{duST} M. du Sautoy and G. Taylor, {\em The zeta function of
$\mathfrak{sl}_2$ and resolution of singularities}. Math. Proc.
Camb. Phil. Soc. 132 (2002), no. 1, 57-73.

\bibitem[dSW]{duSW} M. du Sautoy and L. Woodward, {\em Zeta functions of groups and rings}, Lecture Notes in Mathematics 1925, Springer-Verlag, 2008. 

\bibitem[GSS]{GSS} F.J. Grunewald, D. Segal, and G.C. Smith, {\em Subgroups of
finite index in nilpotent groups}. Invent. Math. 93 (1988),
185-223.

\bibitem[Hu]{Hu} B. Huppert, {\em Endliche Gruppen I}. Springer-Verlag, Berlin-New York
1967, xii+793 pp.

\bibitem[I1]{Igusa} J.-I. Igusa, {\em Some observations on higher degree
characters}. Amer. J. Math. 99 (1977), no. 2, 393-417.

\bibitem[I2]{Igusa2} J.-I. Igusa,  {\em An introduction to the theory of local zeta functions}. 
AMS/IP Stud. Adv. Math. 14, Amer. Math. Soc., Providence, RI, International Press, Cambrigde, MA, 2002.

\bibitem[K]{K} N. Kaplan, {\em $p$-adic integration and subrings of $\Z^n$}, Princeton Senior Thesis, 2007.

\bibitem[L]{Liu} R. Liu,  {\em 
Counting subrings of $\Z^n$ of index k}, 
J. Combin. Theory Ser. A 114 (2007), no. 2, 278-299. 

\bibitem[LS]{LS}  A. Lubotzky and D. Segal, {\em Subgroup growth}. Progress in Mathematics, 212. BirkhŠ\"auser Verlag, Basel, 2003. xxii+453 pp.

\bibitem[Mc]{Maci} A. Macintyre, {\em On definable subsets of $p$-adic fields}.
J. Symbolic Logic 41 (1976), no. 3, 605-610.

\bibitem[Mu]{Muller} P. Mueller, {\em Permutation groups of prime degree, a quick proof of Burnsides theorem}, Archiv der Mathematik 85 (2005), no. 1, 15--17.

\bibitem[Mu-Mu]{Mu-Mu}  M. R. Murty and V. K. Murty, V. {\em Non-vanishing of L-functions and applications}, Progress in Mathematics, 157. Birkh\"aŠuser Verlag, Basel, 1997. xii+196 pp.

\bibitem[N]{N} J. Nakagawa, {\em Orders of a quartic field}, Mem. Amer. Math. Soc. 122 (583) (1996), viii+75.

\bibitem[Ne]{Ne} J. Neukirch, {\em Class field theory}, Grundlehren der Mathematischen Wissenschaften, 280. Springer-Verlag, Berlin, 1986. viii+140 pp.

\bibitem[STT]{S-T-T} J. Shalika, R. Takloo-Bighash, and Yu. Tschinkel, {\em Rational points on
compactifications of semi-simple groups}.  J. Amer. Math. Soc.  20
 (2007),  no. 4, 1135--1186.

\bibitem[V1]{Voll} C. Voll, {\em Functional equations for zeta functions of groups and rings}, Ann. of Math. (2) 172 (2010), no. 2, 1181--1218. 

\bibitem[V2]{Voll2} C. Voll,  {\em A newcomer's guide to zeta functions of groups and rings}. Lectures on profinite topics in group theory, 99-144, London Math. Soc. Stud. Texts, 77, Cambridge Univ. Press, Cambridge, 2011.

\end{thebibliography}
\end{document}